\documentclass[11pt,a4paper]{amsart}

\usepackage{graphicx}
\usepackage{enumerate}
\usepackage{bbm}
\usepackage{amsmath, amssymb, amsthm, amscd}
\usepackage{mathtools}
\usepackage[dvipsnames]{xcolor}
\usepackage{caption}
\usepackage{subcaption}
\usepackage{multirow}
\usepackage{extarrows}
\usepackage{stmaryrd}
\usepackage{tikz-cd}
\usepackage[pdfencoding=auto, psdextra, hidelinks]{hyperref}

\theoremstyle{plain}
\newtheorem{theorem}{Theorem}[section]
\newtheorem{lemma}[theorem]{Lemma}

\newtheorem{proposition}[theorem]{Proposition}

\theoremstyle{definition}
\newtheorem{definition}[theorem]{Definition}

\theoremstyle{remark}
\newtheorem{remark}[theorem]{Remark}

\numberwithin{equation}{section}

\def\softd{{\leavevmode\setbox1=\hbox{d}%
              \hbox to 1.05\wd1{d\kern-0.4ex{\char039}\hss}}}

\newcommand{\ldblbrace}{\{\mskip-5mu\{}
\newcommand{\rdblbrace}{\}\mskip-5mu\}}

\newcommand{\R}{\mathbb{R}}
\newcommand{\N}{\mathbb{N}}

\newcommand{\Cinf}{C_c^\infty}
\newcommand{\tdom}{(0,T)}
\newcommand{\odom}{\tdom\times\Omega}

\newcommand{\F}{\mathcal{F}} 
\newcommand{\D}{\mathcal{D}} 
\newcommand{\M}{\mathcal{M}} 
\newcommand{\K}{\mathcal{K}} 
\newcommand{\E}{\mathcal{E}}
\newcommand{\G}{\mathcal{G}}
\newcommand{\I}{\mathcal{I}}
\newcommand{\Pro}{\mathcal{P}} 
\newcommand{\veps}{\varepsilon}
\newcommand{\X}{\mathcal{X}}

\newcommand{\half}{\frac{1}{2}}
\newcommand{\dx}{\mathrm{d}x}
\newcommand{\dy}{\mathrm{d}y}
\newcommand{\dt}{\mathrm{d}t}

\newcommand{\V}{\mathcal{V}_{t,x,y}}

\renewcommand{\th}{\widetilde{h}}
\newcommand{\tmx}{\widetilde{m_x}}
\newcommand{\tmy}{\widetilde{m_y}}
\renewcommand{\k}{{(k)}}

\newcommand{\T}{\mathcal{T}}
\newcommand{\Lt}{\mathcal{L}_{\mathcal{T}}}

\newcommand{\delt}{\Delta t}
\newcommand{\delx}{\Delta x}
\newcommand{\dely}{\Delta y}
\newcommand{\derx}{\eth_x}
\newcommand{\dery}{\eth_y}
\renewcommand{\ij}{{i,j}}
\newcommand{\ipj}{{i+\half,j}}
\newcommand{\imj}{{i-\half,j}}
\newcommand{\ijp}{{i,j+\half}}
\newcommand{\ijm}{{i,j-\half}}

\newcommand{\norm}[1]{\left\lVert#1\right\rVert}
\newcommand{\abs}[1]{\lvert#1\rvert}
\newcommand{\weakstar}{\overset{\ast}{\rightharpoonup}}

\newcommand{\Dt}{\partial_t}
\newcommand{\Dx}{\partial_x}
\newcommand{\Dy}{\partial_y}

\title[A Convergent Structure-Preserving Scheme for RSW System]{A
  Convergent Structure-Preserving Scheme for Dissipative Solutions of
  the Rotating Shallow Water System} 

\author[Arun]{K.R.\ Arun}
\author[Krishnamurthy]{A.\ Krishnamurthy}
\thanks{The authors are grateful to Maria Luk\'{a}\v{c}ov\'{a} for her
  encouragement, support, and many valuable discussions on the topic.} 
\address{School of Mathematics, Indian Institute of Science Education
  and Research Thiruvananthapuram, Thiruvananthapuram 695551, India} 
\email{arun@iisertvm.ac.in, amoghk0720@iisertvm.ac.in}

\date{\today}



\begin{document}

\begin{abstract}
  We design and analyse a semi-implicit finite volume scheme for the
  two-dimensional rotating shallow water (RSW) equations that is
  energy stable, well-balanced (capable of preserving discrete
  geostrophic steady states), consistent, and covergent. The key idea
  is the introduction of carefully chosen stabilisation terms into the
  convective fluxes of the mass and momentum equations, as well as the
  source terms. Under a CFL-type condition, together with an auxiliary
  time-step restriction arising from the Coriolis forces, we establish
  the energy stability of the scheme. The stabilisation terms are
  constructed to vanish at steady states, thereby ensuring the
  well-balancing property under an appropriate advective CFL
  condition. We derive a sufficient time-step restriction that
  guarantees stability, well-balancing, existence of discrete
  solutions, and positivity simultaneously. Furthermore, under mild
  boundedness assumptions, we obtain a priori estimates showing that
  the stabilisation terms converge to zero as the mesh is refined,
  which establishes the consistency of the scheme. This in turn
  enables us to prove that numerical solutions generate a Young
  measure, identifiable as a dissipative measure-valued solution of
  the RSW system, thereby yielding convergence of the scheme. Finally,
  we confirm the theoretical results through extensive numerical
  experiments.  
\end{abstract}

\maketitle

\section{Introduction}
\label{sec:intro}
The rotating shallow water (RSW) system is a fundamental model in
geophysical fluid dynamics, capturing essential features of
large-scale oceanic and atmospheric flows. Its solutions exhibit a
wide range of nonlinear wave phenomena, including gravity waves,
advection waves, and balanced modes. The presence of Coriolis forcing
introduces additional complexity by admitting nontrivial equilibrium states
such as geostrophic balance, in which pressure gradients are exactly
balanced by rotational effects. Such equilibria play a central role in
geophysics, since many physically relevant flow structures, such as
jets and eddies, can be interpreted as perturbations of these equilibrium
states; see \cite{Ped12}. Consequently, the RSW system serves both as a
simplified prototype for rigorous mathematical analysis and as a
benchmark model for the design of numerical methods.   

In this work, we consider the two-dimensional (2D) RSW system with
source terms accounting for bottom topography and Coriolis
forces. The governing balance laws admit nonlinear shocks,
rarefaction waves, and contact waves induced by discontinuous
topography. To accommodate such phenomena, weak (distributional)
solutions are introduced. However, weak solutions are generally
non-unique, and physically relevant states are typically selected via
an admissibility criterion in the form of an entropy condition. While
this framework is well understood in the scalar case, its limitations
in multiple space dimensions have become apparent. Chiodaroli et al.\
\cite{CLK15} and De Lellis and Sz{\'e}kelyhidi \cite{LS10} established
non-uniqueness of entropy weak solutions for the barotropic Euler
equations, and Feireisl et al.\ \cite{FKK+20} extended this to the
full Euler system. The resulting ill-posedness is linked to the lack
of compactness: even sequences of bounded entropy weak solutions may
exhibit oscillations and concentrations.  

The notion of measure-valued (MV) solutions, first formulated by
DiPerna \cite{DiP85} in the study of general conservation laws, has
gained renewed attention in recent years. MV solutions are particularly
useful since they arise as limits of oscillatory approximate sequences
and can be described through Young measures, which capture the
underlying oscillations; see \cite{FLM+21a}. In this work, we
concentrate on the so-called dissipative measure-valued (DMV)
solutions, a refinement of MV solutions introduced by Feireisl et al.\
for the Euler equations in \cite{BF18a}. DMV solutions provide a natural
weak-topology closure of the solution set, where the dissipative weak 
solutions correspond to the barycenters of the Young measures. This
framework is of special interest because it admits global existence in
time, and it preserves the weak–strong uniqueness property: whenever a 
strong solution exists for the same initial data, the corresponding
DMV solution agrees with it. A detailed account of the subject can be
found in \cite{BF18a, FLM+21a}. 

For hyperbolic balance laws with source terms, it is highly desirable
to construct numerical schemes that preserve discrete analogues of
steady states, commonly referred to as well-balanced schemes; see
\cite{ABB+04, CLP08, FMT11, KP07, LNK07, NPP+06, Par06}. In addition, to accurately
capture complex flow features such as shocks and rarefactions,
properties like positivity, energy stability and consistency are
indispensable. Collectively, these features fall under the broader
class of structure-preserving schemes; see, e.g.\ \cite{LeF14}. The
numerical approximation of conservation laws poses significant
challenges due to the presence of discontinuities in weak solutions,
where classical finite difference methods often fail. In contrast,
finite volume methods are specifically tailored to handle such
discontinuities with high resolution, providing a robust and efficient
alternative. Their development over the past decades is now well
established and extensively documented; see, e.g.\ \cite{GR21}. The
convergence of these schemes to weak solutions of the compressible
Euler equations has been analysed in detail, e.g.\ \cite{HLN18},
though typically under restrictive assumptions. 

Within the DMV framework, however, it can be shown that,
subject to suitable stability and consistency conditions, finite
volume approximations converge in a weak sense to DMV solutions under
natural boundedness requirements. A recent survey of progress in this
direction is given in \cite{FLM+21a}. A key limitation, however, 
is that one obtains only weak convergence of
numerical solutions. To overcome this drawback, the so-called 
$\mathcal{K}$-convergence techniques are employed to recover strong
convergence results. Since numerical approximations converge merely in
the weak sense, leading to a DMV solution as the limit, identifying the
precise limit solution remains nontrivial. Within the
$\mathcal{K}$-convergence framework, the arithmetic averages of
numerical solutions converge strongly to the barycenters
(expectations) of the Young measure associated with the DMV limit; see
\cite{FLM+21a, FLM+21b} for a detailed discussion.  

\subsection{Aims and Scope of the Paper}
\label{subsec:aims-scope}

This work is concerned with the design and analysis of a novel
structure-preserving finite volume scheme for the 2D RSW
equations. Unlike standard approaches, the proposed scheme
simultaneously addresses energy stability, well-balancing with respect
to geostrophic steady states, and convergence toward DMV
solutions. The key mechanism is the introduction of carefully
constructed stabilisation terms into the convective fluxes and source
terms, which guarantee stability and positivity while vanishing at
equilibrium, thereby preserving the delicate steady states of the RSW
system.   

We establish energy stability under a CFL-type condition supplemented
by an auxiliary time-step restriction induced by the Coriolis
force. The vanishing of the stabilisation terms at steady states
allows us to prove that geostrophic equilibria are preserved under an
appropriate advective CFL restriction. Moreover, by deriving strong a
priori estimates, we show that the stabilisation terms disappear in
the limit of mesh refinement, which ensures consistency and enables us
to prove convergence of the numerical solutions to DMV solutions of
the RSW system. To our knowledge, this is the first rigorous
demonstration of DMV convergence for a structure-preserving scheme
applied to the RSW equations. The theoretical results are further
substantiated through extensive numerical experiments, which confirm
the scheme’s stability, well-balancing, and convergence properties. 

\subsection{Rotating Shallow Water Equations}
\label{subsec-RSW}
We consider an initial boundary value problem for the 2D RSW equations
with a non-flat bottom on $\odom$, where $T>0$ and $\Omega$ is an
open, bounded subset of $\R^2$. The system of equations reads   
\begin{subequations}
  \label{eqn:cons-rsw}
  \begin{gather}
    \Dt h + \Dx(hu) + \Dy(hv) = 0, \label{eqn:mss-bal-con} \\
    \Dt(hu) + \Dx(hu^2) + \Dy(huv) + \Dx\Bigl(\half gh^2\Bigr) =
    -gh\Dx b + \omega hv, \label{eqn:x-mom-bal-con}\\ 
    \Dt(hv) + \Dx(huv) + \Dy(hv^2) + \Dy\Bigl(\half gh^2\Bigr) =
    -gh\Dy b - \omega hu, \label{eqn:y-mom-bal-con}\\ 
    h(0,x,y) = h_0(x,y)>0,\quad u(0,x,y) = u_0(x,y),\quad v(0,x,y) =
    v_0(x,y). \label{eqn:ini-data} 
  \end{gather}
\end{subequations}
Here, $h = h(t,x,y)$ denote the water depth measured above the
stationary bottom topography $b = b(x,y) \in
W^{2,\infty}(\Omega)$. The horizontal velocity field is given by
$(u,v) = (u(t,x,y), v(t,x,y))$, where $u$ and $v$ represent the
velocity components in the $x$- and $y$-directions, respectively. The
constants $g$ and $\omega$ denote the gravitational acceleration and
the Coriolis parameter. The initial data are assumed to satisfy
$h_0,\, u_0,\, v_0 \in L^\infty(\Omega)$. For simplicity, we impose
periodic boundary conditions, so that the spatial domain $\Omega$ is
identified with the 2D torus $\mathbb{T}^2$. 

Since one of our main objectives is to design a well-balanced scheme,
we consider the non-conservative form of \eqref{eqn:cons-rsw}, written
in terms of the scalar potential $\phi = g(h+b)$. In this formulation,
the system takes the form 
\begin{subequations}
  \label{eqn:non-cons-rsw}
  \begin{gather}
    \Dt h + \Dx(hu) + \Dy(hv) = 0, \label{eqn:mss-bal-nc} \\
    \Dt(hu) + \Dx(hu^2) + \Dy(huv) + h\Dx\phi = \omega
    hv, \label{eqn:x-mom-bal-nc}\\ 
    \Dt(hv) + \Dx(huv) + \Dy(hv^2) + h\Dy\phi = - \omega
    hu. \label{eqn:y-mom-bal-nc} 
  \end{gather}
\end{subequations}

\subsubsection{Steady State Solutions}

Due to the presence of source terms, the RSW system
\eqref{eqn:non-cons-rsw} admits multiple equilibrium states. Among
these, the states for which the material derivatives vanish are
referred to as the nonlinear geostrophic equilibria. In particular, a
solution $(h,u,v)$ of the RSW system is said to be at geostrophic
equilibrium if the following conditions are satisfied: 
\begin{subequations}
  \label{eqn:geo-bal}
  \begin{gather}
    \Dx u + \Dy v = 0, \\
    \Dx\phi - \omega v = 0,\quad \Dy\phi + \omega u = 0.
  \end{gather}
\end{subequations}
For the present work, we place a particular emphasis on the
geostrophic steady states, i.e.\ solutions at geostrophic equilibrium
that are also stationary in time, commonly referred to as jets in the
rotating frame. These steady states are characterised by either of the
following conditions: 
\begin{gather}
  u = 0,\quad \Dy v = 0,\quad \Dy h = 0,\quad \Dy b = 0,\quad
  \Dx\phi - \omega v = 0, \label{eqn:jet-x} \\
  v = 0,\quad \Dx u = 0,\quad \Dx h = 0,\quad \Dx b = 0,\quad \Dy\phi
  + \omega u = 0. \label{eqn:jet-y} 
\end{gather}
It is straightforward to verify that any solution satisfying either
\eqref{eqn:jet-x} or \eqref{eqn:jet-y} also satisfies the geostrophic
balance condition \eqref{eqn:geo-bal}. However, the converse does not
necessarily hold. Therefore, one of the primary objectives of this
work is to construct a well-balanced finite volume scheme that
guarantees the exact preservation of the steady states
\eqref{eqn:jet-x} and \eqref{eqn:jet-y}.

\subsubsection{Energy stability estimates}

The potential and kinetic energies of the system
\eqref{eqn:non-cons-rsw} are given by $\I = \half gh^2 + gbh$ and $\K
= \half h(u^2+v^2)$, and the total energy $E$ is the sum of these
energies, i.e.\ $E = \I+\K$. The total energy $E$ serves as the
entropy associated to the RSW system \eqref{eqn:cons-rsw} and smooth
solutions of \eqref{eqn:cons-rsw} satisfy  
\begin{equation}
  \label{eqn:eng-eq}
  \Dt E + \Dx\biggl(\Bigl(\phi + \frac{u^2 + v^2}{2}\Bigr)hu\biggr)
  + \Dy\biggl(\Bigl(\phi + \frac{u^2 + v^2}{2}\Bigr)hv\biggr) =
  0. 
\end{equation}
In the case of weak solutions, the total energy is dissipated rather
than conserved. In the discrete setting, it is thus essential to
incorporate a mechanism that ensures the dissipation of the discrete
total energy so as to maintain the stability of the numerical
solutions. To achieve this, we construct a discrete analogue of
\eqref{eqn:eng-eq} by means of a stabilisation technique, the details
of which are presented in the subsequent sections. 

\subsection{Dissipative Measure-Valued Solutions}

The framework of dissipative measure-valued (DMV) solutions is
relatively recent, having been introduced and developed by Feireisl,
Luk\'{a}\v{c}ov\'{a}, and their collaborators in a series of works; see
\cite{BF18a, FLM+21a, FLM+21b} and the references therein for a
comprehensive exposition. A key advantage of DMV solutions lies in
their realisation as limits of suitably  designed numerical schemes;
see, e.g.\ \cite{FLM+21a}.  

To define the notion of a DMV solution, it is convenient to work with
the conservative variables $(h, m_x, m_y)$, where $m_x = hu$ and $m_y
= hv$ denote the $x$- and $y$-components of the momentum,
respectively. Recasting the system \eqref{eqn:cons-rsw} in terms of
these conservative variables yields 
\begin{subequations}
  \label{eqn:cv-rsw}
  \begin{gather}
    \Dt h + \Dx m_x + \Dy m_y = 0, \label{eqn:mss-bal-cv} \\
    \Dt m_x + \Dx\Bigl(\frac{m_x^2}{h}\Bigr) + \Dy\Bigl(\frac{m_x
      m_y}{h}\Bigr) + \Dx\Bigl(\half gh^2\Bigr) = -gh\Dx b + \omega
    m_y, \label{eqn:x-mom-bal-cv}\\ 
    \Dt m_y + \Dx\Bigl(\frac{m_xm_y}{h}\Bigr) +
    \Dy\Bigl(\frac{m_y^2}{h}\Bigr) + \Dy\Bigl(\half gh^2\Bigr) =
    -gh\Dy b - \omega m_x, \label{eqn:y-mom-bal-cv}\\ 
    h(0,x,y) = h_0(x,y),\quad m_x(0,x,y) = m_{0,x}(x,y),\quad
    m_y(0,x,y) = m_{0,y}(x,y). \label{eqn:ini-data-cv} 
    \end{gather}
\end{subequations}

Let $\F$ denote the phase space defined as 
\begin{equation}
  \label{eqn:phase-space}
  \F = \lbrace(\th,\tmx,\tmy)\in\R^3\vert\,\th>0\rbrace
\end{equation}
and let $\Pro(\F)$ denote the set of all probability measures on
$\F$. Inspired by \cite{FLM+21a}, and following the notation therein,
we define a DMV solution of the RSW system as follows. 

\begin{definition}[Dissipative measure-valued solution of the RSW system]
  \label{def:dmv-sol}
  A Young-measure $\mathcal{V} = \lbrace\V\rbrace_{(t,x,y)\in\odom}\in
  L^\infty_{weak-*}(\odom;\Pro(\F))$ along with the concentration
  defect measures  
  \[
    \mathfrak{E}_{cd}\in L^\infty(0,T;\M^+(\Omega)) \text{ and }
    \mathfrak{R}_{cd} = \begin{bmatrix} \mathfrak{R}^{x}_{cd} &
      \mathfrak{R}^{x,y}_{cd} \\ \mathfrak{R}^{x,y}_{cd} &
      \mathfrak{R}^y_{cd}\end{bmatrix}\in
    L^\infty(0,T;\M(\Omega;\R^{2\times 2}_{sym})) 
  \]
  is a dissipative measure-valued solution of the RSW system
  \eqref{eqn:mss-bal-cv}-\eqref{eqn:ini-data-cv} if  
  \begin{itemize}
  \item the map $(t,x,y)\mapsto\langle\V;\th\rangle\in
    C_{weak}(\lbrack0,T\rbrack;L^2(\Omega))$ and the integral identity 
    \begin{equation}
      \label{eqn:dmv-mass}
      \begin{split}
        \biggl\lbrack\iint\limits_\Omega
        h_0\psi(0,\cdot)\,\dx\,\dy\biggr\rbrack_{t=0}^{t=\tau} =
        \int\limits_{0}^{\tau}\!\!\!\iint\limits_{\Omega}\lbrack\langle\V;\th\rangle\Dt\psi
        &+ \langle\V;\tmx\rangle\Dx\psi \\
        &+\langle\V;\tmy\rangle\Dy\psi\rbrack\,\dx\,\dy\,\dt 
      \end{split}
    \end{equation}
    is satisfied for any $\tau\in\lbrack0,T\rbrack$ and
    $\psi\in\Cinf(\lbrack 0,T)\times\Omega)$; 

  \item the maps
    $(t,x,y)\mapsto\langle\V;\tmx\rangle,\,(t,x,y)\mapsto\langle\V;\tmy\rangle\in
    C_{weak}(\lbrack0,T\rbrack;L^{\frac{4}{3}}(\Omega))$ and the
    integral identities 
    \begin{gather}
      \biggl\lbrack\iint\limits_\Omega
      m_{0,x}\Psi(0,\cdot)\,\dx\,\dy\biggr\rbrack_{t=0}^{t=\tau} =
      \int\limits_{0}^{\tau}\!\!\!\iint\limits_{\Omega}\biggl\lbrack\langle\V;\tmx\rangle\Dt\Psi
      + \biggl\langle\V;\frac{\tmx^2}{\th}\biggr\rangle\Dx\Psi
       \nonumber \\  
      +\biggl\langle\V;\frac{\tmx\,\tmy}{\th}\biggr\rangle\Dy\Psi
      +\biggl\langle\V;\half g\th^2\biggr\rangle\Dx\Psi\biggr\rbrack\,\dx\,\dy\,\dt +
      \int\limits_{0}^{\tau}\!\!\!\iint\limits_{\Omega}\lbrack-\langle\V;g\th\rangle\Dx
      b\,\Psi \label{eqn:dmv-x-mom} \\ 
      + \langle\V;\omega\,\tmy\rangle\Psi\rbrack\,\dx\,\dy\,\dt + \int\limits_0^\tau\!\!\!\iint\limits_{\Omega}\Dx\Psi\,\mathrm{d}\mathfrak{R}^x_{cd}(t)\,\dt + \int\limits_0^\tau\!\!\!\iint\limits_{\Omega}\Dy\Psi\,\mathrm{d}\mathfrak{R}^{x,y}_{cd}(t)\,\dt \nonumber
    \end{gather}
    and
    \begin{gather}
      \biggl\lbrack\iint\limits_\Omega m_{0,y}\Phi(0,\cdot)\,\dx\,\dy\biggr\rbrack_{t=0}^{t=\tau} = \int\limits_{0}^{\tau}\!\!\!\iint\limits_{\Omega}\biggl\lbrack\langle\V;\tmy\rangle\Dt\Phi + \biggl\langle\V;\frac{\tmx\,\tmy}{\th}\biggr\rangle\Dx\Phi \nonumber \\ 
      + \biggl\langle\V;\frac{\tmy^2}{\th}\biggr\rangle\Dy\Phi +\biggl\langle\V;\half g\th^2\biggr\rangle\Dy\Phi\biggr\rbrack\,\dx\,\dy\,\dt + \int\limits_{0}^{\tau}\!\!\!\iint\limits_{\Omega}\lbrack-\langle\V;g\th\rangle\Dy b\,\Phi \label{eqn:dmv-y-mom} \\
      - \langle\V;\omega\,\tmx\rangle\Phi\rbrack\,\dx\,\dy\,\dt + \int\limits_0^\tau\!\!\!\iint\limits_{\Omega}\Dx\Phi\,\mathrm{d}\mathfrak{R}^{x,y}_{cd}(t)\,\dt + \int\limits_0^\tau\!\!\!\iint\limits_{\Omega}\Dy\Phi\,\mathrm{d}\mathfrak{R}^{y}_{cd}(t)\,\dt \nonumber
    \end{gather}
    are satisfied for any $\tau\in\lbrack0,T\rbrack$ and $\Psi,\Phi\in\Cinf(\lbrack0,T)\times\Omega)$;

    \item the energy inequality 
    \begin{equation}
      \label{eqn:eng-ineq-dmv}
      \begin{split}
        \iint\limits_\Omega\biggl\langle\V; \frac{\tmx^2}{2\th} + \frac{\tmy^2}{2\th} + \half g\th^2 + gb\th\biggr\rangle\,\dx\,\dy + &\iint\limits_\Omega\mathrm{d}\mathfrak{E}_{cd}(t) \\
        &\leq \iint\limits_\Omega\biggl(\frac{m_{0,x}^2}{2h_0} + \frac{m_{0,x}^2}{2h_0} + \half gh_0^2 + gbh_0\biggr)\,\dx\,\dy.
      \end{split}
    \end{equation}
    is satisfied for a.e. $t\in(0,T)$;

    \item there exist constants $0<\underline{d}\leq\overline{d}$ such that
    \begin{equation}
    \label{eqn:def-comp-dmv}
    \underline{d}\mathfrak{E}_{cd}\leq \text{tr}(\mathfrak{R}_{cd})\leq \overline{d}\mathfrak{E}_{cd}.
    \end{equation}
  \end{itemize}
\end{definition}
Thus, the third and final objective of this work is to establish the
existence of a DMV solution to the RSW system, as specified in
Definition \ref{def:dmv-sol}. This is accomplished through a detailed
consistency and convergence analysis of the numerical solutions
produced by our scheme. In particular, we demonstrate that the
discrete solutions generated by the scheme converge, in the sense of
Young measures, to a DMV solution of the RSW system.

\subsection{Stabilisation}
\label{subsec:stab}

As stated earlier, our goal is to design a finite volume scheme for
the RSW system that satisfies the following key properties: 

\begin{itemize}
\item \textbf{Energy Stability}: The numerical solutions obey a
  discrete analogue of \eqref{eqn:eng-eq} to ensure the 
  dissipation of the discrete total energy. 

\item \textbf{Well-Balancing}: The scheme exactly preserves discrete
  variants of the geostrophic steady states
  \eqref{eqn:jet-x}–\eqref{eqn:jet-y}.  

\item \textbf{Consistency and Convergence}: The discrete solutions are
  consistent with the system \eqref{eqn:cv-rsw} and, upon mesh
  refinement, converge (in the sense of Young measures) to a DMV
  solution of the RSW system.  
\end{itemize}

To achieve these properties, we employ the stabilisation technique,
wherein a shift is introduced in the convective fluxes of the mass and
momentum equations, as well as in the source terms; see also
\cite{AGK23, AA24, PV16} and the references therein. The resulting
modified RSW system takes the form 
\begin{subequations}
\label{eqn:stab-rsw}
    \begin{gather}
        \Dt h + \Dx(hu - q) + \Dy(hv-r) = 0, \label{eqn:mss-bal-stab} \\
        \Dt(hu) + \Dx((hu-q)u) + \Dy((hv-r)u) + h\Dx\phi = \omega (hv - r), \label{eqn:x-mom-bal-stab}\\
        \Dt(hv) + \Dx((hu-q)v) + \Dy((hv-r)v) + h\Dy\phi = - \omega (hu - q). \label{eqn:y-mom-bal-stab}
    \end{gather}
\end{subequations}
Here, $q$ and $r$ denote generic perturbations in the $x$ and $y$-components of the momentum respectively. We can obtain the following energy estimates after performing some formal calculations.

\begin{proposition}[A priori estimates of the modified system]
\label{prop:apri-est-mod}
  The solutions of the modified system \eqref{eqn:stab-rsw} satisfy the following total energy identity:
  \begin{equation}
  \label{eqn:mod-energ-idt}
    \begin{split}
      \Dt E + \Dx\biggl(\Bigl(\phi + \frac{u^2+v^2}{2}\Bigr)(hu-q)\biggr) + &\Dy\biggl(\Bigl(\phi + \frac{u^2+v^2}{2}\Bigr)(hv-r)\biggr) \\
      &= -q(\Dx\phi - \omega v) - r(\Dy\phi + \omega u).
    \end{split}
  \end{equation}  
\end{proposition}

Thus, we can observe that if we choose the perturbation terms as $q =
\eta(\Dx\phi - \omega v)$ and $r = \xi(\Dy\phi + \omega u)$, where
$\eta,\xi>0$, we obtain  
\begin{equation*}
    \begin{split}
        \Dt E + \Dx\biggl(\Bigl(\phi + \frac{u^2+v^2}{2}\Bigr)(hu-q)\biggr) + &\Dy\biggl(\Bigl(\phi + \frac{u^2+v^2}{2}\Bigr)(hv-r)\biggr) \\
        &= -\eta(\Dx\phi - \omega v)^2 - \xi(\Dy\phi + \omega u)^2\leq 0.
    \end{split}
\end{equation*}

Therefore, the introduction of these perturbation terms has a
stabilising effect, ensuring the dissipation of the total energy. Moreover, the stabilisation terms vanish under the steady state conditions \eqref{eqn:jet-x}-\eqref{eqn:jet-y}. This observation is
crucial, as it will be used to establish the well-balancing property
of the proposed scheme. 

\subsection{The Finite Volume Method}
\subsubsection{Mesh and differential operators}
\label{subsubsec:mesh}

We consider a tessellation $\T$ of $\Omega$, consisting of rectangles
called primal cells, $K_\ij = \lbrack x_{i-\half},x_{i+\half}\rbrack\times\lbrack
y_{j-\half},y_{j+\half}\rbrack\in\T$ such that $\overline{\Omega} =
\bigcup_{\ij}K_\ij$. The cell centre of $K_\ij$ is denoted by
$(x_i,y_j)$ where $x_i = (x_{i-\half} + x_{i+\half}) / 2$ and $y_j =
(y_{j-\half} + y_{j+\half}) / 2$. For the sake of convenience, we
suppose that the space steps $\delx$ and $\dely$ are constant, i.e.\
we consider a uniform mesh and $x_{i+\half} - x_{i-\half} =
x_{i+1}-x_i = \delx$ and $y_{j+\half} - y_{j-\half} = y_{j+1}-y_j =
\dely$ for each $\ij$. Therefore, it follows that the 2D
Lebesgue measure of $K_\ij$ is $\abs{K_\ij} = \delx\dely$. By
$\sigma_{\ipj} = \lbrace x_{i+\half}\rbrace \times \lbrack
y_{j-\half},y_{j+\half}\rbrack$, we denote the vertical edge common to
the cells $K_\ij$ and $K_{i+1,j}$ and analogously, we define the
horizontal edge $\sigma_\ijp$ as the common edge of the cells $K_\ij$
and $K_{i,j+1}$. We let $\E^{(1)}$ denote the collection of all
vertical edges $\sigma_\ipj$ and analogously define $\E^{(2)}$ to be
the collection of all horizontal edges $\sigma_\ijp$. In addition, we
also set $\E_{ext}$ to be the collection of all external edges lying
on the boundary $\partial\Omega$. We denote the size of the mesh by
$\delta_\T$ and it is defined by  
\[
    \delta_\T = \max_{\ij}\mathrm{diam}(K_\ij).
\]
and we suppose that $\delta_\T\in (0,\delta_0]$ for a sufficiently
small $\delta_0$. For any two numbers $a_1$ and $a_2$, we say that
$a_1\lesssim a_2$ if $a_1\leq c a_2$ for a constant $c$ which is
independent of the mesh parameters. 

\begin{figure}[htpb]
    \centering
        \begin{tikzpicture}[scale = 0.7]
            \fill [green!20!white] (2.5,0) rectangle (5,4);
            \fill [red!20!white] (5,0) rectangle (7.5,4);
            \draw[black, thick] (0,0) rectangle (10,4);
            \draw[blue, thick] (5,0) -- (5,4);
            \draw[green, thick] (2.5, 0) -- (5,0);
            \draw[green, thick] (2.5, 4) -- (5,4);
            \draw[red, thick] (5,0) -- (7.5,0);
            \draw[red, thick] (5,4) -- (7.5,4);
            \draw[densely dotted, green, thick] (2.5,0) -- (2.5,4);
            \draw[densely dotted, red, thick] (7.5,0) -- (7.5,4);
            \draw[<->, black] (2.5, 4.4) -- (7.5,4.4);
            \node at (5,4.8) {$D_{\ipj}$};
            \filldraw [black] (0,2) circle (2pt) node[anchor = north east]{{\small{$(x_{i-\half}, y_j)$}}};
            \filldraw [black] (2.5,2) circle (2pt) node[anchor = north east]{{\small{$(x_{i}, y_j)$}}};
            \filldraw [black] (5,2) circle (2pt) node[anchor = north east]{{\small{$(x_{i+\half}, y_j)$}}};
            \filldraw [black] (7.5,2) circle (2pt) node[anchor = north west]{{\small{$(x_{i+1}, y_j)$}}};
            \filldraw [black] (10,2) circle (2pt) node[anchor = north west]{{\small{$(x_{i+\frac{3}{2}}, y_j)$}}};
            \draw (1.25,2) node[anchor = south]{$K_\ij$};
            \draw (8.75,2) node[anchor = south]{$K_{i+1,j}$};
            \draw (5,0) node[anchor = north]{{\color{blue}$\sigma_\ipj$}};
        \end{tikzpicture}
    \caption{Primal and dual cells}
    \label{fig:space-disc}
\end{figure}
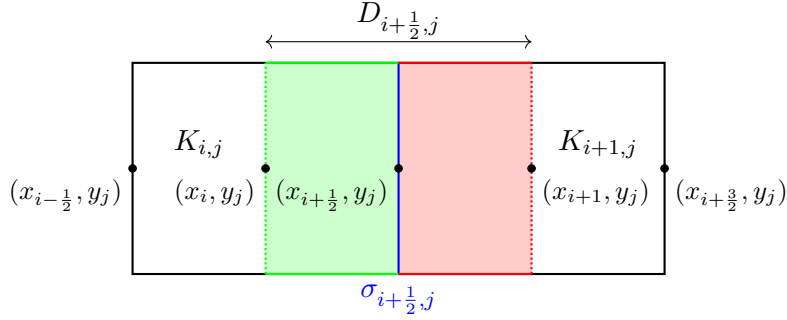

To each edge $\sigma_\ipj$, we associate a horizontal dual cell
$D_\ipj = \lbrack x_i, x_{i+1}\rbrack\times\lbrack y_{j-\half},
y_{j+\half}\rbrack$ and analogously, define a vertical dual cell
$D_\ijp$ associated to the edge $\sigma_\ijp$. By, $\D^{x}$ (resp.\
$\D^{y}$), we denote the collection of all horizontal (resp.\
vertical) dual cells. The space of all scalar valued piecewise
constant functions on the horizontal (resp.\ vertical) dual cells is
denoted by $\mathcal{H}_\E^{x}(\Omega)$ (resp.\
$\mathcal{H}^{y}_\E(\Omega)$). In addition, the set
$\mathcal{H}^{x}_{\E, 0}(\Omega) \subset \mathcal{H}^{x}_\E(\Omega)$
denotes the space of all piecewise constant functions that vanish on
the boundary and analogously, we define
$\mathcal{H}^y_{\E,0}(\Omega)$. 

Let $\Lt(\Omega)$ denote the space of all piecewise constant functions that are constant on each cell $K_\ij$. 
For any $z\in\Lt(\Omega)$, $z_\ij$ will denote its constant value on
$K_\ij\in\T$. We define the standard projection operator $\Pi_\T\colon
L^1(\Omega)\to\Lt(\Omega)$ defined as  
\begin{equation}
\label{eqn:proj-op}
    (\Pi_\T f)_\ij = \frac{1}{\delx\dely}\int_{K_\ij}f(x,y)\,\dx\,\dy.
\end{equation}
For any $z\in\Lt(\Omega)$, its average value and jump across an interface $\sigma_\ipj$ is denoted by 
\[
    \ldblbrace z\rdblbrace_\ipj = \frac{z_\ij + z_{i+1,j}}{2}, \quad \llbracket z\rrbracket_\ipj = z_{i+1,j} - z_{i,j},
\]
and we analogously define $\ldblbrace z\rdblbrace_\ijp$ and $\llbracket z\rrbracket_{\ijp}$.

We introduce the discrete variants of the space derivatives
$\derx,\dery\colon\Lt(\Omega)\to\Lt(\Omega)$, defined as  
\begin{equation}
    \derx z_\ij = \frac{\ldblbrace z\rdblbrace_\ipj - \ldblbrace z\rdblbrace_\imj}{\delx}, \quad \dery z_\ij = \frac{\ldblbrace z\rdblbrace_\ijp - \ldblbrace z\rdblbrace_\ijm}{\dely}. 
\end{equation}

We also define a discrete analogue of the derivative on the dual grids
for the purposes of consistency. The discrete derivative in the
$x$-direction is defined as the map $\eth^\E_x \colon \Lt(\Omega)\to
\mathcal{H}^{x}_{\E, 0}(\Omega)$ while $\eth^\E_y\colon \Lt(\Omega)\to
\mathcal{H}^{y}_{\E,0}(\Omega)$ denotes the $y$-derivative, where 
\begin{equation}
\label{eqn:dual-x-der}
    \eth^\E_x z_{\ipj} = \frac{z_{i+1, j} - z_\ij}{\delx},\quad   \eth^\E_y z_{\ijp} = \frac{z_{i,j+1} - z_\ij}{\dely}. 
\end{equation}

For any $w,z\in\Lt(\Omega)$, the following duality relations hold:
\begin{equation}
\label{eqn:dual-rel}
    \sum_\ij\delx\dely\bigl(w_\ij\derx z_\ij + z_\ij\derx w_\ij\bigr) = \sum_\ij\delx\dely\bigl(w_\ij\dery z_\ij + z_\ij\dery w_\ij\bigr) = 0.
\end{equation}

\subsubsection{Reconstruction Operator}
\label{subsubsec:recon-op}

For the purpose of consistency, we define reconstruction operators
which transform piecewise constant functions on the grid $\T$ to
functions on the grids $\D^x$ or $\D^y$. Namely, we define the
operator
$\mathcal{R}^{x}_\T\colon\Lt(\Omega)\to\mathcal{H}^{x}_\E(\Omega)$ as
follows:  
\begin{align}
    &(\mathcal{R}^{x}_\T z)_\ipj = 
    \begin{dcases}
        \mu_\ipj z_\ij + (1 - \mu_\ipj) z_{i+1, j}, &\text{if }\sigma_\ipj\in\E_{int} \\
        z_\ij, &\text{if }\sigma_\ipj\in\E_{ext},
    \end{dcases}
\end{align}
where $\mu_\ipj\in\lbrack 0,1\rbrack$. Similarly, we define the
operator
$\mathcal{R}^y_\T\colon\Lt(\Omega)\to\mathcal{H}^y_\E(\Omega)$. The
following stability estimates can be easily obtained; see
\cite{HLZ20}: 
\begin{equation}
\label{eqn:recon-stab}
    \lVert\mathcal{R}^x_\T z\rVert_{L^p}\lesssim\lVert z\rVert_{L^p}, \quad \lVert\mathcal{R}^y_\T z\rVert_{L^p}\lesssim\lVert z\rVert_{L^p},
\end{equation}
for any $1\leq p <\infty$.

From \cite{AA24}, we also recall the following estimates. For any
$1\leq p <\infty$, we have  
\begin{align}
    &\lVert\mathcal{R}^{x}_\T z - z\rVert_{L^p}\lesssim
      \delx\lVert\eth^\E_x z\rVert_{L^p},\quad \lVert\mathcal{R}^{y}_\T
      z - z\rVert_{L^p}\lesssim \dely\lVert\eth^\E_y z\rVert_{L^p},
\end{align}
for any $z\in\Lt(\Omega)$.

We now state the following lemma on the weak convergence of these mesh
reconstructions; see \cite{AA24} for a proof. 

\begin{lemma}
\label{lem:mesh-recon-conv}
Let $(\T^{\k})_{k\in\N}$ be a sequence of space
discretisations such that $\delta_{\T^\k}\to 0$ as
$k\to\infty$. Let $(z^\k)_{k\in\N}$ be a uniformly bounded sequence in
$L^p(\Omega)$ such that $z^\k\in\mathcal{L}_{\T^\k}(\Omega)$ for
each $k\in\N$. Let $\mathcal{R}^{x}_k, \mathcal{R}^{y}_k$ denote
the reconstruction operators on $\T^\k$. Then  
\begin{align*}
  \lVert \mathcal{R}^{i}_k z^\k - z^\k\rVert_{L^p}\to 0\text{ as }k\to\infty.
\end{align*}
for $i = x,y$.
    In particular, if $z^\k\rightharpoonup z$ in $L^p(\Omega)$, then 
    \[
        \mathcal{R}^{i}_k z^\k\rightharpoonup z \text{ in }L^p(\Omega).
    \]
\end{lemma}

\subsubsection{The Scheme}
\label{subsubsec:scheme}

Let $0 = t^0 < t^1 < \cdots < t^N = T$ be a discretisation of
$\lbrack0,T\rbrack$ with $\delt = t^{n+1}-t^{n}$ being the constant
time-step. We initialise the scheme by setting  
\[
  h^0_\ij = (\Pi_\T h_0)_\ij,\quad u^0_\ij = (\Pi_\T u_0)_\ij,\quad v^0_\ij = (\Pi_\T v_0)_\ij.
\]

For each $0\leq n\leq N-1$, we consider the following semi-implicit in
time, cell-centred finite volume scheme: 
\begin{subequations}
\label{eqn:scheme}
    \begin{gather}
        \frac{1}{\delt}(h^{n+1}_\ij - h^n_\ij) + \frac{1}{\delx}(\F^{n+1}_\ipj - \F^{n+1}_\imj) + \frac{1}{\dely}(\G^{n+1}_\ijp - \G^{n+1}_\ijm) = 0, \label{eqn:disc-mss-bal} \\
        \frac{1}{\delt}(h^{n+1}_\ij u^{n+1}_\ij - h^n_\ij u^n_\ij) + \frac{1}{\delx}(\F^{n+1}_\ipj u^n_\ipj - \F^{n+1}_\imj u^n_\imj) + \frac{1}{\dely}(\G^{n+1}_\ijp u^n_\ijp \label{eqn:disc-x-mom-bal} \\
        - \G^{n+1}_\ijm u^n_\ijm) 
        + h^{n+1}_\ij\derx\phi^{n+1}_\ij = \omega h^{n+1}_\ij v^n_\ij - \omega r^{n+1}_\ij, \nonumber \\
        \frac{1}{\delt}(h^{n+1}_\ij v^{n+1}_\ij - h^n_\ij v^n_\ij) + \frac{1}{\delx}(\F^{n+1}_\ipj v^n_\ipj - \F^{n+1}_\imj v^n_\imj) + \frac{1}{\dely}(\G^{n+1}_\ijp v^n_\ijp \label{eqn:disc-y-mom-bal} \\
        - \G^{n+1}_\ijm v^n_\ijm) + h^{n+1}_\ij\dery\phi^{n+1}_\ij = -\omega h^{n+1}_\ij u^n_\ij + \omega q^{n+1}_\ij. \nonumber 
    \end{gather}
\end{subequations}
Here, $\phi^{n+1}_\ij = g(h^{n+1}_\ij + b_\ij)$, where $b_\ij =
(\Pi_\T b)_\ij$, and $q^{n+1}_\ij$ and $r^{n+1}_\ij$ are the
stabilisation terms, cf.\ Subsection~\ref{subsec:stab}, whose choices
will be specified after performing the stability analysis. The convective
fluxes in \eqref{eqn:disc-mss-bal} are given by 
\begin{equation}
\label{eqn:mss-flux}
    \begin{split}
    &\F^{n+1}_\ipj = \ldblbrace h^{n+1}u^n\rdblbrace_\ipj - \ldblbrace q^{n+1}\rdblbrace_\ipj ,\quad \G^{n+1}_\ijp = \ldblbrace h^{n+1}v^n\rdblbrace_\ijp - \ldblbrace r^{n+1}\rdblbrace_\ijp.
    \end{split}
\end{equation}

Further, one can rewrite the height update \eqref{eqn:disc-mss-bal} as 
\begin{equation}
    \begin{split}
    \label{eqn:disc-mss-bal-2}
        \frac{1}{\delt}(h^{n+1}_\ij - h^n_\ij) + \derx(h^{n+1}u^n)_\ij + \dery(h^{n+1}v^n)_\ij - \derx q^{n+1}_\ij - \dery r^{n+1}_\ij = 0.
    \end{split}
\end{equation}

The edge centered velocities appearing in the momentum balances
\eqref{eqn:disc-x-mom-bal}-\eqref{eqn:disc-y-mom-bal} are 
upwind, namely  
\begin{equation}
    w^n_\ipj = 
    \begin{dcases}
        w^n_\ij, &\text{if }\F^{n+1}_\ipj\geq 0, \\
        w^n_{i+1, j}, &\text{otherwise},
    \end{dcases}
    \quad
    w^n_\ijp =
    \begin{dcases}
      w^n_\ij, &\text{if }\G^{n+1}_\ijp\geq 0; \\
      w^n_{i, j+1}, &\text{otherwise},
    \end{dcases}
\end{equation}
for $w = u,v$. 

Now, from the momentum balance \eqref{eqn:disc-x-mom-bal} and the mass
balance \eqref{eqn:disc-mss-bal}, we can obtain the following update
of the $x$-velocity: 
\begin{equation}
\label{eqn:disc-x-vel}
    \begin{split}
        \frac{h^{n+1}_\ij}{\delt}(u^{n+1}_\ij - u^n_\ij) + &\frac{1}{\delx}\Bigl(\F^{n+1,-}_\ipj\llbracket u^n\rrbracket_\ipj + \F^{n+1,+}_\imj\llbracket u^n\rrbracket_\imj\Bigr) + \frac{1}{\dely}\Bigl(\G^{n+1,-}_\ijp\llbracket u^n\rrbracket_\ijp \\
        &+\G^{n+1,+}_\ijm\llbracket u^n\rrbracket_\ijm\Bigr)+ h^{n+1}_\ij\derx\phi^{n+1}_\ij = \omega h^{n+1}_\ij v^n_\ij - \omega r^{n+1}_\ij,
    \end{split}
\end{equation}
where $a^{\pm} = (a\pm\abs{a})/2$ so that $a = a^+ + a^-$ with
$a^+\geq 0$ and $a^-\leq 0$. Similarly, we have the $y$-velocity
update: 
\begin{equation}
\label{eqn:disc-y-vel}
     \begin{split}
        \frac{h^{n+1}_\ij}{\delt}(v^{n+1}_\ij - v^n_\ij) + &\frac{1}{\delx}\Bigl(\F^{n+1,-}_\ipj\llbracket v^n\rrbracket_\ipj + \F^{n+1,+}_\imj\llbracket v^n\rrbracket_\imj\Bigr) + \frac{1}{\dely}\Bigl(\G^{n+1,-}_\ijp\llbracket v^n\rrbracket_\ijp \\
        &+\G^{n+1,+}_\ijm\llbracket v^n\rrbracket_\ijm\Bigr)+ h^{n+1}_\ij\dery\phi^{n+1}_\ij = -\omega h^{n+1}_\ij u^n_\ij + \omega q^{n+1}_\ij.
    \end{split}
\end{equation}

Since the pressure forces appear in non-conservative form in
\eqref{eqn:disc-x-mom-bal}–\eqref{eqn:disc-y-mom-bal}, we state the
following proposition, which establishes the discrete balance of the total
momentum. 

\begin{proposition}[Discrete Balance of the Total Momentum]
  The discrete total momentum satisfies the following balance equations. 
  \begin{equation}
    \label{eqn:tot-x-mom-bal}
    \begin{split}
      \sum_\ij\delx\dely\,h^{n+1}_\ij u^{n+1}_\ij = &\sum_{\ij}\delx\dely\, h^n_\ij u^n_\ij - \delt\sum_\ij\delx\dely\,gh^{n+1}_\ij\derx b_\ij \\
      &+\delt\sum_\ij\delx\dely(\omega h^{n+1}_\ij v^n_\ij - \omega r^{n+1}_\ij).
    \end{split}
  \end{equation}
  \begin{equation}
    \label{eqn:tot-y-mom-bal}
    \begin{split}
      \sum_\ij\delx\dely\,h^{n+1}_\ij v^{n+1}_\ij = &\sum_{\ij}\delx\dely\, h^n_\ij v^n_\ij - \delt\sum_\ij\delx\dely\,gh^{n+1}_\ij\dery b_\ij \\
      &+\delt\sum_\ij\delx\dely(-\omega h^{n+1}_\ij u^n_\ij + \omega q^{n+1}_\ij).
    \end{split}
  \end{equation}
\end{proposition}
\begin{proof}
    For any piecewise constant function $z\in\Lt(\Omega)$, observe that 
    \begin{equation}
        z_\ij\derx z_\ij =  \derx\Bigl(\frac{z^2}{2}\Bigr)_\ij - \frac{(\llbracket z\rrbracket_{\ipj}^2 - \llbracket z\rrbracket_{\imj}^2)}{4\delx}, \quad  z_\ij\dery z_\ij = \dery\Bigl(\frac{z^2}{2}\Bigr)_\ij - \frac{(\llbracket z\rrbracket_{\ijp}^2 - \llbracket z\rrbracket_{\ijm}^2)}{4\dely}.
    \end{equation}
    Consequently, 
    \[
        \sum_{\ij}\delx\dely\,z_\ij\derx z_\ij = \sum_\ij\delx\dely\,z_\ij\dery z_\ij = 0.
    \]
    The relations \eqref{eqn:tot-x-mom-bal}-\eqref{eqn:tot-y-mom-bal} then follow upon substituting for $\phi^{n+1}_\ij$ in the equations \eqref{eqn:disc-x-mom-bal}-\eqref{eqn:disc-y-mom-bal} and using the above identity.
\end{proof}

\subsection{Main Results and Organization}

In this section, we present the main results of this work concerning the
energy stability, well-balancing, and consistency of the scheme. We
begin with the main theorem, which establishes the stability of the
numerical solutions generated by the proposed scheme. 

\begin{theorem}[Energy Stability]
  \label{thm:eng-stab}
  Suppose the stabilisation terms are chosen as
    \begin{equation}
    \label{eqn:stab-term-exp}
        q^{n+1}_\ij = \eta\delt(\derx\phi^{n+1}_\ij - \omega v^n_\ij),\quad r^{n+1}_\ij = \eta\delt(\dery\phi^{n+1}_\ij +\omega u^n_\ij).
    \end{equation}
    Further, assume that the following conditions hold:
    \begin{enumerate}
        \item $\eta > \dfrac{3}{2} h^{n+1}_\ij$;
        \item $\displaystyle 1 + \frac{3\delt}{h^{n+1}_\ij}\biggl(\frac{1}{\delx}\Bigl(\F^{n+1,-}_\ipj - \F^{n+1,+}_\imj\Bigr) + \frac{1}{\dely}\Bigl(\G^{n+1,-}_\ijp - \G^{n+1,+}_\ijm\Bigr)\biggr)\geq 0$;
        \item  $\displaystyle\delt^2\leq\zeta\biggl(\frac{2h^{n+1}_\ij}{3\omega^2\eta^2}\Bigl(\eta - \frac{3}{2}h^{n+1}_\ij\Bigr)\biggr)$, $\zeta\in(0,1)$.
    \end{enumerate}
    Then, for any $0\leq n\leq N-1$, the following local in-time energy inequality holds:
    \begin{equation}
    \label{eqn:loc-eng-ineq}
        \sum_{\ij}\delx\dely E^{n+1}_\ij\leq \sum_{\ij}\delx\dely E^n_\ij.
    \end{equation}
    Furthermore, we can also establish the following global energy estimate for $1\leq n\leq N$:
    \begin{gather}
    \label{eqn:glob-ent-ineq}
        \sum_{\ij}\delx\dely E^n_\ij + \frac{g}{2}\sum_{k=0}^{n-1}\sum_\ij\delx\dely(h^{k+1}_\ij - h^k_\ij)^2 \\ 
        +\sum_{k=0}^{n-1}\sum_{\ij}\delx\dely\frac{(1-\zeta)}{\eta^2}\Bigl(\eta - \frac{3}{2}h^{k+1}_\ij\Bigr)\bigl((q^{k+1}_\ij)^2 + (r^{k+1}_\ij)^2\bigr) \leq\sum_{\ij}\delx\dely E^0_\ij. \nonumber
    \end{gather}
\end{theorem}

The time-step restriction stated in condition \textit{(2)} of the above
theorem originates from the upwind choice of the momentum convection
fluxes. However, its direct implementation is not straightforward.
Therefore, we introduce an alternative sufficient condition, which is
simpler to apply in practice. This condition can be derived in a manner
similar to that in \cite{CDV17}. 

\begin{lemma}[Sufficient Time-step Condition]
\label{lem:suff-tstep}
    Suppose that the time-step $\delt$ satisfies 
    \begin{equation}
    \label{eqn:suff-tstep}
        \begin{split}
            &\delt\biggl(\frac{2}{\delx} + \frac{2}{\dely}\biggr)(\max\lbrace\abs{u^n_{\ij}},\abs{u^{n}_{i+1,j}}\rbrace + \sqrt{\frac{\eta}{\max\lbrace h^{n+1}_\ij, h^{n+1}_{i+1,j}\rbrace}}\Lambda^{n+1}_\ipj) < \min\biggl\lbrace 1, \frac{\min\lbrace h^n_\ij, h^n_{i+1,j}\rbrace}{\max\lbrace h^{n+1}_\ij, h^{n+1}_{i+1,j}\rbrace}\biggr\rbrace \\
            &\delt\biggl(\frac{2}{\delx} + \frac{2}{\dely}\biggr)(\max\lbrace\abs{v^n_{\ij}},\abs{v^{n}_{i,j+1}}\rbrace + \sqrt{\frac{\eta}{\max\lbrace h^{n+1}_\ij, h^{n+1}_{i,j+1}\rbrace}}\Lambda^{n+1}_\ijp) < \min\biggl\lbrace 1, \frac{\min\lbrace h^n_\ij, h^n_{i,j+1}\rbrace}{\max\lbrace h^{n+1}_\ij, h^{n+1}_{i,j+1}\rbrace}\biggr\rbrace,
        \end{split}
    \end{equation}
    where $\Lambda^{n+1}_\ipj = \sqrt{\max\lbrace\abs{\derx\phi^{n+1}_\ij - \omega v^n_\ij}, \abs{\derx\phi^{n+1}_{i+1,j} - \omega v^n_{i+1,j}}\rbrace}$ and\\ 
    $\Lambda^{n+1}_\ijp = \sqrt{\max\lbrace\abs{\dery\phi^{n+1}_\ij + \omega u^n_\ij}, \abs{\dery\phi^{n+1}_{i,j+1} + \omega u^n_{i,j+1}}\rbrace}$. Then, $\delt$ satisfies the condition (2) given in Theorem \ref{thm:eng-stab}.
\end{lemma}

Note that the stabilisation terms in the balance of the water height
\eqref{eqn:disc-mss-bal-2} lead to a linear elliptic problem for
$\phi^{n+1}_\ij$. Under the time-step restriction
\eqref{eqn:suff-tstep}, we can show that the matrix associated with this
linear problem is invertible, thereby guaranteeing a unique solution
$\phi^{n+1}_\ij$. Consequently, we recover $h^{n+1}_\ij =
\frac{\phi^{n+1}_\ij}{g} - b_\ij$. The velocities $u^{n+1}_\ij$ and
$v^{n+1}_\ij$ are then obtained explicitly from
\eqref{eqn:disc-x-mom-bal}–\eqref{eqn:disc-y-mom-bal}. Furthermore,
the positivity of the water height $h^{n+1}_\ij$ is ensured, provided
$h^n_\ij > 0$. We summarise these observations in the following
theorem. 

\begin{theorem}[Existence and Positivity of a Solution]
\label{thm:exis-pos}
    Assume that $(h^n_\ij, u^n_\ij, v^n_\ij)$ is known with
    $h^n_\ij>0$ for each $(\ij)$. Then, under the time-step
    restriction \eqref{eqn:suff-tstep}, there exists a solution
    $(h^{n+1}_\ij, u^{n+1}_\ij, v^{n+1}_\ij)$ of the scheme
    \eqref{eqn:scheme} such that $h^{n+1}_\ij >0$ for each $(\ij)$. In
    particular, if $h^0_\ij>0$, then $h^n_\ij>0$ for each $1\leq n\leq
    N$.  
\end{theorem}

The next result addresses the well-balancing property of the scheme
and is stated as follows. 
\begin{theorem}[Well-Balancing]
\label{thm:wb}
    Suppose that the solution $(h^n_\ij, u^n_\ij, v^n_\ij)$ at time $t^n$ satisfies
    \begin{equation}
    \label{eqn:disc-jet-x}
        u^n_\ij = 0,\ \dery h^n_\ij = 0,\ \dery b_\ij = 0,\ 
        \dery v^n_\ij = 0,\ \derx\phi^n_\ij - \omega v^n_\ij = 0
    \end{equation}
    for each ($\ij$). Then, under the time-step restriction \eqref{eqn:suff-tstep}, the updated solution $(h^{n+1}_\ij, u^{n+1}_\ij, v^{n+1}_\ij)$ at time $t^{n+1}$ satisfies $(h^{n+1}_\ij, u^{n+1}_\ij, v^{n+1}_\ij) = (h^n_\ij, u^n_\ij, v^n_\ij)$ and
    \[
        u^{n+1}_\ij = 0,\ \dery h^{n+1}_\ij = 0,\ \dery v^{n+1}_\ij = 0,\ \derx\phi^{n+1}_\ij - \omega v^{n+1}_\ij = 0.
    \]
    An analogous conclusion holds if we assume the condition \eqref{eqn:jet-y} at time $t^n$.
\end{theorem}

To prove the consistency of the scheme, we require some additional
assumptions, namely the boundedness of the water height and a bound on
the discrete derivatives of the water height. The former assumption is
needed to obtain bounds on the stabilisation terms, ensuring that
these terms vanish in the limit $\delta_\T \to 0$. The latter bound on
the discrete derivatives is required to `convert’ the non-conservative
pressure term into conservative form. 

In the following, given a piecewise constant function $(z_\ij)_{\ij}$ on the
mesh $\T$, we denote by $z_\T = \sum_{\ij}z_\ij \X_\ij$ where $\X_\ij$
denotes the indicator function of the cell $K_\ij$.

The statement of the theorem is as follows. 
\begin{theorem}[Consistency]
\label{thm:cons}
    Suppose that there exist constants $0<\underline{h}\leq\overline{h}$ such that 
    \begin{equation}
    \label{eqn:h-bnd}
        0<\underline{h}\leq h^n_\ij\leq\overline{h}
    \end{equation}
    uniformly for each $n$ and $(\ij)$. Also, assume that 
    \begin{equation}
    \label{eqn:wk-bv-assmp}
        \norm{\derx^\E h^{n+1}_\T}_{L^2}\lesssim \delx^{-1/2},\quad \norm{\dery^\E h^{n+1}_\T}_{L^2}\lesssim \dely^{-1/2}. 
    \end{equation}
    Further, suppose that $\delta_\T\in (0,\delta_0]$ for $\delta_0$ sufficiently small with $\delt\approx\delta_\T$. Then, the numerical solutions generated by the scheme are consistent with the RSW system \eqref{eqn:cv-rsw}, i.e.\
    \begin{equation}
    \label{eqn:cons-mss-bal}
            -\iint\limits_{\Omega} h^0_\T\psi(0,\cdot)\,\dx\,\dy = \sum_{n = 0}^{N-1}\int\limits_{t^{n}}^{t^{n+1}}\!\!\!\!\iint\limits_{\Omega}\bigl\lbrack h^n_\T\Dt\psi + h^{n+1}_\T u^n_\T\Dx\psi + h^{n+1}_\T v^n_\T\Dy\psi\bigr\rbrack\,\dx\,\dy\,\dt + \mathcal{C}^{mass}_{\T,\delt},
    \end{equation}
    for any $\psi\in\Cinf([0,T)\times\Omega)$;
    \begin{gather}
    \label{eqn:cons-xmom-bal}
        -\iint\limits_{\Omega}h^0_\T u^0_\T \Psi(0,\cdot)\,\dx\,\dy = \sum_{n = 0}^{N-1}\int\limits_{t^{n}}^{t^{n+1}}\!\!\!\!\iint\limits_{\Omega} \lbrack h^n_\T u^n_\T\Dt\Psi + m^{n+1}_{x,\D^{x}} u_{\D^{x}}^{n}\Dx\Psi + m^{n+1}_{y, \D^{y}} u^n_{\D^{y}}\Dy\Psi \\
        + p^{n+1}_\T\Dx\Psi\rbrack\,\dx\,\dy\,\dt + \sum_{n = 0}^{N-1}\int\limits_{t^{n}}^{t^{n+1}}\!\!\!\!\iint\limits_{\Omega}\lbrack-gh^{n+1}_\T\Dx b\,\Psi + \omega h^{n+1}_\T v^n_\T\Psi\rbrack\,\dx\,\dy\,\dt + \mathcal{C}^{mom,x}_{\T,\delt} = 0, \nonumber
    \end{gather}
    for any $\Psi\in\Cinf([0,\T)\times\Omega)$;
    \begin{gather}
    \label{eqn:cons-ymom-bal}
        -\iint\limits_{\Omega}h^0_\T v^0_\T \Phi(0,\cdot)\,\dx\,\dy = \sum_{n = 0}^{N-1}\int\limits_{t^{n}}^{t^{n+1}}\!\!\!\!\iint\limits_{\Omega} \lbrack h^n_\T v^n_\T\Dt\Phi + m^{n+1}_{x,\D^{x}} v_{\D^{x}}^{n}\Dx\Phi + m^{n+1}_{y, \D^{y}} v^n_{\D^{y}}\Dy\Phi \\
        + p^{n+1}_\T\Dy\Phi\rbrack\,\dx\,\dy\,\dt + \sum_{n = 0}^{N-1}\int\limits_{t^{n}}^{t^{n+1}}\!\!\!\!\iint\limits_{\Omega}\lbrack-gh^{n+1}_\T\Dy b\,\Phi - \omega h^{n+1}_\T u^n_\T\Phi\rbrack\,\dx\,\dy\,\dt + \mathcal{C}^{mom,y}_{\T,\delt} = 0, \nonumber
    \end{gather}
    for any $\Phi\in\Cinf([0,T)\times\Omega)$. Here, any function with the subscript $\D^x$ (resp.\ $\D^y$) denotes the reconstruction of said function on the dual grid $\D^x$ (resp.\ $\D^y$). The consistency errors are such that 
    \[
    \abs{\mathcal{C}^{mass}_{\T,\delt}}, \abs{\mathcal{C}^{mom,x}_{\T,\delt}}, \abs{\mathcal{C}^{mom,y}_{\T,\delt}}\to 0\ \text{as}\ \delta_\T\to 0.
    \]
\end{theorem}

Finally, we present the main theorem stating the convergence of numerical solutions of the present scheme to a DMV solution of the RSW system.

\begin{theorem}[Convergence]
\label{thm:conv}
    Let $\lbrace(\T^\k, \delt^\k)\rbrace_{k\in\N}$ be a sequence of space-time discretizations such that $\delta_\T^\k\to 0$ as $k\to\infty$ and let the assumptions of Theorem \ref{thm:cons} hold true. Let $\lbrace(h^\k, m_x^\k, m_y^\k)\rbrace_{k\in\N}$ be the sequence of numerical solutions generated by the scheme \eqref{eqn:scheme} corresponding to $\lbrace(\T^\k, \delt^\k)\rbrace_{k\in\N}$. Then, 
    \begin{gather*}
        h^\k\weakstar\langle\V;\th\rangle\ \text{in}\ L^\infty(\odom); \\
        m_x^\k\weakstar\langle\V;\tmx\rangle\ \text{and}\ m_y^\k\weakstar\langle\V;\tmy\rangle\ \text{in}\ L^\infty(0,T; L^2(\Omega)),
    \end{gather*}
    where $\mathcal{V} = \lbrace\V\rbrace_{(t,x,y)\in\odom}$ is a DMV solution of the RSW system in the sense of Definition \ref{def:dmv-sol}.
\end{theorem}

The remainder of this manuscript is organised as
follows. Section~\ref{sec:eng-stab} is devoted to the proof of
Theorem~\ref{thm:eng-stab}. In Section~\ref{sec:wb}, we establish
Theorems~\ref{thm:exis-pos}–\ref{thm:wb}. Section~\ref{sec:cons-conv}
contains the proofs of
Theorems~\ref{thm:cons}–\ref{thm:conv}. Numerical results illustrating
the performance of the scheme are presented in
Section~\ref{sec:num-res}, and concluding remarks are given in
Section~\ref{sec:conc}. 

\section{Energy Stability}
\label{sec:eng-stab}

In this section, we prove Theorem \ref{thm:eng-stab}. To begin with,
we first establish discrete identities satisfied by the potential energy $\I$ and the kinetic energy $\K$.

\begin{proposition}[Discrete Energy Identities]
    The discrete solutions of the scheme \eqref{eqn:scheme} satisfy
    \begin{enumerate}
        \item a discrete potential energy identity
        \begin{equation}
        \label{eqn:disc-pot-bal}
            \begin{split}
                \frac{1}{\delt}(\I^{n+1}_\ij - \I^n_\ij) + \phi^{n+1}_\ij\lbrack\derx(h^{n+1}u^n)_\ij + \dery(h^{n+1}v^n)_\ij\rbrack + \frac{g(h^{n+1}_\ij - h^n_\ij)^2}{2\delt}\\
                = \phi^{n+1}_\ij\lbrack\derx q^{n+1}_\ij + \dery r^{n+1}_\ij\rbrack ,
            \end{split}
        \end{equation}
        wherein we have set $\I^k_\ij = \half g (h^k_\ij)^2 + gb_\ij h^k_\ij$ for $k =n,n+1$;

        \item a discrete kinetic energy identity
        \begin{gather}
        \label{eqn:disc-ke}
            \frac{1}{\delt}(\K^{n+1}_\ij - \K^n_\ij) + \frac{1}{2\delx}\Bigl\lbrack\F^{n+1}_\ipj((u^n_\ipj)^2+(v^n_\ipj)^2) -\F^{n+1}_\imj((u^n_\imj)^2+(v^n_\imj)^2)\Bigr\rbrack \\
            +\frac{1}{2\dely}\Bigl\lbrack\G^{n+1}_\ijp((u^n_\ijp)^2+(v^n_\ijp)^2) -\G^{n+1}_\ijm((u^n_\ijm)^2+(v^n_\ijm)^2)\Bigr\rbrack + h^{n+1}_\ij u^{n}_\ij\derx\phi^{n+1}_\ij \nonumber \\
            +h^{n+1}_\ij v^n_\ij\dery\phi^{n+1}_\ij = \omega v^n_\ij q^{n+1}_\ij-\omega u^n_\ij r^{n+1}_\ij + R^{n+1}_\ij, \nonumber 
        \end{gather}
        wherein we have set $\K^k_\ij = \half h^k_\ij((u^k_\ij)^2 + (v^k_\ij)^2)$ for $k=n,n+1$ and the remainder term $R^{n+1}_\ij$ reads 
        \begin{gather}
        \label{eqn:kin-rem}
            R^{n+1}_\ij = \frac{h^{n+1}_\ij}{2\delt}\lbrack(u^{n+1}_\ij - u^n_\ij)^2+(v^{n+1}_\ij - v^n_\ij)^2\rbrack \\
            + \frac{1}{2\delx}\Bigl(\F^{n+1,-}_{\ipj}\bigl(\llbracket u^n\rrbracket_\ipj^2+\llbracket v^n\rrbracket_\ipj^2\bigr) - \F^{n+1,+}_{\imj}\bigl(\llbracket u^n\rrbracket_\imj^2 + \llbracket v^n\rrbracket_\imj^2\bigr)\Bigr) \nonumber \\
            + \frac{1}{2\dely}\Bigl(\G^{n+1,-}_\ijp\bigl(\llbracket u^n\rrbracket^2_\ijp + \llbracket v^n\rrbracket^2_\ijp\bigr) - \G^{n+1,+}_\ijm\bigl(\llbracket u^n\rrbracket^2_\ijm + \llbracket v^n\rrbracket^2_\ijm\bigr)\Bigr).\nonumber
        \end{gather}
    \end{enumerate}
\end{proposition}

\begin{proof}
    In order to establish \eqref{eqn:disc-pot-bal}, we note that the
    potential energy $\I$ satisfies $\partial_h\I = \phi$,
    $\partial^2_{h}\I = g$. Then, performing a Taylor expansion in
    time for the difference $\I^{n+1}_\ij - \I^n_\ij$ and using the
    balance of the water height \eqref{eqn:disc-mss-bal-2} to simplify
    will yield the desired equation. To obtain \eqref{eqn:disc-ke}, we
    multiply the velocity balance \eqref{eqn:disc-x-vel} by $u^n_\ij$,
    \eqref{eqn:disc-y-vel} by $v^n_\ij$ and add the two. Then, we use
    the identity $(c-d)d = (c^2 - d^2 - (c-d)^2)/2$ to simplify the
    the time derivative terms as well as the flux terms. Finally, we
    multiply \eqref{eqn:disc-mss-bal} by $((u^n_\ij)^2 +
    (v^n_\ij)^2)/2$, add it to the resultant expression and further
    simplify using the definition of the positive and negative part of
    the fluxes to obtain \eqref{eqn:disc-ke}. We refer to \cite{AA24}
    for detailed calculations.
\end{proof}

We are now in a position to prove Theorem \ref{thm:eng-stab}.

\begin{proof}[Proof of Theorem \ref{thm:eng-stab}]
    To begin, we first estimate the remainder term \eqref{eqn:kin-rem}
    which is present in the discrete kinetic energy balance
    \eqref{eqn:disc-ke}. Using the $x$-velocity balance
    \eqref{eqn:disc-x-vel} and the identity $(A+B+C)^2\leq
    3(A^2+B^2+C^2)$ we obtain 
    \begin{gather}
    \label{eqn:kin-rem-est-1}
        \frac{h^{n+1}_\ij}{2\delt}(u^{n+1}_\ij - u^n_\ij)^2 \leq \frac{3\delt}{2 h^{n+1}_\ij}\biggl\lbrack\biggl\lbrace\frac{1}{\delx}\Bigl(\F^{n+1,-}_\ipj\llbracket u^n\rrbracket_\ipj + \F^{n+1,+}_\imj\llbracket u^n\rrbracket_\imj\Bigr) \\
        + \frac{1}{\dely}\Bigl(\G^{n+1,-}_\ijp\llbracket u^n\rrbracket_\ijp
        +\G^{n+1,+}_\ijm\llbracket
        u^n\rrbracket_\ijm\Bigr)\biggr\rbrace^2 +
        (h^{n+1}_\ij)^2(\derx\phi^{n+1}_\ij - \omega v^n_\ij)^2 +
        \omega^2 (r^{n+1}_\ij)^2\biggr\rbrack. \nonumber 
    \end{gather}
    Denoting the entire term enclosed in the curly braces as
    $\mathcal{A}^{n+1}_\ij$, we rewrite it as   
    \begin{gather*}
        \mathcal{A}^{n+1}_\ij = \biggl\lbrack\biggl(\sqrt{\frac{-1}{\delx}\F^{n+1,-}_\ipj}\llbracket u^n\rrbracket_\ipj\biggr)\cdot\biggl(-\sqrt{\frac{-1}{\delx}\F^{n+1,-}_\ipj}\biggr) + \biggl(\sqrt{\frac{1}{\delx}\F^{n+1,+}_\imj}\llbracket u^n\rrbracket_\imj\biggr)\cdot\biggl(\sqrt{\frac{1}{\delx}\F^{n+1,+}_\imj}\biggr) \\
        + \biggl(\sqrt{\frac{-1}{\dely}\G^{n+1,-}_\ijp}\llbracket u^n\rrbracket_\ijp\biggr)\cdot\biggl(-\sqrt{\frac{-1}{\dely}\G^{n+1,-}_\ijp}\biggr) + \biggl(\sqrt{\frac{1}{\dely}\G^{n+1,+}_\ijm}\llbracket u^n\rrbracket_\ijm\biggr)\cdot\biggl(\sqrt{\frac{1}{\dely}\G^{n+1,+}_\ijm}\biggr)\biggr\rbrack^2.
    \end{gather*}
    Next, we use the Cauchy-Schwarz inequality to estimate the above
    expression and obtain 
    \begin{align*}
        \mathcal{A}^{n+1}_\ij \leq &\biggl\lbrack\frac{1}{\delx}\Bigl(\F^{n+1,-}_{\ipj}\llbracket u^n\rrbracket_\ipj^2 - \F^{n+1,+}_{\imj}\llbracket u^n\rrbracket_\imj^2\Bigr)
        + \frac{1}{\dely}\Bigl(\G^{n+1,-}_\ijp\llbracket u^n\rrbracket^2_\ijp - \G^{n+1,+}_\ijm\llbracket u^n\rrbracket^2_\ijm\Bigr)\biggr\rbrack \\
        &\times\biggl\lbrack\frac{1}{\delx}\Bigl(\F^{n+1,-}_\ipj - \F^{n+1,+}_\imj\Bigr) + \frac{1}{\dely}\Bigl(\G^{n+1,-}_\ijp - \G^{n+1,+}_\ijm\Bigr)\biggr\rbrack.
    \end{align*}
    Combining the above with \eqref{eqn:kin-rem-est-1} yields an estimate on $\frac{h^{n+1}_\ij}{2\delt}(u^{n+1}_\ij - u^n_\ij)^2$. Analogously, one obtains a similar estimate for the term $\frac{h^{n+1}_\ij}{2\delt}(v^{n+1}_\ij - v^n_\ij)^2$. These estimates in turn yield 
    \begin{gather}
    \label{eqn:kin-rem-est-2}
        R^{n+1}_\ij\leq\biggl\lbrace\frac{\F^{n+1,-}_\ipj}{2\delx}\Bigl(\llbracket u^n\rrbracket_\ipj^2 + \llbracket v^n\rrbracket_\ipj^2\Bigr) - \frac{\F^{n+1,+}_\imj}{2\delx}\Bigl(\llbracket u^n\rrbracket_\imj^2 + \llbracket v^n\rrbracket_\imj^2\Bigr) \\
        +\frac{\G^{n+1,-}_\ijp}{2\delx}\Bigl(\llbracket u^n\rrbracket_\ijp^2 + \llbracket v^n\rrbracket_\ijp^2\Bigr) - \frac{\G^{n+1,+}_\ijm}{2\delx}\Bigl(\llbracket u^n\rrbracket_\ijm^2 + \llbracket v^n\rrbracket_\ijm^2\Bigr) \biggr\rbrace \nonumber \\
        \times\biggl\lbrace 1 + \frac{3\delt}{h^{n+1}_\ij}\biggl(\frac{1}{\delx}\Bigl(\F^{n+1,-}_\ipj - \F^{n+1,+}_\imj\Bigr)
        + \frac{1}{\dely}\Bigl(\G^{n+1,-}_\ijp - \G^{n+1,+}_\ijm\Bigr)\biggr)\biggr\rbrace + \frac{3}{2}\delt h^{n+1}_\ij\bigl((\derx\phi^{n+1}_\ij - \omega v^n_\ij)^2 \nonumber\\
        +(\dery\phi^{n+1}_\ij + \omega u^n_\ij)^2\bigr) + \frac{3\omega^2\delt}{2h^{n+1}_\ij}\bigl((q^{n+1}_\ij)^2 + (r^{n+1}_\ij)^2\bigr).\nonumber
    \end{gather}
    It is easy to observe that the product of the two terms that are
    enclosed in the curly braces will remain non-positive under the
    restriction 
    \begin{equation}
    \label{eqn:cfl-cond}
        1 + \frac{3\delt}{h^{n+1}_\ij}\biggl(\frac{1}{\delx}\Bigl(\F^{n+1,-}_\ipj - \F^{n+1,+}_\imj\Bigr) + \frac{1}{\dely}\Bigl(\G^{n+1,-}_\ijp - \G^{n+1,+}_\ijm\Bigr)\biggr)\geq 0.
    \end{equation}
    Consequently, we have the following estimate on the remainder term $R^{n+1}_\ij$:
    \begin{gather}
    \label{eqn:kin-rem-est-3}
        R^{n+1}_\ij\leq\frac{3}{2}\delt h^{n+1}_\ij\bigl((\derx\phi^{n+1}_\ij - \omega v^n_\ij)^2 +(\dery\phi^{n+1}_\ij + \omega u^n_\ij)^2\bigr) + \frac{3\omega^2\delt}{2h^{n+1}_\ij}\bigl((q^{n+1}_\ij)^2 + (r^{n+1}_\ij)^2\bigr).
    \end{gather}

    Next, we multiply the potential energy balance \eqref{eqn:disc-pot-bal} with $\abs{K_\ij} = \delx\dely$ and sum over all $\ij$. Similarly, we multiply the kinetic energy balance \eqref{eqn:disc-ke} with $\delx\dely$ and sum over all $\ij$, and this causes the convective flux terms to sum to 0. Then, we add the two resulting expressions and repeatedly use the duality relations \eqref{eqn:dual-rel} to get
    \begin{equation}
    \label{eqn:eng-ineq-2}
        \begin{split}
        \sum_{\ij}\delx\dely\frac{1}{\delt}(E^{n+1}_\ij - E^n_\ij) = &-\sum_\ij\delx\dely q^{n+1}_\ij\bigl(\derx\phi^{n+1}_\ij - \omega v^n_\ij\bigr) - \sum_\ij\delx\dely r^{n+1}_\ij\bigl(\dery\phi^{n+1}_\ij + \omega u^n_\ij\bigr) \\
        & - \frac{g}{2\delt}\sum_{\ij}\delx\dely (h^{n+1}_\ij - h^n_\ij)^2+ \sum_\ij\delx\dely R^{n+1}_\ij.
        \end{split}
    \end{equation}

    Now, we choose the stabilisation terms as 
    \begin{align}
    \label{eqn:stab-terms}
        q^{n+1}_\ij = \eta\delt\bigl(\derx\phi^{n+1}_\ij - \omega v^n_\ij\bigr),\quad r^{n+1}_\ij = \eta\delt\bigl(\dery\phi^{n+1}_\ij + \omega u^n_\ij\bigr),
    \end{align}
    where $\eta>0$. Next, using the estimate \eqref{eqn:kin-rem-est-3}
    for the remainder term $R^{n+1}_\ij$, substituting
    $\derx\phi^{n+1}_\ij - \omega v^n_\ij =
    \frac{q^{n+1}_\ij}{\eta\delt}$, $\dery\phi^{n+1}_\ij + \omega
    u^n_\ij = \frac{r^{n+1}_\ij}{\eta\delt}$ in the inequality
    \eqref{eqn:eng-ineq-2}, we get  
    \begin{equation}
    \label{eqn:eng-ineq-3}
        \begin{split}
            \sum_{\ij}\delx\dely\frac{1}{\delt}(E^{n+1}_\ij - E^n_\ij) \leq &\sum_\ij\delx\dely\,\biggl(\frac{3 h^{n+1}_\ij}{2\eta^2\delt} + \frac{3\omega^2\delt}{2h^{n+1}_\ij} - \frac{1}{\eta\delt}\biggr)[(q^{n+1}_\ij)^2 + (r^{n+1}_\ij)^2] \\
            &- \frac{g}{2\delt}\sum_{\ij}\delx\dely (h^{n+1}_\ij - h^n_\ij)^2.
        \end{split}
    \end{equation}
    Observe that under the conditions,
    \begin{equation}
    \label{eqn:eta-aux-tstep-cond}
        \eta>\frac{3}{2}h^{n+1}_\ij, \quad
        \delt^2\leq\zeta\biggl(\frac{2h^{n+1}_\ij}{3\omega^2\eta^2}\Bigl(\eta - \frac{3}{2}h^{n+1}_\ij\Bigr)\biggr),
    \end{equation}
    for $\zeta\in(0,1)$, we get
    \begin{equation}
    \label{eqn:eng-ineq-4}
        \frac{3 h^{n+1}_\ij}{2\eta^2\delt} + \frac{3\omega^2\delt}{2h^{n+1}_\ij} - \frac{1}{\eta\delt}\leq \frac{(\zeta-1)}{\eta^2\delt}\Bigl(\eta - \frac{3}{2}h^{n+1}_\ij\Bigr)\leq 0
    \end{equation}
    Hence, from \eqref{eqn:eng-ineq-3} and \eqref{eqn:eng-ineq-4}, we get the desired local in-time energy inequality \eqref{eqn:loc-eng-ineq}.
    Next, to prove the global estimate, note that from \eqref{eqn:eng-ineq-3} and \eqref{eqn:eng-ineq-4}, we have for any $0\leq k\leq N-1$
    \begin{gather*}
        \sum_{\ij}\delx\dely(E^{k+1}_\ij - E^k_\ij) + \frac{g}{2}\sum_\ij\delx\dely(h^{k+1}_\ij - h^k_\ij)^2 \\ 
        +\sum_\ij\delx\dely\frac{(1-\zeta)}{\eta^2}\Bigl(\eta - \frac{3}{2}h^{k+1}_\ij\Bigr)\bigl((q^{k+1}_\ij)^2 + (r^{k+1}_\ij)^2\bigr) \leq 0.
    \end{gather*}
    Then, for $1\leq n\leq N$, summing up the above inequality from $k=0$ to $k=n-1$ yields \eqref{eqn:glob-ent-ineq}.
\end{proof}

\section{Well-Balancing}
\label{sec:wb}

Since the stabilisation terms are given by \eqref{eqn:stab-terms}, one
can observe that the balance for the water height, cf.\
\eqref{eqn:disc-mss-bal-2}, is implicit in nature. Therefore, we now
proceed with the proof of the existence of a solution to the present
scheme, i.e.\ Theorem \ref{thm:exis-pos}, under the time-step
restriction \eqref{eqn:suff-tstep}. 

\begin{proof}[Proof of Theorem \ref{thm:exis-pos}]
    Using $h^{n}_\ij = (\phi^n_\ij/g) - b_\ij$, we rewrite the mass
    balance \eqref{eqn:disc-mss-bal-2} as  
    \begin{equation}
    \label{eqn:exist-1}
        \begin{split}
            \phi^{n+1}_\ij - \phi^n_\ij + \delt\derx(\phi^{n+1}u^n)_\ij + \delt\dery(\phi^{n+1}v^n)_\ij - g\delt\derx q^{n+1}_\ij - g\delt\dery r^{n+1}_\ij \\
            = g\delt\derx(bu^n)_\ij + g\delt\dery(bv^n)_\ij
        \end{split}
    \end{equation}
    Next, substituting for the stabilisation terms from
    \eqref{eqn:stab-terms} and rearranging, we obtain the following
    discrete linear elliptic problem for $\phi^{n+1}$:
    \begin{equation}
    \label{eqn:ellpt-prob}
        \begin{split}
            \phi^{n+1}_\ij - g\eta\delt^2\derx(\derx\phi^{n+1})_\ij - g\eta\delt^2\dery(\dery\phi^{n+1})_\ij + \delt\derx(\phi^{n+1}u^n)_\ij + \delt\dery(\phi^{n+1}v^n)_\ij \\
            = \phi^n_\ij + g\delt\derx(bu^n)_\ij + g\delt\dery(bv^n)_\ij -g\omega\eta\delt^2\derx v^n_\ij + g\omega\eta\delt^2\dery u^n_\ij.
        \end{split}
    \end{equation}

    Now, upon substituting for the discrete derivatives and expanding, we see that the matrix associated to the given problem is strictly diagonally dominant under the conditions
    \[
        \frac{\delt}{\delx}\abs{u^n_\ij}<\half,\quad \frac{\delt}{\dely}\abs{v^n_\ij}<\half,
    \]
    which are ensured under the time-step restriction \eqref{eqn:suff-tstep}. Consequently, there exists a unique solution $\phi^{n+1}$ which solves \eqref{eqn:ellpt-prob} and hence, we can obtain $h^{n+1}_\ij = (\phi^{n+1}_\ij/g) - b_\ij$. Furthermore, the velocities $u^{n+1}_\ij$ and $v^{n+1}_\ij$ can now be evaluated explicitly from \eqref{eqn:disc-x-mom-bal}-\eqref{eqn:disc-y-mom-bal} respectively.

    Now, to prove the positivity of the water height, we first note that performing straightforward calculations will yield the following condition from \eqref{eqn:suff-tstep}:
    \begin{equation}
    \label{eqn:exp-time-step}
        \frac{\delt}{\delx}\Bigl(\lvert\F^{n+1,}_\ipj\rvert + \lvert\F^{n+1}_\imj\rvert\Bigr) + \frac{\delt}{\dely}\Bigl(\lvert\G^{n+1}_\ijp\rvert + \lvert\G^{n+1}_\ijm\rvert\Bigr)\leq\frac{h^{n}_\ij}{4}.
    \end{equation}
    In view of the above, observe  
    \begin{equation*}
        h^{n+1}_\ij - \frac{3}{4}h^n_\ij \geq h^{n+1}_\ij - h^n_\ij + \frac{\delt}{\delx}\Bigl(\lvert\F^{n+1,}_\ipj\rvert + \lvert\F^{n+1}_\imj\rvert\Bigr) + \frac{\delt}{\dely}\Bigl(\lvert\G^{n+1}_\ijp\rvert + \lvert\G^{n+1}_\ijm\rvert\Bigr) \geq 0
    \end{equation*}
    Analogously, we can also show that $h^{n+1}_\ij - \frac{5}{4}h^n_\ij \leq 0$
    and thus, we have the relation
    \begin{equation}
    \label{eqn:imp-exp-height}
        \frac{3}{4}h^n_\ij\leq h^{n+1}_\ij\leq \frac{5}{4}h^n_\ij\text{ for each }n=0,\dots,N-1
    \end{equation}
    under \eqref{eqn:exp-time-step}. Hence, if $h^n_\ij > 0$ for each $(\ij)$, we get that $h^{n+1}_\ij >0$ for each $(\ij)$. Using induction, if $h^0_\ij > 0$, then $h^k_\ij > 0$ for each $k=1,\dots, N$. 
\end{proof}

\begin{remark}
\label{rem:exp-choices}
    As a further consequence of the relation \eqref{eqn:imp-exp-height}, the implicit conditions \eqref{eqn:eta-aux-tstep-cond} required for the energy stability can now be made explicit. Indeed, we can choose 
    \[
        \eta > \frac{15}{8}h^n_\ij, \quad
        \delt^2\leq \zeta\biggl(\frac{h^{n}_\ij}{2\omega^2\eta^2}\biggl(\eta - \frac{15}{8}h^n_\ij\biggr)\biggr)
    \]
    to ensure that \eqref{eqn:eta-aux-tstep-cond} holds.
\end{remark}

We can now proceed towards proving the well-balancing property of the scheme. 
\begin{proof}[Proof of Theorem \ref{thm:wb}]
    We suppose that \eqref{eqn:disc-jet-x} holds. Under the given assumptions, the elliptic problem \eqref{eqn:ellpt-prob} reduces to 
    \begin{equation}
    \label{eqn:wb-1}
        \begin{split}
            \phi^{n+1}_\ij - g\eta\delt^2\derx(\derx\phi^{n+1})_\ij - g\eta\delt^2\dery(\dery\phi^{n+1})_\ij + \delt v^n_\ij\dery\phi^{n+1}_\ij \\
            = \phi^n_\ij - g\eta\delt^2\derx(\derx \phi^n)_\ij.
        \end{split}
    \end{equation}
    As $\dery\phi^n_\ij = 0$ under \eqref{eqn:disc-jet-x}, we can rewrite the right-hand side of \eqref{eqn:wb-1} as
    \begin{equation}
    \label{eqn:wb-2}
        \begin{split}
            \phi^{n+1}_\ij - g\eta\delt^2\derx(\derx\phi^{n+1})_\ij - g\eta\delt^2\dery(\dery\phi^{n+1})_\ij + \delt v^n_\ij\dery\phi^{n+1}_\ij \\
            = \phi^n_\ij - g\eta\delt^2\derx(\derx \phi^n)_\ij - g\eta\delt^2\dery(\dery \phi^n)_\ij +\delt v^n_\ij\dery\phi^n_\ij.
        \end{split}
    \end{equation}
    Observe that $\phi^{n+1}_\ij = \phi^n_\ij$ is a solution of \eqref{eqn:wb-2}. Since we have already shown that the above linear problem has a solution, cf. Theorem \ref{thm:exis-pos}, this ensures that $\phi^{n+1}_\ij = \phi^n_\ij$ is the unique solution of \eqref{eqn:wb-2}. Consequently, we obtain $h^{n+1}_\ij = h^n_\ij$, $\dery h^{n+1}_\ij = \dery \phi^{n+1}_\ij = \dery\phi^n_\ij = 0$ and $\derx\phi^{n+1}_\ij = \derx\phi^{n}_\ij$. Next, note that 
    \begin{align*}
        &q^{n+1}_\ij = \eta\delt(\derx\phi^{n+1}_\ij - \omega v^n_\ij) = \eta\delt(\derx\phi^n_\ij - \omega v^n_\ij) = 0,\quad r^{n+1}_\ij = \eta\delt(\dery \phi^{n+1}_\ij + \omega u^n_\ij) = 0.
    \end{align*}
    Thus, the stabilisation terms vanish under steady state conditions. Accordingly,  \eqref{eqn:mss-flux} yields
    \[
        \F^{n+1}_\ipj = \F^{n+1}_\imj = 0, \quad \G^{n+1}_\ijp = \G^{n+1}_\ijm.
    \]
    In view of the above, the $x$-momentum balance \eqref{eqn:disc-x-mom-bal} implies that $u^{n+1}_\ij = u^n_\ij = 0$. Also, the $y$-momentum balance will simplify to 
    \[
        \frac{1}{\delt}(h^{n+1}_\ij v^{n+1}_\ij - h^n_\ij v^n_\ij) + \frac{\G^{n+1}_\ijp}{\dely}(v^n_\ijp - v^n_\ijm) = 0.
    \]
    Since $v^n_\ij$ is constant in the $y$ direction, we readily obtain $v^{n+1}_\ij = v^n_\ij$. This yields that $(h^{n+1}_\ij, u^{n+1}_\ij, v^{n+1}_\ij) = (h^n_\ij, u^n_\ij, v^n_\ij)$ and consequently, 
    \[
        u^{n+1}_\ij = 0,\ \dery h^{n+1}_\ij = 0,\ \dery v^{n+1}_\ij = 0,\ \derx\phi^{n+1}_\ij - \omega v^{n+1}_\ij = 0,
    \]
    which proves the well-balancing property of the scheme.
\end{proof}

\begin{remark}
    Slight modifications to the above proof will also allow us to prove that the present scheme is capable of preserving the lake at rest steady state given by $u = v = 0$ and $\phi =$ constant.
\end{remark}

\section{Consistency and Convergence}
\label{sec:cons-conv}

As stated earlier in \eqref{eqn:h-bnd}, we assume that the water
height remains bounded from above and also from below, away from
zero. Under this assumption, note that we can choose the parameter
$\eta$ as $\eta>\frac{3}{2}\overline{h}$ in order to satisfy the
stability conditions required in Theorem \ref{thm:eng-stab}. Thus,
$\eta$ is now a constant independent of the mesh parameters.  

\subsection{A Priori Estimates}
\label{subsec:apri-est}

In view of the above, \eqref{eqn:glob-ent-ineq} along with
\eqref{eqn:h-bnd} yield the following a priori estimates: 
\begin{align}
    &\sum_{n=0}^{N-1}\big[\lVert{q^{n+1}_\T}\rVert_{L^2(\Omega)}^2 +
      \lVert r^{n+1}_\T\rVert_{L^2(\Omega)}^2\big]\lesssim
      1, \label{eqn:stab-term-est} \\ 
    &\norm{u^n_\T}_{L^2(\Omega)} + \norm{v^n_\T}_{L^2(\Omega)}\lesssim
      1, \label{eqn:vel-est} \\ 
    &\sum_{n=0}^{N-1}\sum_\ij\delx\dely(h^{n+1}_\ij -
      h^n_\ij)^2\lesssim 1. \label{eqn:time-bv} 
\end{align}

\begin{proof}[Proof of Theorem \ref{thm:cons}]
  For $\delta_\T\in (0,\delta_0]$ for a sufficiently small $\delta_0$,
  we assume that $\delt\approx \delta_\T$. 
  Let $\psi\in\Cinf([0,T)\times\Omega)$ and let $\psi^n_\ij =
  \frac{1}{\delx\dely}\iint_{K_\ij}\psi(t^n,\cdot)\dx\,\dy$ denote
  its interpolant on the mesh $\T$. We multiply the mass balance
  \eqref{eqn:disc-mss-bal-2} by $\delt\delx\dely\psi^{n+1}_\ij$ and
  sum over $n = 0$ to $n = N-1$ and over all $(\ij)$ to obtain $ T_1
  + T_2 + T_3 = 0$,  where
  \begin{align}
    &T_1 =
      \sum_{n=0}^{N-1}\delt\sum_\ij\delx\dely\biggl(\frac{h^{n+1}_\ij
      - h^n_\ij}{\delt}\biggr)\psi^{n+1}_\ij, \\ 
    &T_2 = \sum_{n =
      0}^{N-1}\delt\sum_\ij\delx\dely(\derx(h^{n+1}u^n)_\ij +
      \dery(h^{n+1}v^n)_\ij)\psi^{n+1}_\ij, \\ 
    &T_3 = - \sum_{n = 0}^{N-1}\delt\sum_\ij\delx\dely(\derx
      q^{n+1}_\ij + \dery r^{n+1}_\ij)\psi^{n+1}_\ij. 
  \end{align}

  Using summation by parts, we can rewrite $T_1$ as 
  \[
    T_1 = - \sum_{n=0}^{N-1}\delt\sum_\ij\delx\dely
    h^{n}_\ij\biggl(\frac{\psi^{n+1}_\ij -
      \psi^n_\ij}{\delt}\biggr) - \sum_\ij\delx\dely
    h^0_\ij\psi^0_\ij. 
  \]
  Now, using the duality relations \eqref{eqn:dual-rel}, we can
  rewrite $T_2$ and $T_3$ as 
  \begin{align*}
    &T_2 = -\sum_{n=0}^{N-1}\delt\sum_\ij\delx\dely\,h^{n+1}_\ij
      u^n_\ij\derx\psi^{n+1}_\ij -
      \sum_{n=0}^{N-1}\delt\sum_\ij\delx\dely\,h^{n+1}_\ij
      v^n_\ij\dery\psi^{n+1}_\ij, \\ 
    &T_3 =
      \sum_{n=0}^{N-1}\delt\sum_\ij\delx\dely\,q^{n+1}_\ij\derx\psi^{n+1}_\ij
      +
      \sum_{n=0}^{N-1}\delt\sum_\ij\delx\dely\,r^{n+1}_\ij\dery\psi^{n+1}_\ij. 
    \end{align*}
    In view of the above, we get the following simplified expression:
    \begin{equation}
      -\iint\limits_{\Omega} h^0_\T\psi(0,\cdot)\,\dx\,\dy = \sum_{n =
        0}^{N-1}\int\limits_{t^{n}}^{t^{n+1}}\!\!\!\!\iint\limits_{\Omega}\bigl\lbrack
      h^n_\T\Dt\psi + h^{n+1}_\T u^n_\T\Dx\psi + h^{n+1}_\T
      v^n_\T\Dy\psi\bigr\rbrack\,\dx\,\dy\,\dt +
      \mathcal{C}^{mass}_{\T,\delt}, 
    \end{equation}
    where the mass consistency error, $\mathcal{C}^{mass}_{\T,\delt}$, is given by 
    \begin{gather}
    \label{eqn:mss-cons-err}
            \mathcal{C}^{mass}_{\T,\delt} = -T_3 + \iint\limits_{\Omega} h^0_\T(\psi^0_\T - \psi(0,\cdot))\,\dx\,\dy + \sum_{n = 0}^{N-1}\int\limits_{t^{n}}^{t^{n+1}}\!\!\!\!\iint\limits_{\Omega} h^{n}_\T(\eth_t\psi^{n+1}_\T - \Dt\psi)\,\dx\,\dy\,\dt  \\
            +\sum_{n = 0}^{N-1}\int\limits_{t^{n}}^{t^{n+1}}\!\!\!\!\iint\limits_{\Omega} h^{n+1}_\T u^n_\T(\derx\psi^{n+1}_\T - \Dx\psi)\,\dx\,\dy\,\dt + \sum_{n = 0}^{N-1}\int\limits_{t^{n}}^{t^{n+1}}\!\!\!\!\iint\limits_{\Omega} h^{n+1}_\T v^n_\T(\dery\psi^{n+1}_\T - \Dy\psi)\,\dx\,\dy\,\dt \nonumber \\
            =: -T_3 + R_1 + R_2 + R_3 + R_4. \nonumber
    \end{gather}
    Using the a priori estimates, cf. Subsection
    \ref{subsec:apri-est}, the assumption \eqref{eqn:h-bnd} on the
    water height and standard interpolation inequalities,
    cf. \cite{FLM+21a}, we can obtain the following estimates on the
    remainder terms $R_1,\dots, R_4$.   
    \begin{gather*}
      \abs{R_1}\lesssim \delta_\T\norm{\psi}_{C^1},\quad
      \abs{R_2}\lesssim \delt\norm{\psi}_{C^2}, \quad 
      \abs{R_3}\lesssim\delta_\T\norm{\psi}_{C^2}, \quad
      \abs{R_3}\lesssim\delta_\T\norm{\psi}_{C^2}.
    \end{gather*}
    We now estimate $T_3$. To this end,  we use the Holder's
    inequality along with the estimate \eqref{eqn:stab-term-est} and
    obtain 
    \[
        \abs{T_3}\lesssim \sqrt{\delt}\norm{\psi}_{C^2}.
    \]
    This ensures that $T_3\to 0$ as we refine the mesh and hence $\mathcal{C}^{mass}_{\T,\delt}\to 0$ upon passing to the limit.

    Next, we move onto formulating the consistency of the momentum
    balance. Let $\Psi\in\Cinf(\lbrack0,T)\times\Omega)$. We multiply
    \eqref{eqn:disc-x-mom-bal} with $\delt\delx\dely\Psi^{n+1}_\ij$,
    substitute for the fluxes and sum over $n = 0$ to $n = N-1$ and
    all $\ij$ to obtain the relation 
    \begin{equation}
      \label{eqn:xmom-cons-S}
        S_1 + S_2 + S_3 + S_4 + S_5 + S_6 = S_7 + S_8 + S_9,
    \end{equation}
    where 
    \begin{align*}
        &S_1 = \sum_{n = 0}^{N-1}\delt\sum_\ij\delx\dely\biggl(\frac{h^{n+1}_\ij u^{n+1}_\ij - h^n_\ij u^n_\ij}{\delt}\biggr)\Psi^{n+1}_\ij, \\
        &S_2 = \sum_{n = 0}^{N-1}\delt\sum_\ij\delx\dely\biggl(\frac{ \ldblbrace h^{n+1} u^n\rdblbrace_{\ipj}u^n_\ipj - \ldblbrace h^{n+1} u^n\rdblbrace_{\imj} u^n_\imj}{\delx}\biggr)\Psi^{n+1}_\ij, \\
        &S_4 = -\sum_{n = 0}^{N-1}\delt\sum_\ij\delx\dely\biggl(\frac{ \ldblbrace q^{n+1}\rdblbrace_{\ipj}u^n_\ipj - \ldblbrace q^{n+1}\rdblbrace_{\imj} u^n_\imj}{\delx}\biggr)\Psi^{n+1}_\ij, \\
        &S_6 = \sum_{n = 0}^{N-1}\delt\sum_\ij\delx\dely\,gh^{n+1}_\ij\derx h^{n+1}_\ij\Psi^{n+1}_\ij, \quad S_7 = -\sum_{n = 0}^{N-1}\delt\sum_\ij\delx\dely\,gh^{n+1}_\ij\derx b_\ij\Psi^{n+1}_\ij, \\
        &S_8 = \sum_{n = 0}^{N-1}\delt\sum_\ij\delx\dely\,\omega\,h^{n+1}_\ij v^n_\ij \Psi^{n+1}_\ij, \quad S_9 = - \sum_{n=0}^{N-1}\delt\sum_\ij\delx\dely\,\omega\,r^{n+1}_\ij\Psi^{n+1}_\ij. 
    \end{align*}
    In \eqref{eqn:xmom-cons-S}, $S_3$ and $S_5$ are terms that arise
    from the fluxes in the $y$-direction, and are analogous to the
    terms $S_2$ and $S_4$ respectively.   
    As in the case of the mass balance, we can simply rewrite $S_1$ as 
    \[
        S_1 = - \sum_{n=0}^{N-1}\delt\sum_\ij\delx\dely\,h^n_\ij u^n_\ij\biggl(\frac{\Psi^{n+1}_\ij - \Psi^n_\ij}{\delt}\biggr) - \sum_\ij\delx\dely\,h^0_\ij u^0_\ij\Psi^0_\ij.
    \] 
    Re-indexing the summations will allow us to rewrite $S_2$ and $S_4$ as 
    \begin{align*}
        &S_2 = -\sum_{n=0}^{N-1}\delt\sum_\ij\delx\dely\ldblbrace h^{n+1} u^n\rdblbrace_{\ipj} u^n_\ipj\derx^\E\Psi^{n+1}_\ipj, \\
        &S_4 =  \sum_{n=0}^{N-1}\delt\sum_\ij\delx\dely\ldblbrace q^{n+1}\rdblbrace_{\ipj}u^n_\ipj\derx^\E\Psi^{n+1}_\ipj,
    \end{align*}
    and analogously, for $S_3$ and $S_5$.  We rewrite $S_6$ as
    \begin{gather*}
        S_6 = \sum_{n=0}^{N-1}\delt\sum_{\ij}\delx\dely\derx\biggl(\half g(h^{n+1})^2\biggr)_\ij\Psi^{n+1}_\ij \\
        - \sum_{n = 0}^{N-1}\delt\sum_\ij\delx\dely\biggl(\frac{g(h^{n+1}_{i+1,j} - h^{n+1}_\ij)^2}{4\delx} - \frac{g(h^{n+1}_{i-1,j} - h^{n+1}_\ij)^2}{4\delx} \biggr)\Psi^{n+1}_\ij.
    \end{gather*}
    Denoting $p^{n+1}_\ij = \half g (h^{n+1}_\ij)^2$, we use the
    discrete grad-div duality \eqref{eqn:dual-rel} in the above to simplify the
    first term, and rewrite the second term after re-indexing the
    summation to obtain 
    \[
        S_6 = -\sum_{n=0}^{N-1}\delt\sum_{\ij}\delx\dely p^{n+1}_\ij\derx\Psi^{n+1}_\ij + \sum_{n=0}^{N-1}\delt\sum_{\ij}\delx\dely \frac{(h^{n+1}_{i+1, j} - h^{n+1}_\ij)^2}{4}\derx^\E\Psi^{n+1}_\ipj.
    \]
    Thus, we obtain the following simplified expression:
    \begin{gather}
        -\iint\limits_{\Omega}h^0_\T u^0_\T \Psi(0,\cdot)\,\dx\,\dy = \sum_{n = 0}^{N-1}\int\limits_{t^{n}}^{t^{n+1}}\!\!\!\!\iint\limits_{\Omega} \lbrack h^n_\T u^n_\T\Dt\Psi + m^{n+1}_{x,\D^{x}} u_{\D^{x}}^{n}\Dx\Psi + m^{n+1}_{y, \D^{y}} u^n_{\D^{y}}\Dy\Psi \\
        + p^{n+1}_\T\Dx\Psi\rbrack\,\dx\,\dy\,\dt + \sum_{n = 0}^{N-1}\int\limits_{t^{n}}^{t^{n+1}}\!\!\!\!\iint\limits_{\Omega}\lbrack-gh^{n+1}_\T\Dx b\,\Psi + \omega h^{n+1}_\T v^n_\T\Psi\rbrack\,\dx\,\dy\,\dt + \mathcal{C}^{mom,x}_{\T,\delt} = 0, \nonumber
    \end{gather}
    where $m^{n+1}_{x,\D^x}$ (resp.\ $m^{n+1}_{y,\D^y}$) denotes the reconstruction of the $x$-component (resp. $y$-component) of the momentum on the dual grid $\D^x$ (resp.\ $\D^y$) and analogously for $u^n_{\D^x}$ and $u^n_{\D^y}$. The reconstructions are defined as 
    \begin{gather*}
        (m^{n+1}_{x, \D^x})_\ipj = \ldblbrace h^{n+1} u^n\rdblbrace_\ipj,\quad (m^{n+1}_{y, \D^y})_\ijp = \ldblbrace h^{n+1} v^n\rdblbrace_\ijp, \\
        (u^n_{\D^x})_{\ipj} = u^n_\ipj, \quad (u^n_{\D^y})_{\ijp} = u^n_\ijp. 
    \end{gather*}
    The consistency error $\mathcal{C}^{mom,x}_{\T,\delt}$ is defined as
    \begin{gather}
    \label{eqn:mom-cons-err}
        \mathcal{C}^{mom,x}_{\T,\delt} = S_4 + S_5 + S_9 + \iint\limits_\Omega h^0_\T u^0_\T(\Psi^0_\T - \Psi(0,\cdot))\,\dx\,\dy + \sum_{n = 0}^{N-1}\int\limits_{t^{n}}^{t^{n+1}}\!\!\!\!\iint\limits_{\Omega} h^n_\T u^n_\T(\eth_t\Psi^{n+1}_\T - \Dt\Psi)\,\dx\,\dy\,\dt \nonumber \\
        + \sum_{n = 0}^{N-1}\int\limits_{t^{n}}^{t^{n+1}}\!\!\!\!\iint\limits_{\Omega}\lbrack m^{n+1}_{x,\D^{x}}u^n_{\D^x}(\derx^\E\Psi^{n+1}  -\Dx\Psi) + m^{n+1}_{y,\D^{y}}u^n_{\D^y}(\dery^\E\Psi^{n+1} - \Dy\Psi)\rbrack\,\dx\,\dy\,\dt \nonumber \\
        + \sum_{n = 0}^{N-1}\int\limits_{t^{n}}^{t^{n+1}}\!\!\!\!\iint\limits_{\Omega} p^{n+1}_\T(\derx\Psi^{n+1}_\T - \Dx\Psi)\,\dx\,\dy\,\dt -  \sum_{n = 0}^{N-1}\int\limits_{t^{n}}^{t^{n+1}}\!\!\!\!\iint\limits_{\Omega} g h^{n+1}_\T(\Psi^{n+1}_\T \derx b_\T  - \Psi\Dx b)\,\dx\,\dy\,\dt \\
    +\sum_{n = 0}^{N-1}\int\limits_{t^{n}}^{t^{n+1}}\!\!\!\!\iint\limits_{\Omega} \omega h^{n+1}_\T u^n_\T(\Psi^{n+1}_\T - \Psi)\,\dx\,\dy\,\dt -\sum_{n=0}^{N-1}\delt\sum_{\ij}\delx\dely \frac{(h^{n+1}_{i+1, j} - h^{n+1}_\ij)^2}{4}\derx^\E\Psi^{n+1}_\ipj \nonumber \\
    =: S_4 + S_5 + S_6 + Q_1 + Q_2 + Q_3 + Q_4 + Q_5 + Q_6 + Q_7. \nonumber
    \end{gather}
    The terms $Q_1,\dots,Q_6$ can all be estimated using the a priori estimates obtained in Subsection \ref{subsec:apri-est}, the assumption on the water height \eqref{eqn:h-bnd}, the stability of the reconstruction operator \eqref{eqn:recon-stab} and standard interpolation inequalities as follows:
    \begin{align*}
        &\abs{Q_1}\lesssim\delta_\T\norm{\Psi}_{C^1}, \quad \abs{Q_2}\lesssim\delt\norm{\Psi}_{C^2}, \quad \abs{Q_3}\lesssim\delta_\T\norm{\Psi}_{C^2}, \\
        &\abs{Q_4}\lesssim\delta_\T\norm{\Psi}_{C^2}, \quad \abs{Q_5}\lesssim\delta_\T\norm{b}_{W^{2,\infty}}\norm{\Psi}_{C^1}, \quad\abs{Q_6}\lesssim \delta_\T\norm{\Psi}_{C^1}.
    \end{align*}
    Now, to estimate the terms $S_4$ and $S_5$, we utilize \eqref{eqn:recon-stab} along with Holder's inequality and the a priori estimates from Subsection \ref{subsec:apri-est} to obtain
    \[
        \abs{S_4 + S_5}\lesssim\sqrt{\delt}\norm{\Psi}_{C^2}.
    \]
    Using \eqref{eqn:stab-term-est}, we estimate $S_9$ as
    \[
        \abs{S_9}\lesssim\sqrt{\delt}\norm{\Psi}_{L^\infty}.
    \]
    In order to estimate the term $Q_7$, we use \eqref{eqn:wk-bv-assmp} and observe that
    \[
        \sum_{n=0}^{N-1}\delt\sum_{\ij}\delx\dely (h^{n+1}_{i+1,j} - h^{n+1}_\ij)^2\lesssim \delx\lesssim\delta_\T.
    \]
   From the above, we readily obtain
    \[
        \abs{Q_7}\lesssim\delta_\T\norm{\Psi}_{C^2}.
    \]
    This then yields $\mathcal{C}^{mom,x}_{\T,\delt}\to 0$ as we refine the mesh.

   Analogously, for $\Phi\in\Cinf(\lbrack 0, T)\times\Omega)$, we obtain the consistency for the $y$-momentum balance which is given by 
    \begin{gather}
        -\iint\limits_{\Omega}h^0_\T v^0_\T \Phi(0,\cdot)\,\dx\,\dy = \sum_{n = 0}^{N-1}\int\limits_{t^{n}}^{t^{n+1}}\!\!\!\!\iint\limits_{\Omega} \lbrack h^n_\T v^n_\T\Dt\Phi + m^{n+1}_{x,\D^{x}} v_{\D^{x}}^{n}\Dx\Phi + m^{n+1}_{y, \D^{y}} v^n_{\D^{y}}\Dy\Phi \\
        + p^{n+1}_\T\Dy\Phi\rbrack\,\dx\,\dy\,\dt + \sum_{n = 0}^{N-1}\int\limits_{t^{n}}^{t^{n+1}}\!\!\!\!\iint\limits_{\Omega}\lbrack-gh^{n+1}_\T\Dy b\,\Phi - \omega h^{n+1}_\T u^n_\T\Phi\rbrack\,\dx\,\dy\,\dt + \mathcal{C}^{mom,y}_{\T,\delt} = 0, \nonumber
    \end{gather}
    where the velocity reconstructions are defined as 
    \begin{gather*}
        (v^n_{\D^x})_{\ipj} = v^n_\ipj, \quad (v^n_{\D^y})_{\ijp} = v^n_\ijp,
    \end{gather*}
    and the consistency error $\mathcal{C}^{mom,y}_{\T,\delt}\to 0$ as the mesh parameters vanish.

\end{proof}

Finally, we proceed towards the proof of the convergence of the scheme.

\begin{proof}[Proof of Theorem \ref{thm:conv}]

  For a sequence of space-time discretizations $(\T^\k,\delt^\k)$
  such that $\delta_{\T^\k}\to 0$ as $k\to\infty$, let $h^\k, \,
  m_x^\k = h^\k u^\k$ and $ m_y^\k = h^\k v^\k$ denote the sequence
  of discrete numerical solutions, $E^\k$ denote the corresponding
  sequence of the total energies and let $b^\k = b_{\T^\k}$. As a
  consequence of the density bound \eqref{eqn:h-bnd} and the the
  global energy estimate \eqref{eqn:glob-ent-ineq}, we can recover
  the following uniform bounds: 
  \begin{subequations}
    \label{eqn:conv-est}
        \begin{gather}
            h^\k,\, \half g(h^\k)^2, gb^\k h^\k \in L^\infty(\odom), \label{eqn:h-conv-est} \\
            m_x^\k,\,m_y^\k\in L^\infty(0,T;L^2(\Omega)), \label{eqn:mom-conv-est} \\
            \frac{(m_x^\k)^2}{h^\k},\, \frac{(m_y^\k)^2}{h^\k},\,\frac{m_x^\k m_y^\k}{h^\k}\in L^\infty(0,T;L^1(\Omega)), \label{eqn:mom-flx-conv-est}\\
            E^\k\in L^\infty(0,T;L^1(\Omega)). \label{eqn:eng-conv- est}
        \end{gather}
    \end{subequations}

    Consequently, invoking the fundamental theorem of Young measures, cf.\ \cite{FLM+21a,Ped97}, the sequence
    $(h^\k, m_x^\k, m_y^\k)_{k}$ will generate (passing to a subsequence if needed) a Young measure \\
    $\mathcal{V} = \lbrace\V\rbrace_{(t,x,y)\in\odom}\in L^\infty_{weak-*}(\odom,\Pro(\F))$ such that
    \begin{subequations}
    \label{eqn:main-conv}
        \begin{gather}
            h^\k\weakstar\langle\V;\th\rangle\text{ in }L^\infty(\odom), \\
            m^\k_x\weakstar\langle\V;\tmx\rangle\text{ and } m^\k_y\weakstar\langle\V;\tmy\rangle\text{ in }L^\infty(0,T;L^2(\Omega)),
        \end{gather}
    \end{subequations}
    as $k\to\infty$.

    Now, given the bounds \eqref{eqn:h-conv-est}, we can deduce that the sequences $\bigl(\half g (h^\k)^2\bigr)_{k\in\N}$ and $(gb^\k h^\k)_{k\in\N}$ have weakly-* convergent subsequences in $L^\infty((0,T)\times\Omega)$. Consequently, we get
        \begin{gather}
            \half g(h^\k)^2\weakstar\biggl\langle\V;\half g\th^2\biggr\rangle,\, gb^\k h^\k\weakstar\langle\V;gb\th\rangle\text{ in }L^\infty(\odom), \label{eqn:h-con}
        \end{gather}
    Next, the bounds given in \eqref{eqn:mom-flx-conv-est} are not strong enough to guarantee the existence of a weak limit in $L^\infty(0,T; L^1(\Omega))$ for the sequence of functions $(\frac{(m_x^\k)^2}{h^\k})_{k\in\N},\, (\frac{(m_y^\k)^2}{h^\k})_{k\in\N},\,(\frac{m_x^\k m_y^\k}{h^\k})_{k\in\N}$ . However, we can still deduce that the aforementioned sequences will have weakly-* convergent subsequences in $L^\infty(0,T; \M(\Omega))$ due to the embedding $L^1(\Omega)\hookrightarrow \M(\Omega)$. Hence, we obtain 
    \begin{gather}
    \label{eqn:non-lin-conv}
        \frac{(m_x^\k)^2}{h^k}\weakstar \frac{\overline{m_x^2}}{h},\, \frac{(m_y^\k)^2}{h^k}\weakstar \frac{\overline{m_y^2}}{h},\, \frac{m_x^\k m_y^k}{h^k}\weakstar \frac{\overline{m_x m_y}}{h}\ \text{in}\ L^\infty(0,T;\M(\Omega)).
    \end{gather}
    Similarly, for the sequence of total energies $(E^\k)_{k\in\N}$, we obtain
    \begin{equation}
    \label{eqn:enrg-conv}
        E^\k \weakstar\half\overline{\frac{m_x^2}{h}} + \half\overline{\frac{m_y^2}{h}} + \biggl\langle\V;\half g\th^2 + gb\th\biggr\rangle\text{ in }L^\infty(0,T;\M(\overline{\Omega})).
    \end{equation}
    as a consequence of \eqref{eqn:h-con} and \eqref{eqn:non-lin-conv}.
    Next, we introduce the following defect measures
    \begin{subequations}
    \label{eqn:def-meas}
        \begin{gather}
            \mathfrak{E}_{cd} = \half\overline{\frac{m_x^2}{h}} + \half\overline{\frac{m_y^2}{h}} - \biggl\langle\V;\half\frac{\tmx^2}{h} + \half\frac{\tmy^2}{h}\biggr\rangle, \\
            \mathfrak{R}^x_{cd} = \overline{\frac{m_x^2}{h}} - \biggl\langle\V;\frac{\tmx^2}{h}\biggr\rangle,\, \mathfrak{R}^y_{cd} = \overline{\frac{m_y^2}{h}} - \biggl\langle\V;\frac{\tmy^2}{h}\biggr\rangle,\, \mathfrak{R}^{x,y}_{cd} = \overline{\frac{m_x m_y}{h}} - \biggl\langle\V;\frac{\tmx\tmy}{h}\biggr\rangle
        \end{gather}
    \end{subequations}
    and we set $\mathfrak{R}_{cd} = \begin{bmatrix} \mathfrak{R}_{cd}^x & \mathfrak{R}^{x,y}_{cd} \\ \mathfrak{R}^{x,y}_{cd} & \mathfrak{R}^y_{cd} \end{bmatrix}$. Now, the non-negativity of the measure $\mathfrak{E}_{cd}$ is a direct consequence of \cite[Corollary 5.2]{FLM+21a}.  
    Thus, we can rewrite the convergences presented in \eqref{eqn:non-lin-conv}-\eqref{eqn:enrg-conv} as
    \begin{gather*}
        \frac{(m_x^\k)^2}{h}\weakstar\biggl\langle\V;\frac{\tmx^2}{\th}\biggr\rangle + \mathfrak{R}^x_{cd},\quad \frac{(m_y^\k)^2}{h}\weakstar\biggl\langle\V;\frac{\tmy^2}{\th}\biggr\rangle + \mathfrak{R}^y_{cd} \\
        \frac{m_x^\k m_y^\k}{h}\weakstar\biggl\langle\V;\frac{\tmx\tmy}{\th}\biggr\rangle + \mathfrak{R}^{x,y}_{cd}, \quad E^\k\weakstar\biggl\langle\half\frac{\tmx^2}{h} + \half\frac{\tmy^2}{h} + \half g\th^2 + gb\th\biggr\rangle + \mathfrak{E}_{cd}.
    \end{gather*}
    Accordingly, we use \eqref{eqn:conv-est}-\eqref{eqn:def-meas} along with Lemma \ref{lem:mesh-recon-conv} to pass to the limit $k\to\infty$ in the consistency formulation \eqref{eqn:cons-mss-bal}, \eqref{eqn:cons-xmom-bal} and \eqref{eqn:cons-ymom-bal} to see that $\mathcal{V}$ and $\mathfrak{R}_{cd}$ will satisfy \eqref{eqn:dmv-mass}-\eqref{eqn:dmv-y-mom}. Also, from \eqref{eqn:glob-ent-ineq}, it follows that 
    \[
        \iint_\Omega E^\k(t,\cdot)\,\dx\,\dy \leq \iint_\Omega E^\k(0,\cdot)\,\dx\,\dy
    \]
    for a.e.\ $t\in (0,T)$. Passing to the limit above and using \eqref{eqn:enrg-conv} will yield \eqref{eqn:eng-ineq-dmv}. Also, observe that $\mathrm{tr}(\mathfrak{R}_{cd}) = 2\mathfrak{E}_{cd}$ which yields \eqref{eqn:def-comp-dmv}, and this completes the proof.
        
\end{proof}

\section{Numerical Results}
\label{sec:num-res}

In this section, we present numerical results for the quasi
one-dimensional (1D) and 2D systems. The 1D RSW system is given by:    
\begin{subequations}
\label{eqn:1d-rsw}
  \begin{gather}
    \Dt h + \Dx(hu) = 0, \\
    \Dt(hu) + \Dx(hu^2) + h\Dx\phi = \omega hv, \\
    \Dt(hv) + \Dx(huv) = -\omega hu. 
  \end{gather}
\end{subequations}
Throughout, we consider a uniform grid with either $k$-cells for a 1D
problem, or a $k\times k$ mesh for a 2D problem.   

\subsection{Well-Balancing Test Cases}
This section is devoted to demonstrating the well-balancing property
of the scheme. We consider four test problems. The first corresponds
to the classical lake-at-rest steady state, while the remaining three
involve steady states of the jet in a rotating frame
(cf. \eqref{eqn:jet-x}–\eqref{eqn:jet-y}), with progressively more
complex configurations that include both flat and non-flat bottom
topographies. 

For all four cases, the computations are carried out up to a final
time of $T = 10$. The $L^2$-errors of the potential $\phi$ and the
horizontal velocities $u$ and $v$ are reported in Table
\ref{tab:wb-err}. As expected, the proposed scheme successfully
preserves the steady states to high accuracy. 

\subsubsection*{Test 1 (\cite{BAL+14}):} The computational domain is
$\Omega = [0, 10]^2$ with $\omega = g = 1$ and we consider a $100\times 100$ 
grid. The initial data is as follows: 
\begin{gather}
  h(0,x,y) + b(x,y) = 10,\quad b(x,y) = 5 \exp(-0.4 ((x-5)^2 + (y-5)^2)), \\
  u(0,x,y) = v(0,x,y) = 0.
\end{gather}

\subsubsection*{Test 2 (\cite{CLP08}):} The computational domain is
$\Omega = \lbrack -0.5, 0.5\rbrack$ with $\omega = 2$, $g = 9.81$ and
we consider a grid with $1000$ cells. We consider equilibrium boundary
conditions and the initial data is as follows: 
\begin{gather}
    h(0, x) = \frac{4}{g} +\biggl(\frac{\omega}{g}\biggr)x,\quad
    b\equiv 0, \quad u(0,x) = 0,\quad v(0,x) = 1. 
\end{gather}

\subsubsection*{Test 3 (\cite{CDK+18}):} The computational domain is
$\Omega = \lbrack -5, 5\rbrack$ with $\omega = 10$, $g = 1$ and we
consider a grid with $1000$ cells. We consider equilibrium boundary
conditions and the initial data is as follows: 
\begin{gather}
  h(0, x) = \frac{2}{g} - \exp(-x^2),\quad b\equiv 0,\quad u(0,x) =
  0,\quad v(0,x) = \frac{2g}{\omega}x\exp(-x^2). 
\end{gather}

\subsubsection*{Test 4 (\cite{CDK+18}):} The computational domain is
$\Omega = \lbrack -5, 5\rbrack$ with $\omega = 1$, $g = 1$ and we
consider a grid with $1000$ cells. We consider periodic boundary
conditions and the initial data is as follows: 
\begin{gather}
    h(0, x) = 1,\quad b(x) =
    \frac{\omega}{g}\sin\biggl(\frac{\pi}{5}x\biggr),\quad u(0,x) =
    0,\quad v(0,x) = \frac{\pi}{5}\cos\biggl(\frac{\pi}{5}x\biggr). 
\end{gather}

\begin{table}[htpb]
    \centering
    \begin{tabular}{c|c|c|c}
         & $L^2$-error in $\phi$ & $L^2$-error in $u$ & $L^2$-error in $v$ \\
    \hline
    Test 1 & $1.76\times10^{-16}$ & $1.44\times 10^{-16}$ & $2.08\times 10^{-16}$ \\
    \hline
    Test 2 & $3.61\times 10^{-15}$ & $3.21\times 10^{-14}$ & $2.12\times 10^{-15}$\\
    \hline
    Test 3 & $2.76\times 10^{-15}$& $1.01\times 10^{-13}$ & $1.15\times 10^{-14}$ \\
    \hline
    Test 4 & $2.35\times 10^{-15}$& $1.12\times 10^{-13}$ & $2.21\times 10^{-14}$\\
    \end{tabular}
    \caption{Results of the well-balancing case studies}
    \label{tab:wb-err}
\end{table}

\subsection{Rossby Adjustment on an Open Domain (\cite{BLZ04,
    CDK+18})} 

We now consider the Rossby adjustment problem, which consists of
superimposing a perturbation in the $y$-velocity onto a stationary
state. Subsequently, shocks form on the right-moving and left-moving
fronts. The computational domain is $\Omega = [-8,12],$ the
topography is flat ($b \equiv 0$), and we set $\omega = g =
1$. Extrapolation boundary conditions are imposed on both sides of the
domain. The initial data are as follows: 
\begin{equation}
  h(0,x) = 1, \quad u(0,x) = 0, \quad v(0,x) = \frac{2(1 +
    \tanh(2x+2))(1-\tanh(2x-2))}{(1+\tanh(2))^2}. 
\end{equation}
We compute the solution for various times on a grid with $1000$ cells
and present the profiles of the water height $h$ in Figure
\ref{fig:ra-dn}.  

Upon computation, we observe pre- and post-shock oscillations (shown
in blue) in the water height $h$, arising from the non-conservative
discretization of the pressure gradient; see also \cite{PV16} for
analogous considerations. To suppress these oscillations, we introduce
additional stabilization terms proportional to $\partial^2_{xx}(hu)$
and $\partial^2_{yy}(hv)$ in the $x$- and $y$-momentum balances,
respectively, following the approach of \cite{ADG+25}. Accordingly, in
the $x$-momentum balance, the pressure term is modified as follows: we
replace $h^{n+1}_\ij\derx\phi^{n+1}_\ij$ by
$h^{n+1}_\ij(\derx\phi^{n+1}_\ij - \derx^\T\Lambda^{n+1}_\ij)$, where   
\begin{gather}
  \derx^\T\Lambda^{n+1}_\ij = \frac{\Lambda^{n+1}_{\ipj} -
    \Lambda^{n+1}_\imj}{\delx},\quad \Lambda^{n+1}_{\ipj} =
  \frac{g\alpha\delt}{\beta_\T}\biggl(\frac{h^{n+1}_{i+1,j} u^n_{i+1,
      j} - h^{n+1}_\ij u^n_\ij}{\delx}\biggr). 
\end{gather}
Similarly, the term $h^{n+1}_\ij\dery\phi^{n+1}_\ij$ in the
$y$-momentum balance is replaced by $h^{n+1}_\ij(\dery\phi^{n+1}_\ij -
\dery^\T\Theta^{n+1}_\ij)$ where 
\begin{gather}
  \dery^\T\Theta^{n+1}_\ij = \frac{\Theta^{n+1}_{\ijp} -
    \Theta^{n+1}_\ijm}{\dely},\quad \Theta^{n+1}_{\ijp} =
  \frac{g\alpha\delt}{\beta_\T}\biggl(\frac{h^{n+1}_{i,j+1} v^n_{i,
      j+1} - h^{n+1}_\ij v^n_\ij}{\dely}\biggr). 
\end{gather}
Here, $\beta_\T = \biggl(\dfrac{2}{\delx} +
\dfrac{2}{\dely}\biggr)^{-1}$ and $\alpha > 0$.  

In our computations, we set $\alpha = 1$. The added diffusion
effectively damps out the oscillations, as illustrated by the red
curve in the figure. 
\begin{figure}[htpb]
  \centering
  \includegraphics[height=0.25\textheight]{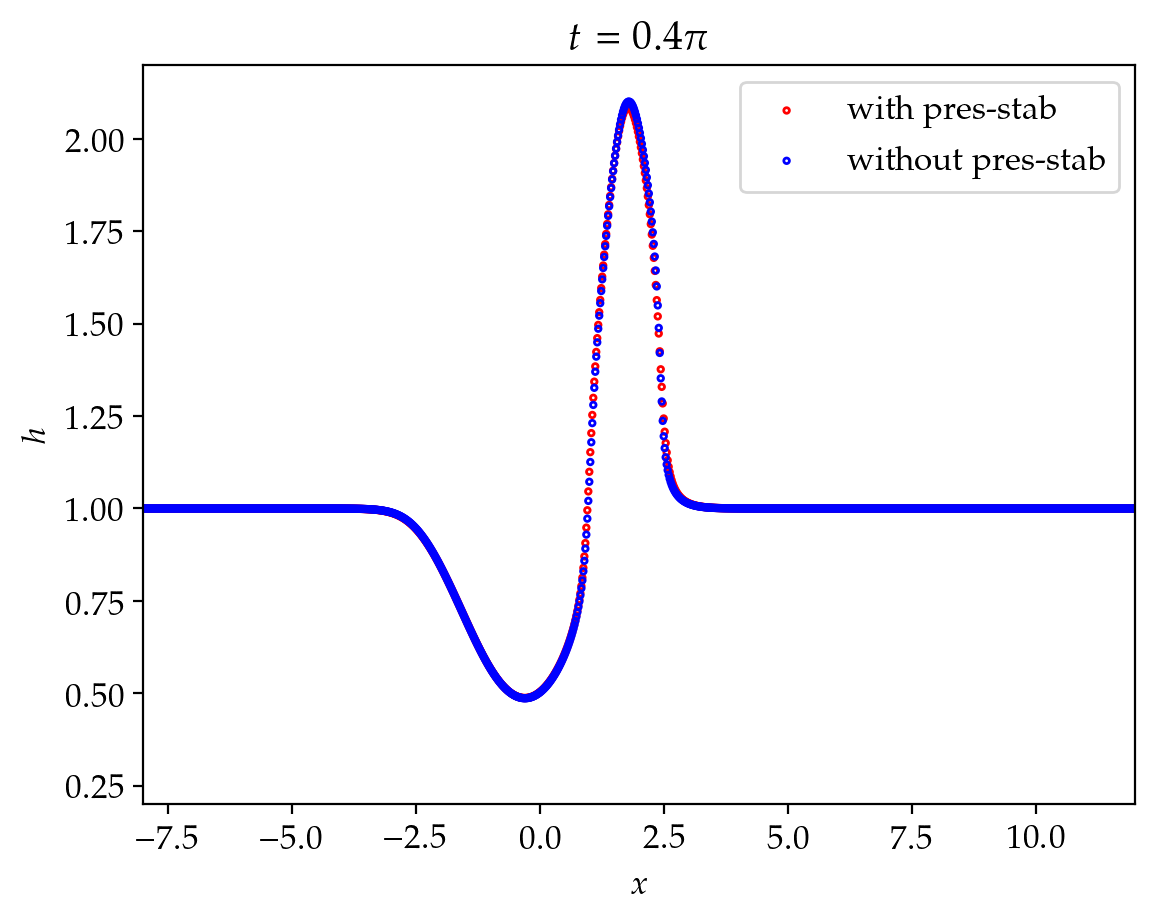}
  \includegraphics[height=0.25\textheight]{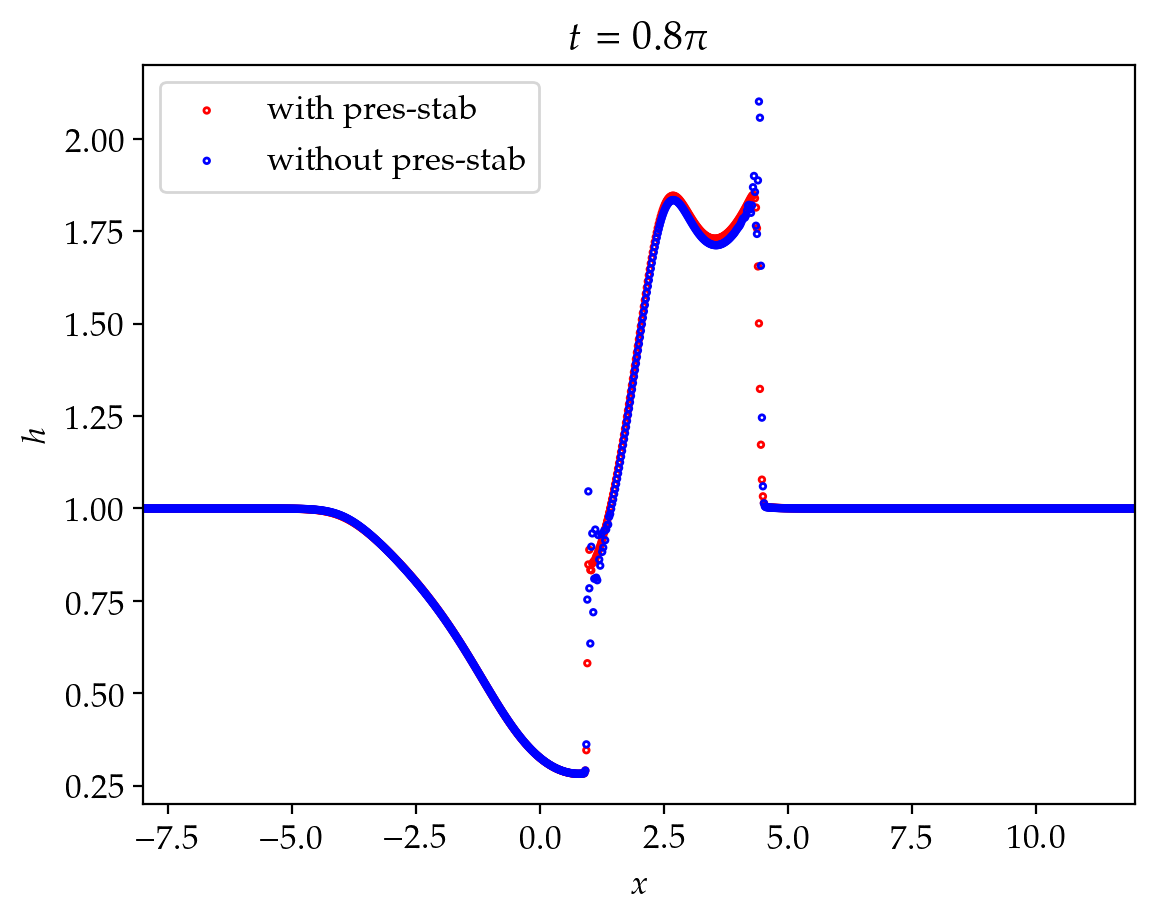}
  \includegraphics[height=0.25\textheight]{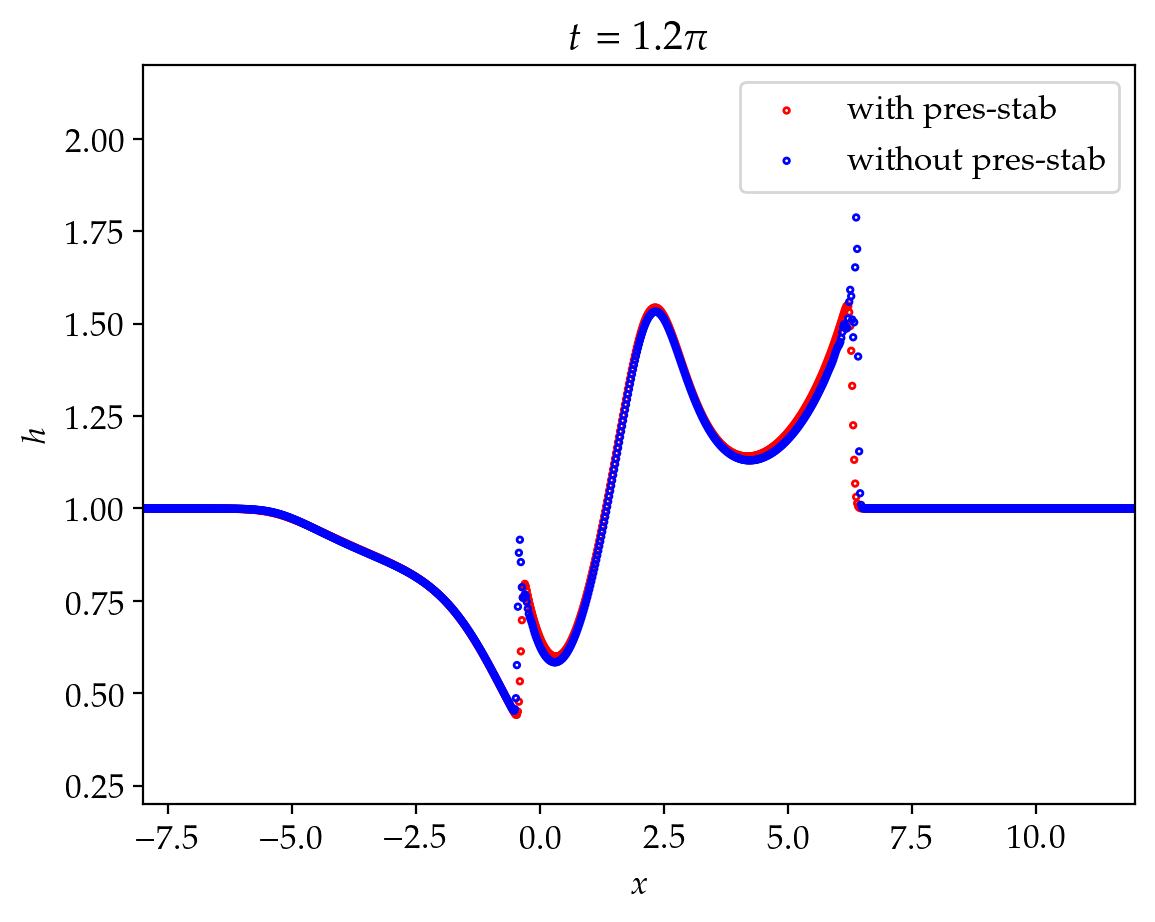}
  \includegraphics[height=0.25\textheight]{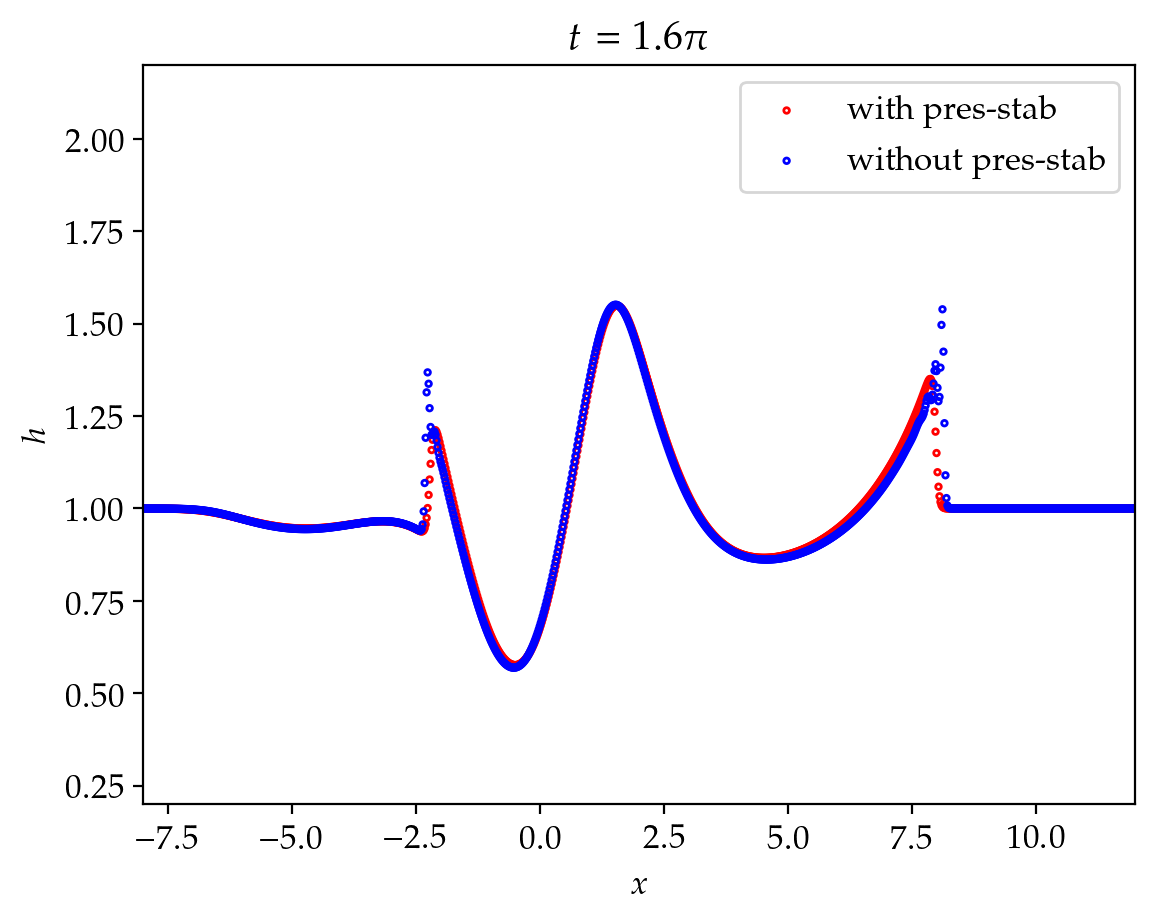}
  \includegraphics[height=0.25\textheight]{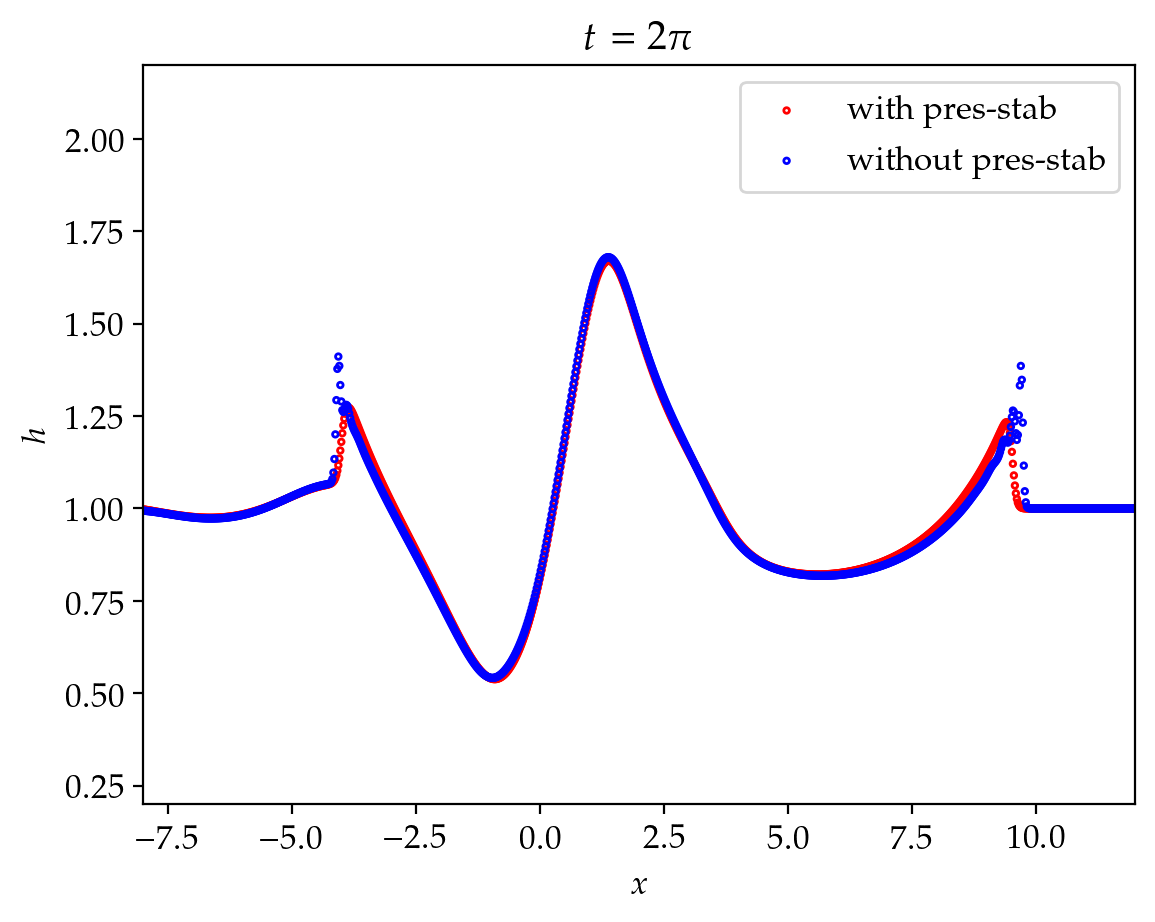}
  \caption{The water height $h$ at various times of the Rossby adjustment problem.}
  \label{fig:ra-dn}
\end{figure}

\begin{remark}
  We emphasize that the stabilization terms added to the scheme do not
  compromise its well-balancing property. This is because the
  additional stabilization appears only in the $x$- and $y$-momentum
  balances, leaving the elliptic problem \eqref{eqn:ellpt-prob}
  unaffected. Moreover, under the discrete analogues of
  \eqref{eqn:jet-x} or \eqref{eqn:jet-y}, we observe that
  $\Lambda^{n+1}_\ipj = \Theta^{n+1}_\ijp = 0$ in both cases, which
  leads to the same conclusions as in Theorem \ref{thm:wb}. 
\end{remark}

\subsection{Stationary State in Space (\cite{DM22})}

The computational domain is $\Omega = [0,1]$, with a flat bottom ($b
\equiv 0$) and parameters $\omega = g = 1$. An exact spatially
constant solution of the RSW system is given by 
\begin{gather}
  h(t,x) = h_0, \quad u(t,x) = u_0\cos(\omega t) + v_0\sin(\omega
  t),\quad v(t,x) = v_0\cos(\omega t) - u_0\sin(\omega t). 
\end{gather}
We consider a final time of $T = 1$ with periodic boundary conditions
and the above solution as initial data with $h_0 = u_0 = v_0 = 1$. The
computations are carried out on a grid with 1000 cells, and Figure
\ref{fig:des-ex-soln} presents a comparison between the exact and
computed momenta over time. The numerical solutions show excellent
agreement with the exact solution, with the corresponding curves
nearly indistinguishable.  
\begin{figure}
  \centering
  \includegraphics[height = 0.25\textheight]{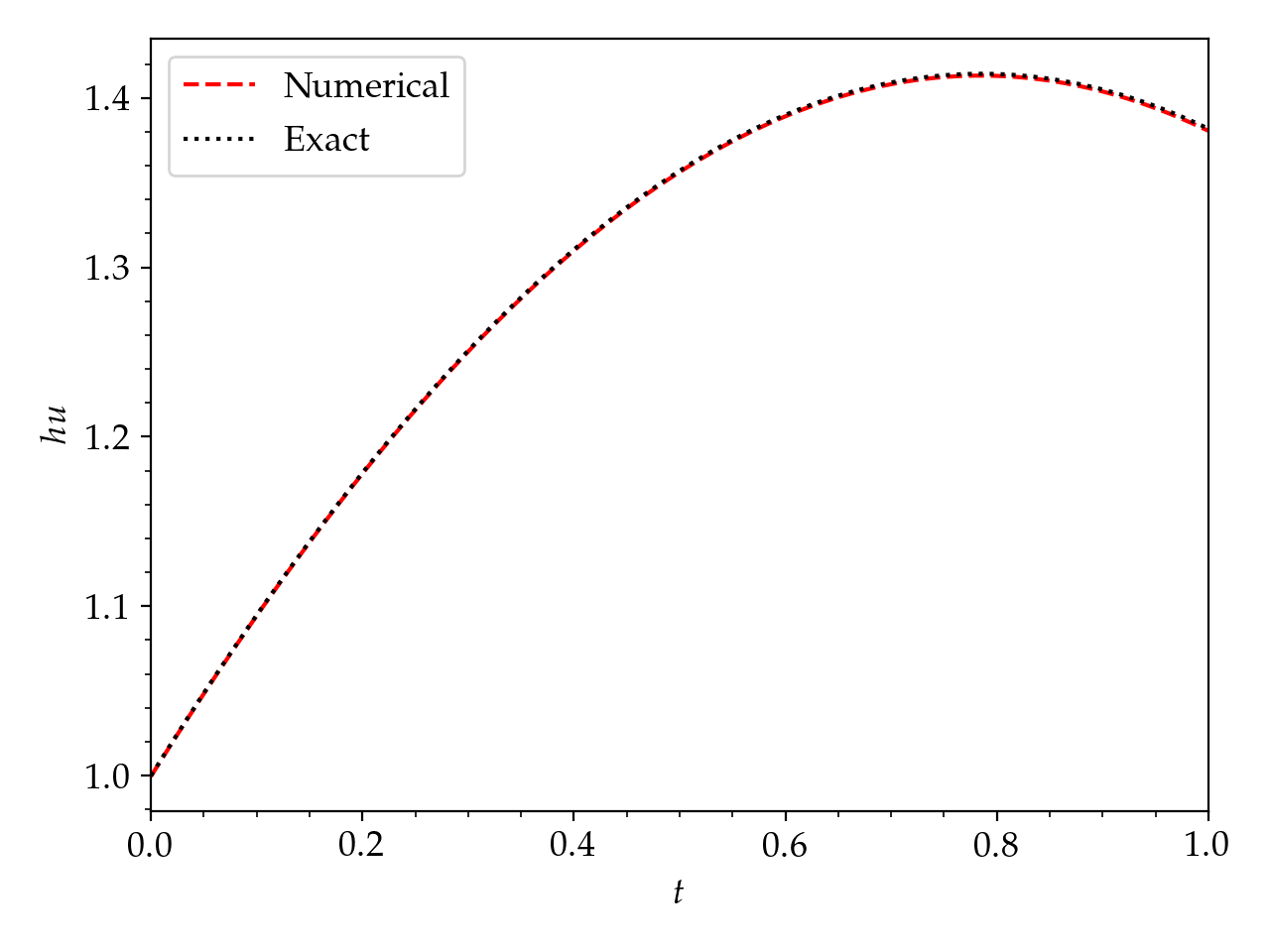}
  \includegraphics[height = 0.25\textheight]{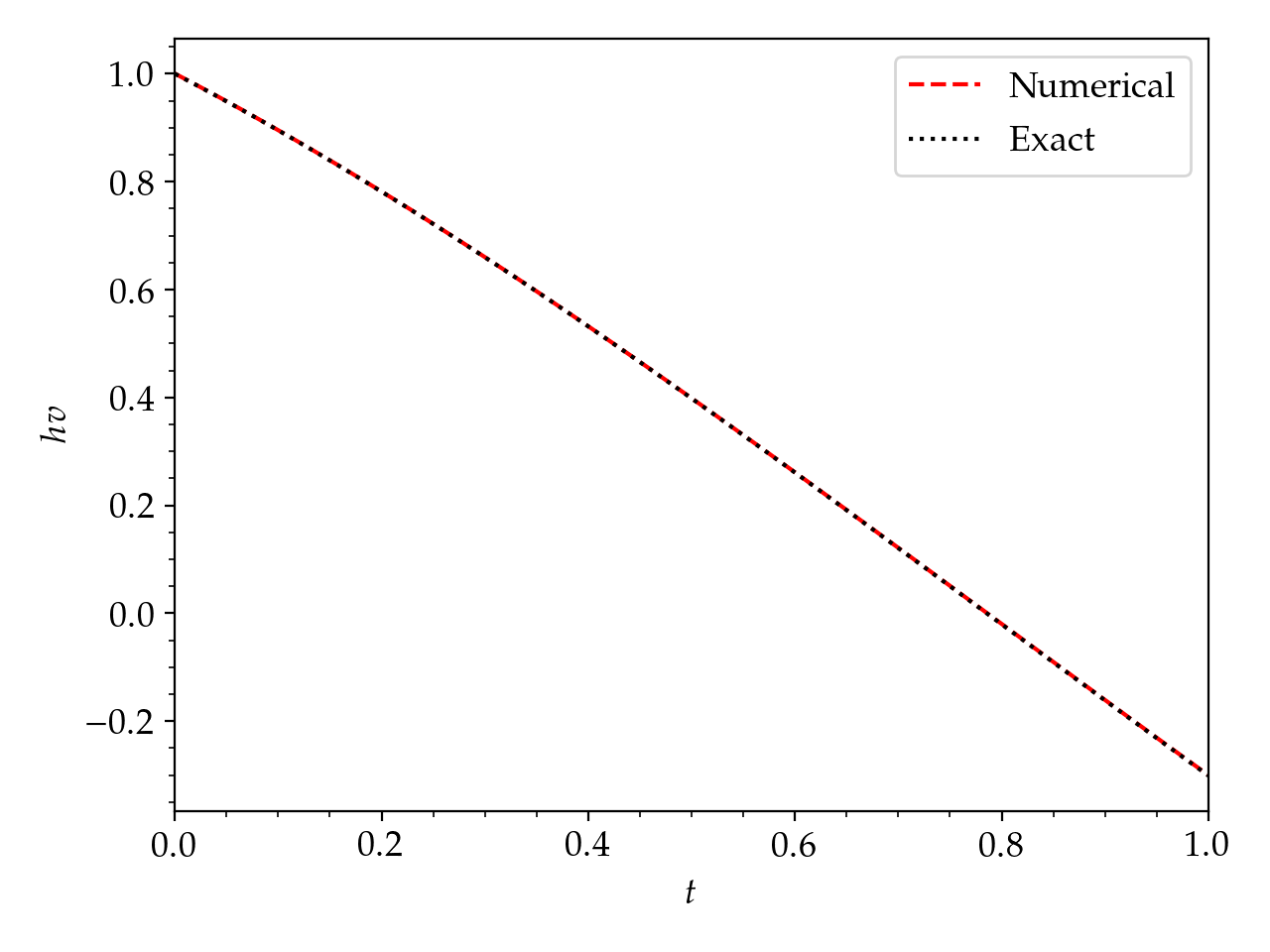}
  \caption{Momentum profiles over time, $hu$ (left) and $hv$ (right).}
  \label{fig:des-ex-soln}
\end{figure}

\subsection{Convergence Test (\cite{LNK07})}

The computational domain is $\Omega = [0,1]^2$ with parameters $\omega
= 1$ and $g = 9.8$. This test case is designed to illustrate the
convergence properties of the scheme and to determine the experimental
order of convergence (EOC). The initial data are given by 
\begin{gather}
  h(0,x,y) = 10 + \exp(\sin(2\pi x))\cos(2\pi y),\quad b(x,y) =
  \sin(2\pi x) + \cos(2\pi y), \\ 
  hu(0,x,y) = \sin(\cos(2\pi x))\sin(2\pi y),\quad hv(0,x,y) =
  \cos(2\pi x)\cos(\sin(2\pi y)). 
\end{gather}
We impose periodic boundary conditions and set the final time to $T =
0.05$. The solution is computed on successive grids of size $k \times
k$, with $k = 2^j$ for $j = 4, \dots, 8$. Since the exact solution is
not available, a reference solution is obtained on a $512 \times 512$
grid. The $L^2$-errors between the reference and numerical solutions
are reported in Table \ref{tab:eoc}. The results confirm first-order
convergence across all variables.  

\begin{table}[htpb]
  \centering
  \begin{tabular}{|c|c|c|c|c|c|c|}
    \hline
    $k$ & $L^2$-error in $h$ & EOC & $L^2$-error in $u$ & EOC &
                                                                $L^2$-error in $v$ & EOC  \\ 
    \hline
    16 & 0.2038 & - & 0.0868 & - & 0.2022 & - \\
    \hline
    32 & 0.0897 & 1.18 & 0.0312 & 1.47 & 0.0882 & 1.19\\
    \hline
    64 & 0.0368 & 1.28 & 0.0119 & 1.38 & 0.0356 & 1.30 \\
    \hline
    128 & 0.0147 & 1.32 & 0.0048 & 1.30 & 0.0139 & 1.35 \\ 
    \hline
    256 & 0.0047 & 1.63 & 0.0015 & 1.61 & 0.0044 & 1.66\\
    \hline
  \end{tabular}
  \caption{Errors and EOC values for the convergence test case.}
  \label{tab:eoc}
\end{table}

\subsection{Stationary Vortex (\cite{CDK+18})} 
The computational domain is $\Omega = [-1,1]^2$, with a flat bottom
($b \equiv 0$), $\omega = 1/\varepsilon$, and $g = 1/\varepsilon^2$,
where $\varepsilon = 0.05$, corresponding to a low-Froude
regime in geophysics. Denoting $r = \sqrt{x^2 + y^2}$, the initial
data are given by 
\begin{gather}
  h(0,x,y) = 1 + \veps^2
  \begin{dcases}
    \frac{5}{2}(1+5\veps^2)r^2, &r<\frac{1}{5}, \\
    \frac{1}{10}(1+5\veps^2) + \delta(r) + \veps^2\kappa(r),
    &\frac{1}{5}\leq r < \frac{2}{5}, \\ 
    \frac{1}{5}(1-10\veps^2 + 20\veps^2\ln(2)), &r\geq \frac{2}{5},
  \end{dcases} \\
  u(0, x, y) = -\veps y \Gamma(r),\quad v(0,x,y) = \veps x\Gamma(r), \quad \Gamma(r) = 
    \begin{dcases}
      5, &r<\frac{1}{5}, \\
      \frac{2}{r}-5, &\frac{1}{5}\leq r < \frac{2}{5}, \\
      0, &r\geq \frac{2}{5},
    \end{dcases}
  \end{gather}
  where
  \[
    \delta(r) = 2r - \frac{3}{10}-\frac{5}{2}r^2, \quad \kappa(r) =
    4\ln(5r) + \frac{7}{2} - 20r + \frac{25}{2}r^2. 
  \]

The computations are carried out on a $200 \times 200$ grid with
extrapolation boundary conditions up to the final time $T =
10$. Figure \ref{fig:sta_vort_ht} shows the water height at the final
time. Although the chosen configuration represents a steady state of
the RSW system, it is not a geostrophic steady state and, therefore,
not an exact discrete steady state of the present
scheme. Nevertheless, the well-balanced scheme successfully preserves
the circular and symmetric structure of the vortex, as illustrated in
both plots of Figure \ref{fig:sta_vort_ht}. 

\begin{figure}[htpb]
  \centering
  \includegraphics[height = 0.25\textheight]{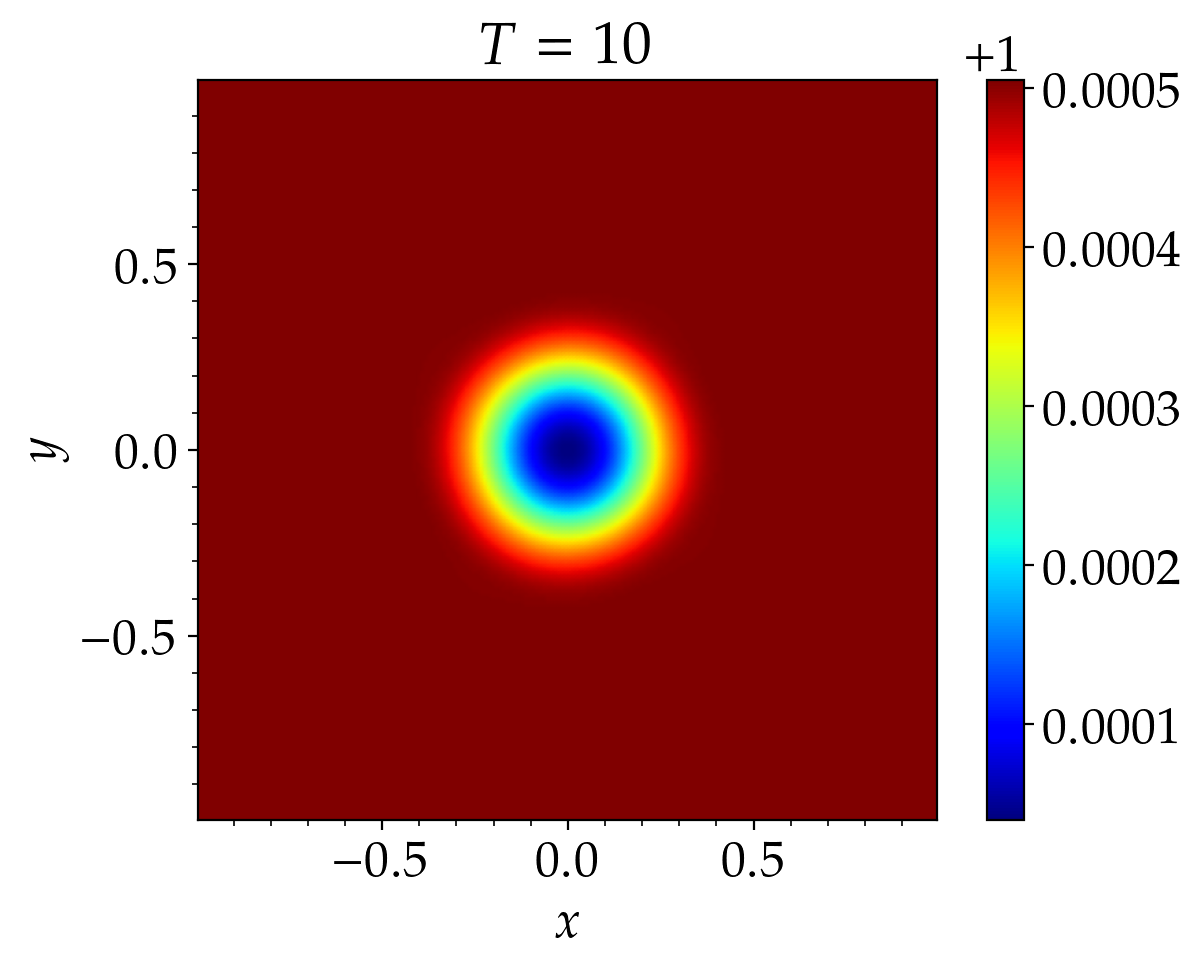}
  \includegraphics[height = 0.25\textheight]{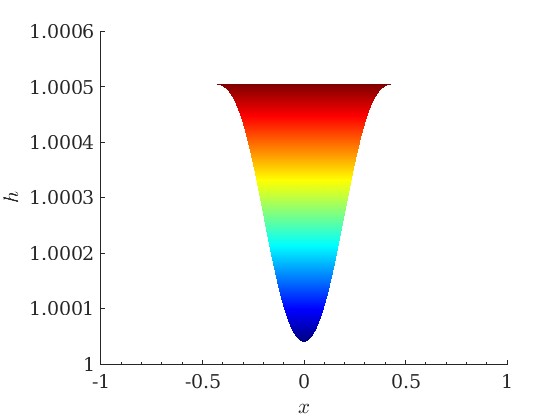}
  \caption{The height $h$ of the stationary vortex computed at $T =
    10$. Pseudocolor plot (left) and the slice along $y = 0$ (right).} 
  \label{fig:sta_vort_ht}
\end{figure}

\subsection{Geostrophic Adjustment(\cite{CLP08})}

In this test case, an elliptical perturbation in the water height is
introduced into a lake-at-rest configuration. This perturbation
generates an outward-propagating shock wave with a rotating flow
behind it. The computational domain is $\Omega = [-10,10]^2$, with
flat topography ($b \equiv 0$) and parameters $\omega = g = 1$. The
initial data are specified by 
\begin{gather}
  h(0,x,y) = 1 + \frac{1}{4}(1 - \tanh(10(\sqrt{2.5x^2 + 0.4y^2} -
  1)), \quad u(0, x, y) = v(0, x, y) = 0. 
\end{gather}
Figure \ref{fig:geo-adj} displays the water height $h$, computed on
grids of size $200 \times 200$ and $400 \times 400$, at times $T = 4$
and $T = 8$. Despite the non-conservative formulation of the pressure
term, the scheme accurately captures the propagation of
discontinuities over time as well as the rotating flows behind the
shock. 
\begin{figure}
    \centering
    \includegraphics[height = 0.20\textheight]{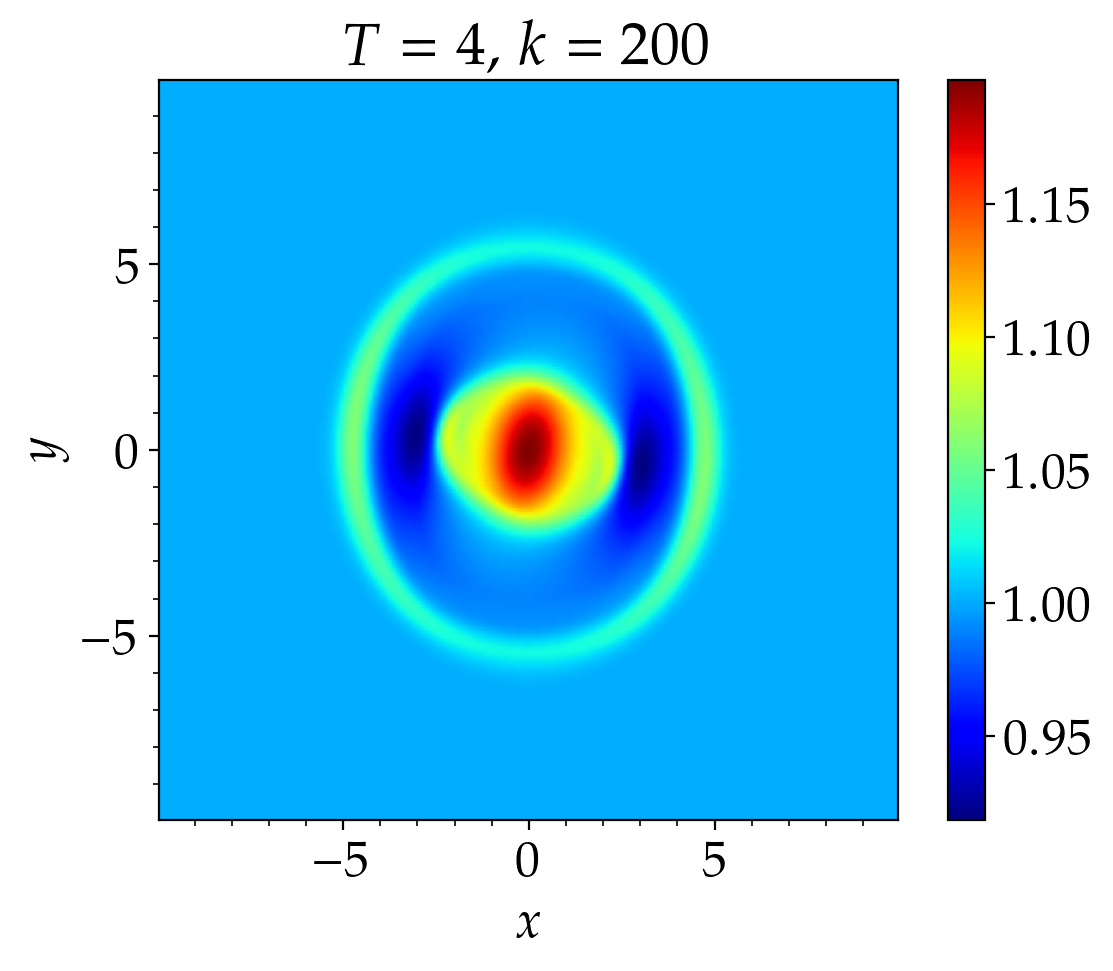}
    \includegraphics[height = 0.20\textheight]{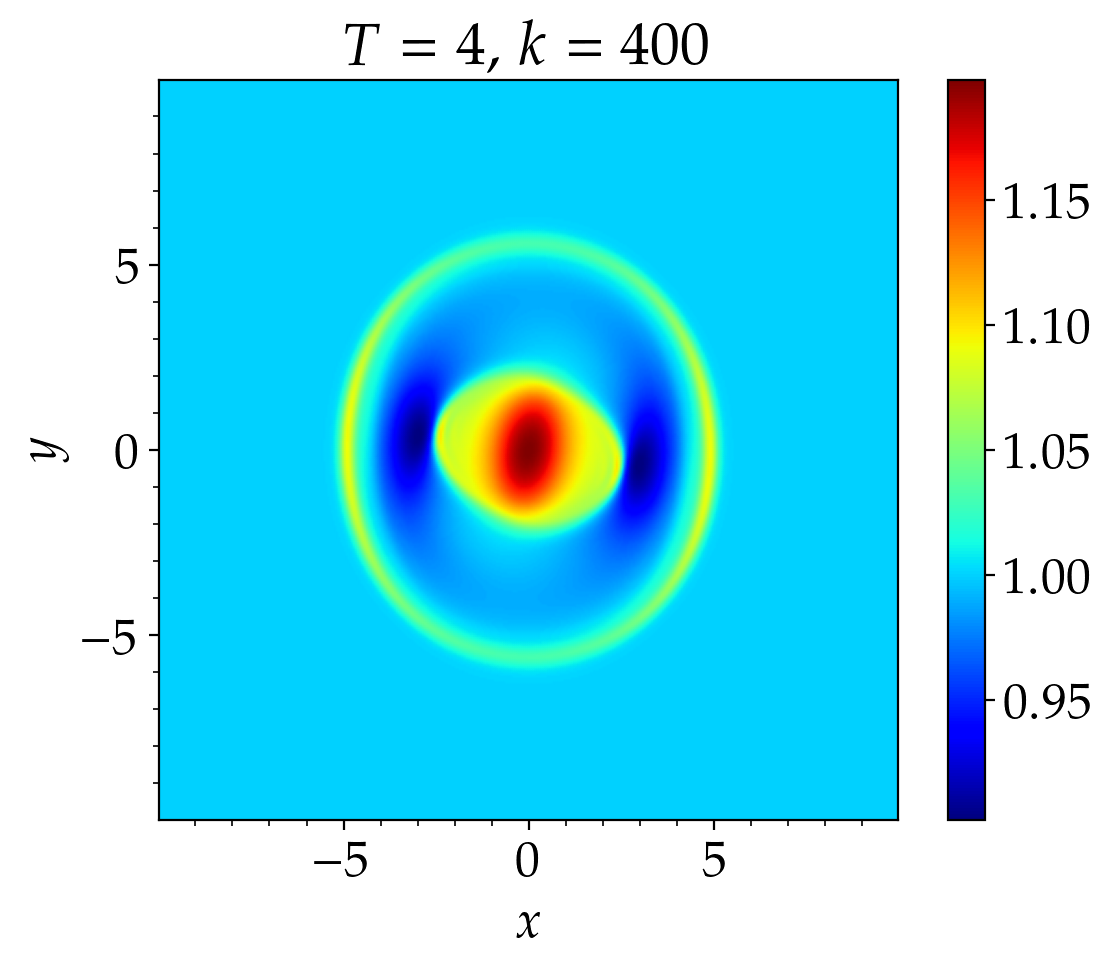}
    \includegraphics[height = 0.20\textheight]{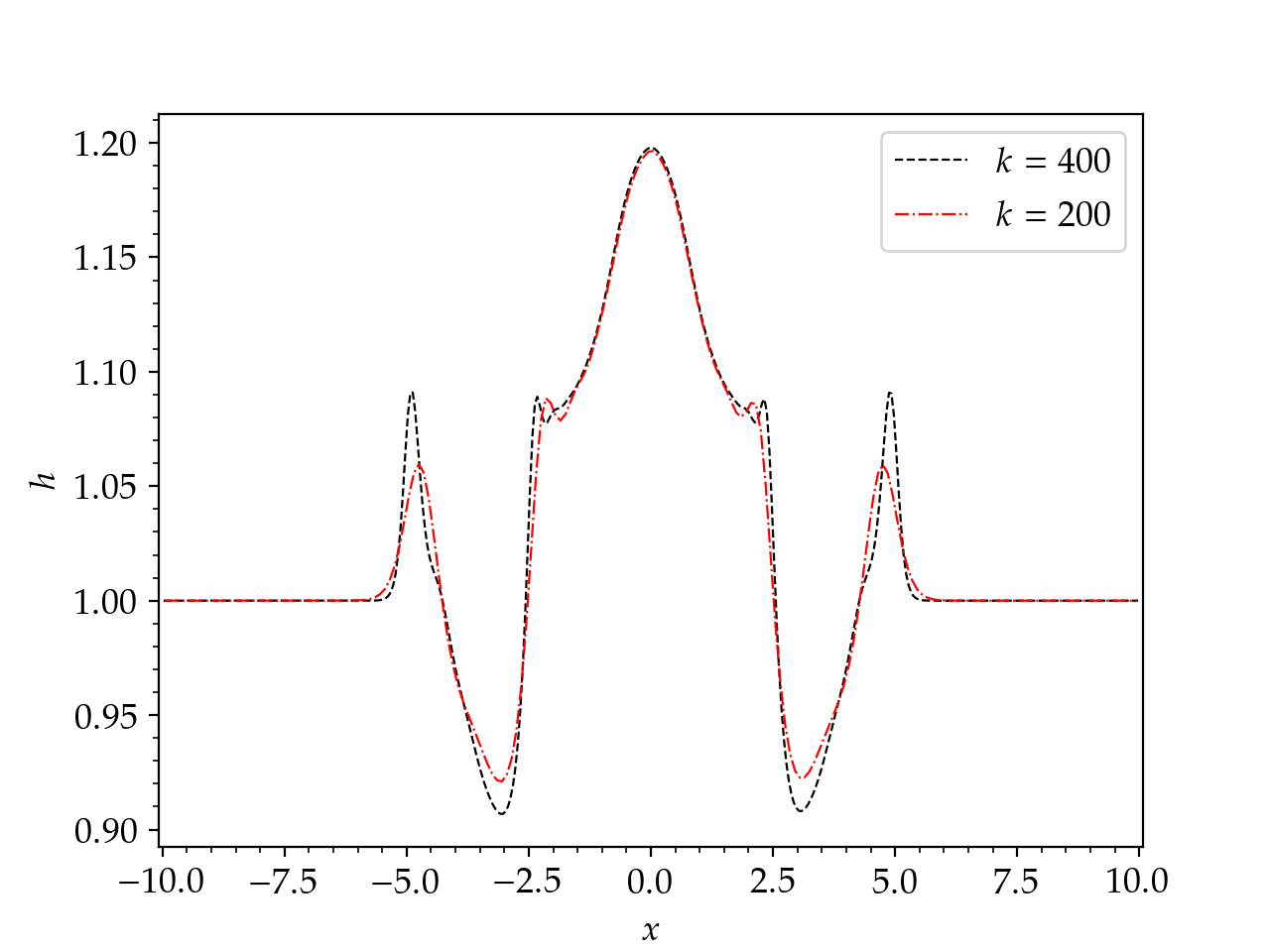}
    \includegraphics[height = 0.20\textheight]{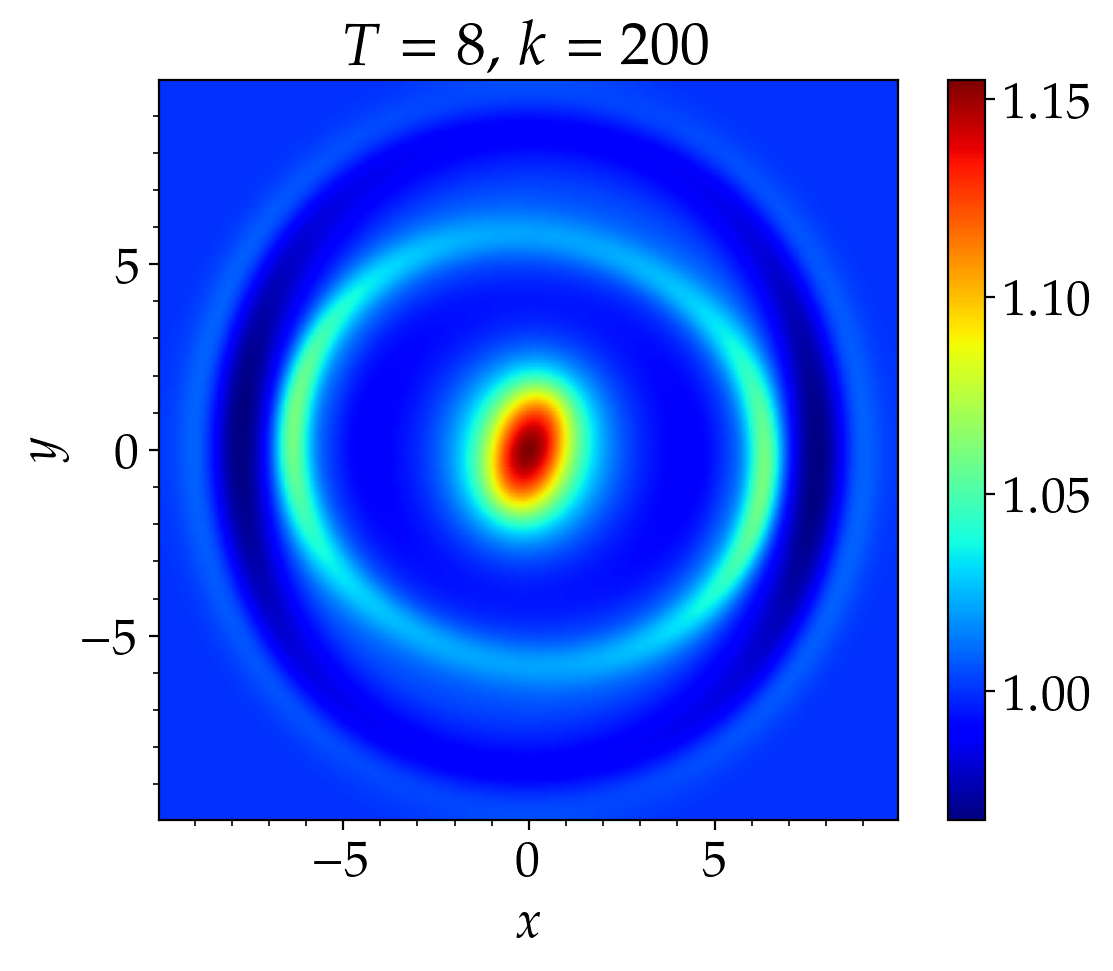}
    \includegraphics[height = 0.20\textheight]{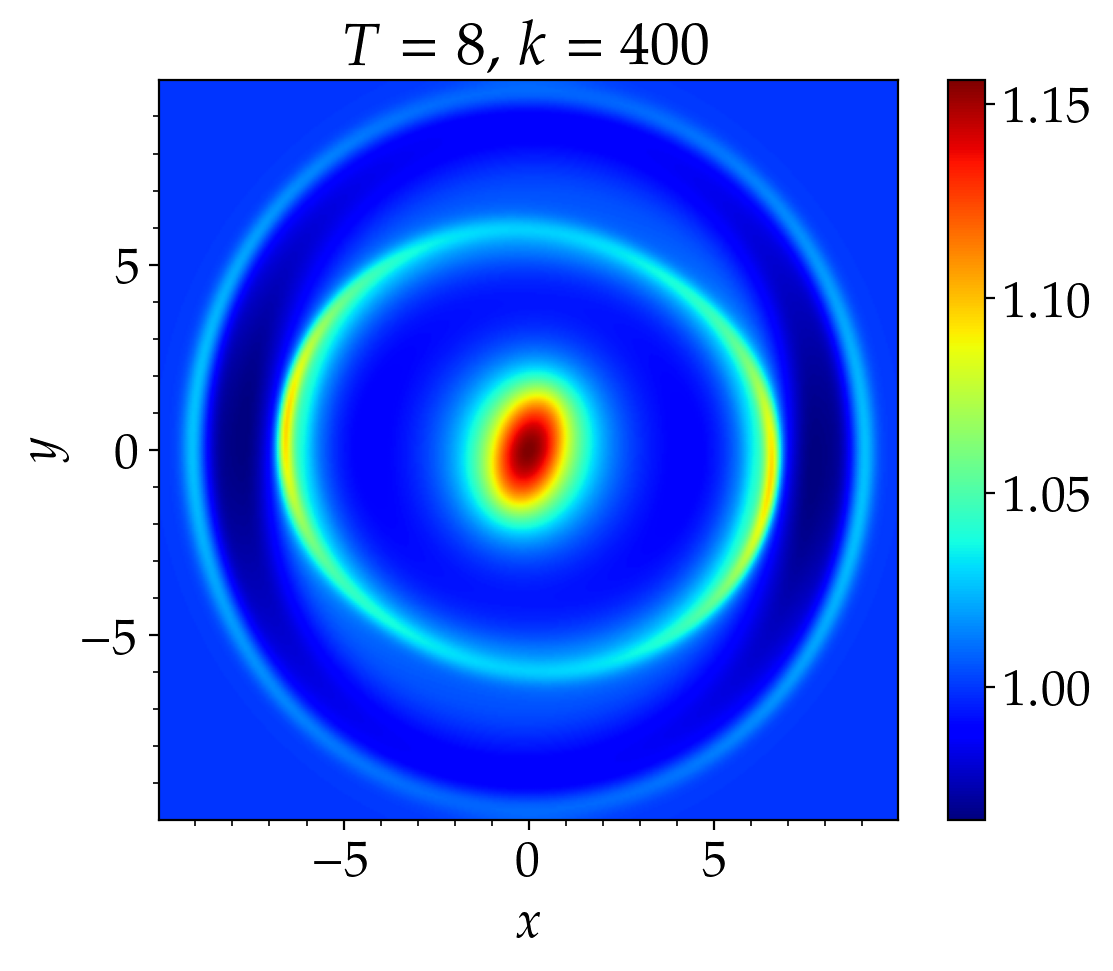}
    \includegraphics[height = 0.20\textheight]{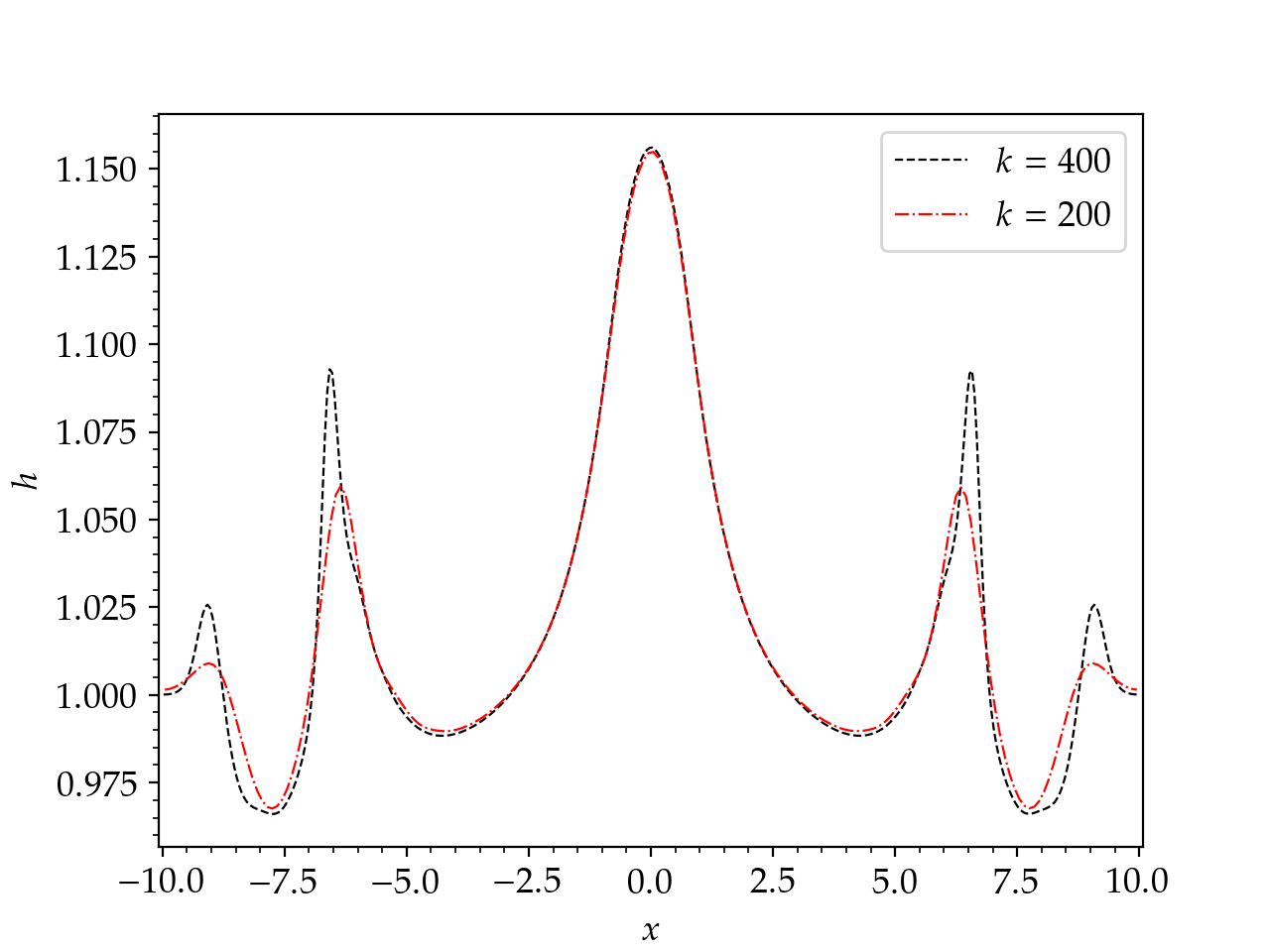}
    \caption{The water height $h$ at times $T = 4$ (top) and $T = 8$ (bottom). \\ Pseudocolor plots (left and middle) and the slice across $y = 0$ (right).}
    \label{fig:geo-adj}
\end{figure}

\subsection{Shear Flow Evolution (\cite{WB19})}

We consider a doubly periodic domain $[0, L_x]\times[0, L_y]$, with
$L_x = L_y = 5000$ km. We set $b\equiv 0$ and $g = 9.80616\text{
  m}/\text{s}^2$ and $\omega = 6.147\times 10^{-5}\text{ s}^{-1}$. The
initial data is as follows: 

\begin{gather}
    h(0,x,y) = H_0 - H^\prime\frac{y^{\prime\prime}}{\sigma_y}\exp\biggl(\frac{-y{^{\prime}}^2}{2\sigma_y^2} + \half\biggr)\biggl(1 + \kappa\sin\biggl(\frac{2\pi x^\prime}{\lambda_x}\biggr)\biggr),\\
    u(0,x,y) = \frac{gH^\prime}{\omega\sigma_y L_y}\biggl(c(y) - \frac{y{^{\prime\prime}}^2}{\sigma_y^2}\biggr)\exp\biggl(\frac{-y{^{\prime}}^2}{2\sigma_y^2} + \half\biggr)\biggl(1 + \kappa\sin\biggl(\frac{2\pi x^\prime}{\lambda_x}\biggr)\biggr), \\
    v(0,x,y) = -\frac{gH^\prime}{\omega L_x}\frac{2\pi x^\prime}{\lambda_x}\frac{y^{\prime\prime}}{\sigma_y}\exp\biggl(\frac{-y{^{\prime}}^2}{2\sigma_y^2} + \half\biggr)\cos\biggl(\frac{2\pi x^\prime}{\lambda_x}\biggr).
\end{gather}
where $c(y) = \cos\bigl(\frac{2\pi}{L_y}\bigl(y - \frac{L_y}{2}\bigr)\bigr)$, $x^\prime = \frac{x}{L_x},\quad y^\prime = \frac{1}{\pi}\sin\bigl(\frac{\pi}{L_y}\bigl(y - \frac{L_y}{2}\bigr)\bigr),\ y^{\prime\prime} = \frac{1}{2\pi}\sin\bigl(\frac{2\pi}{L_y}\bigl(y - \frac{L_y}{2}\bigr)\bigr)$
with $\lambda_x = 0.5,\ \sigma_y = \frac{1}{12},\ \kappa = 0.1$.
Finally, $H_0 = 1076\text{ m}$ and $H^\prime = 30\text{ m}$. The above
choice of the parameters yield a flow in the quasi-geostrophic
regime \cite{Ped12}. We present the pseudocolor plots of the evolution of the
height $h$ over a period of 15 days computed on a $512\times 512$ grid
in Figure \ref{fig:ht-shear-flow}. In addition, we also present the
pseudocolor plots of the potential vorticity $PV$, defined as 
\[
  PV = \dfrac{\omega + \text{curl}(u,v)}{h},\quad \text{curl}(u,v) =
  \Dx v - \Dy u,
\]
in Figure \ref{fig:pv-shear-flow}. The initially imposed perturbations
cause the flow to evolve into two pairs of counter-rotating
vortices. By day 7, these vortex pairs are clearly separated and
individually distinguishable, cf.\ Figure
\ref{fig:ht-shear-flow}. Under the action of the Coriolis force, the
rotating fluid subsequently generates even finer vortical structures,
as evidenced by the potential vorticity ($PV$) plots for days 13 and
15. Hence, although the scheme is well-balanced with respect to jets
in the rotational frame, it also demonstrates the capability to
capture flows in the quasi-geostrophic regime with desirable accuracy.  
\begin{figure}[htpb]
  \centering
  \includegraphics[height =
  0.2\textheight]{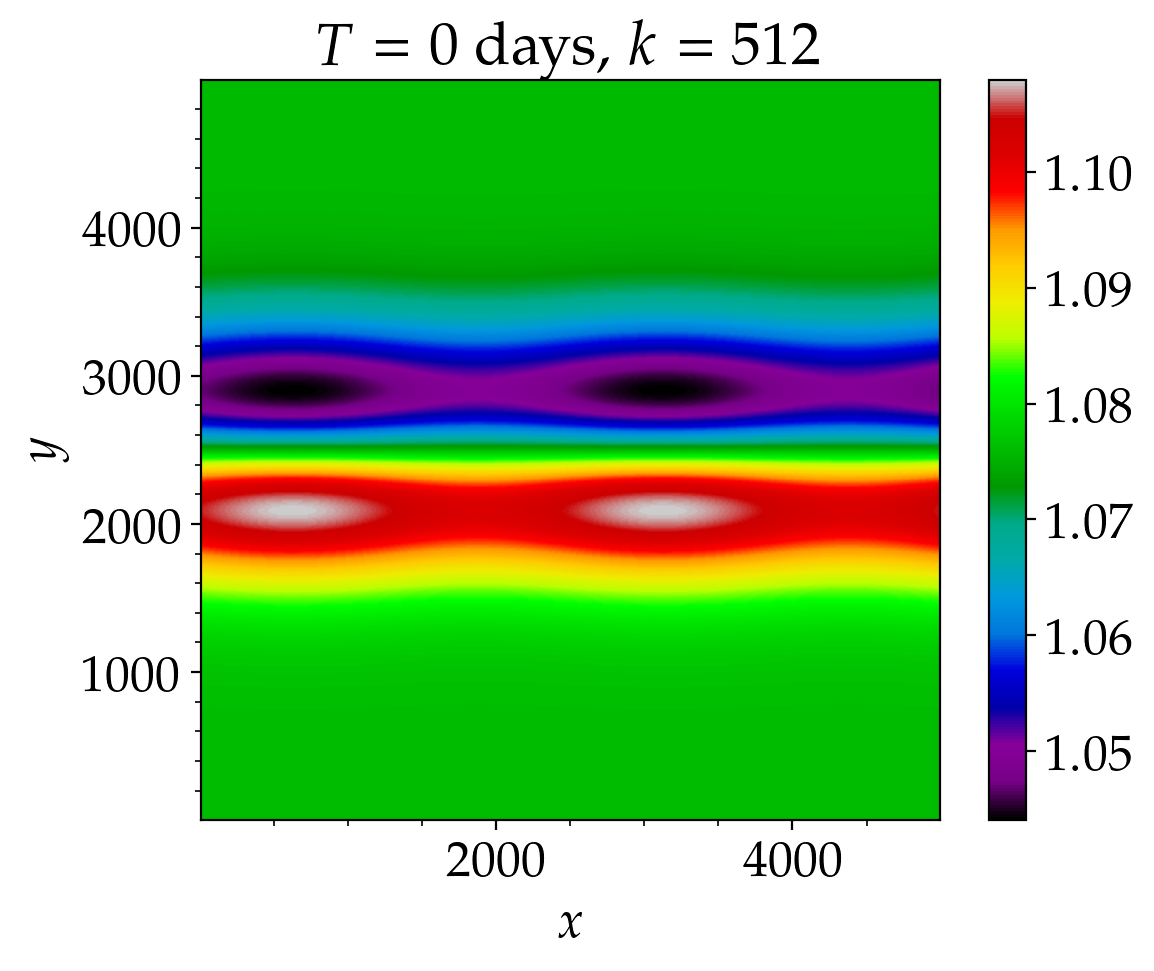} 
  \includegraphics[height =
  0.2\textheight]{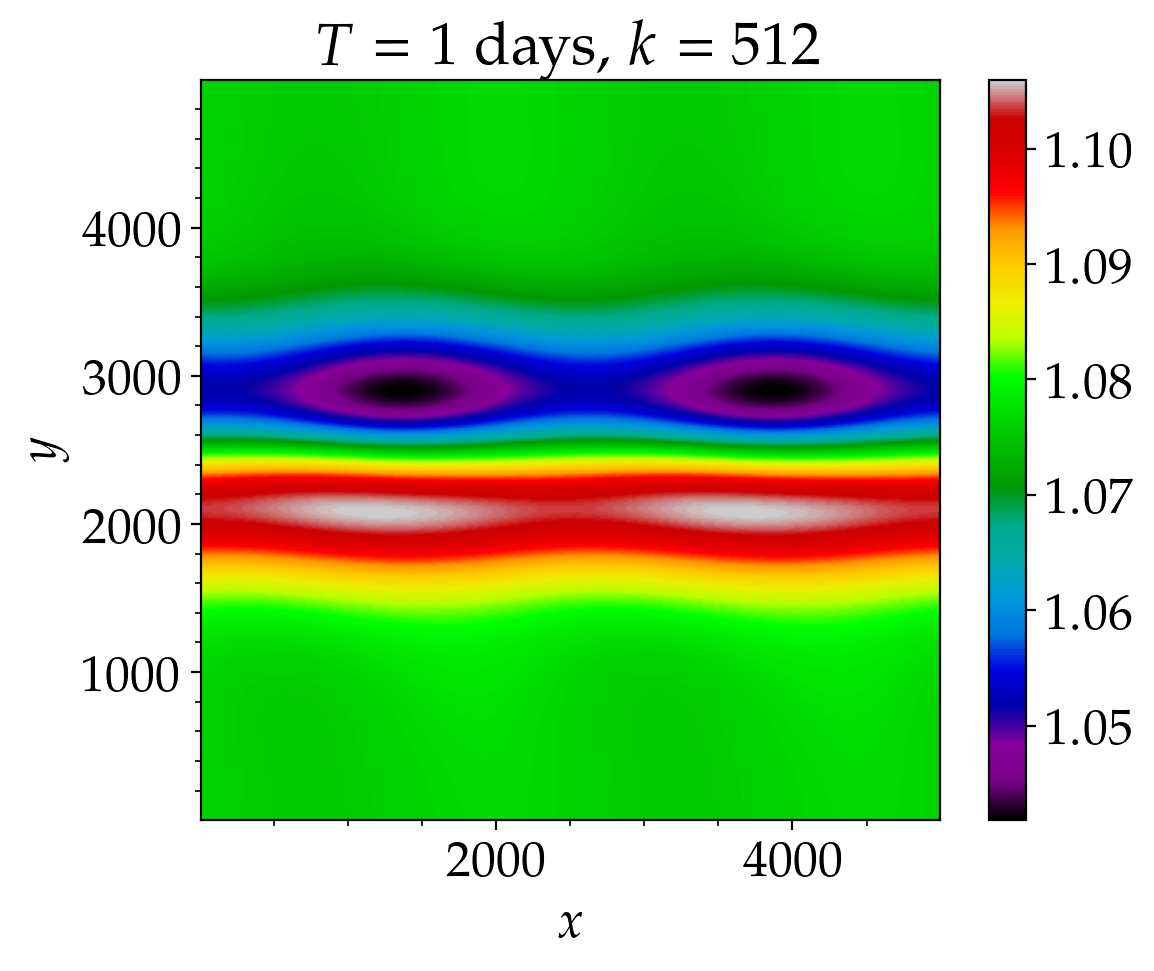} 
  \includegraphics[height =
  0.2\textheight]{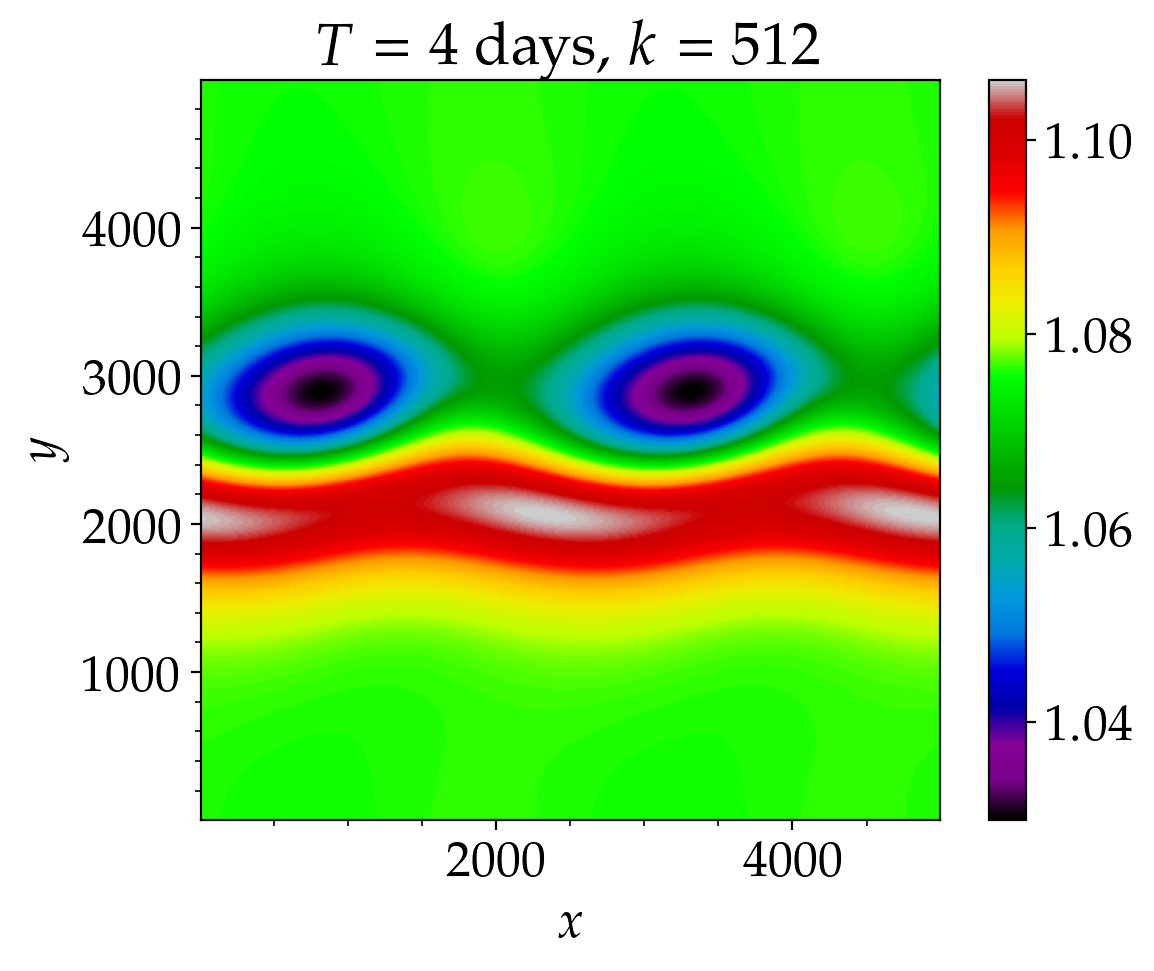} 
  \includegraphics[height =
  0.2\textheight]{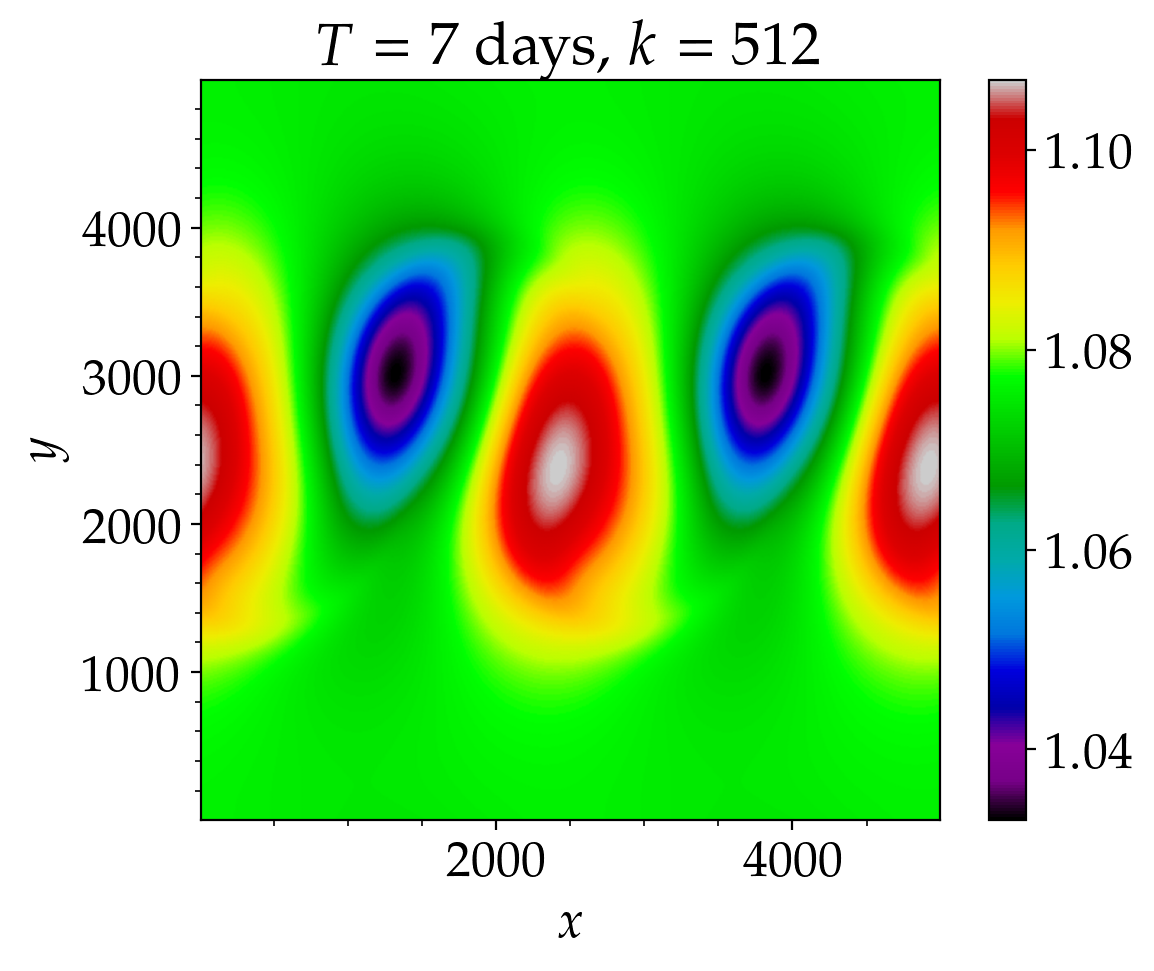} 
  \includegraphics[height =
  0.2\textheight]{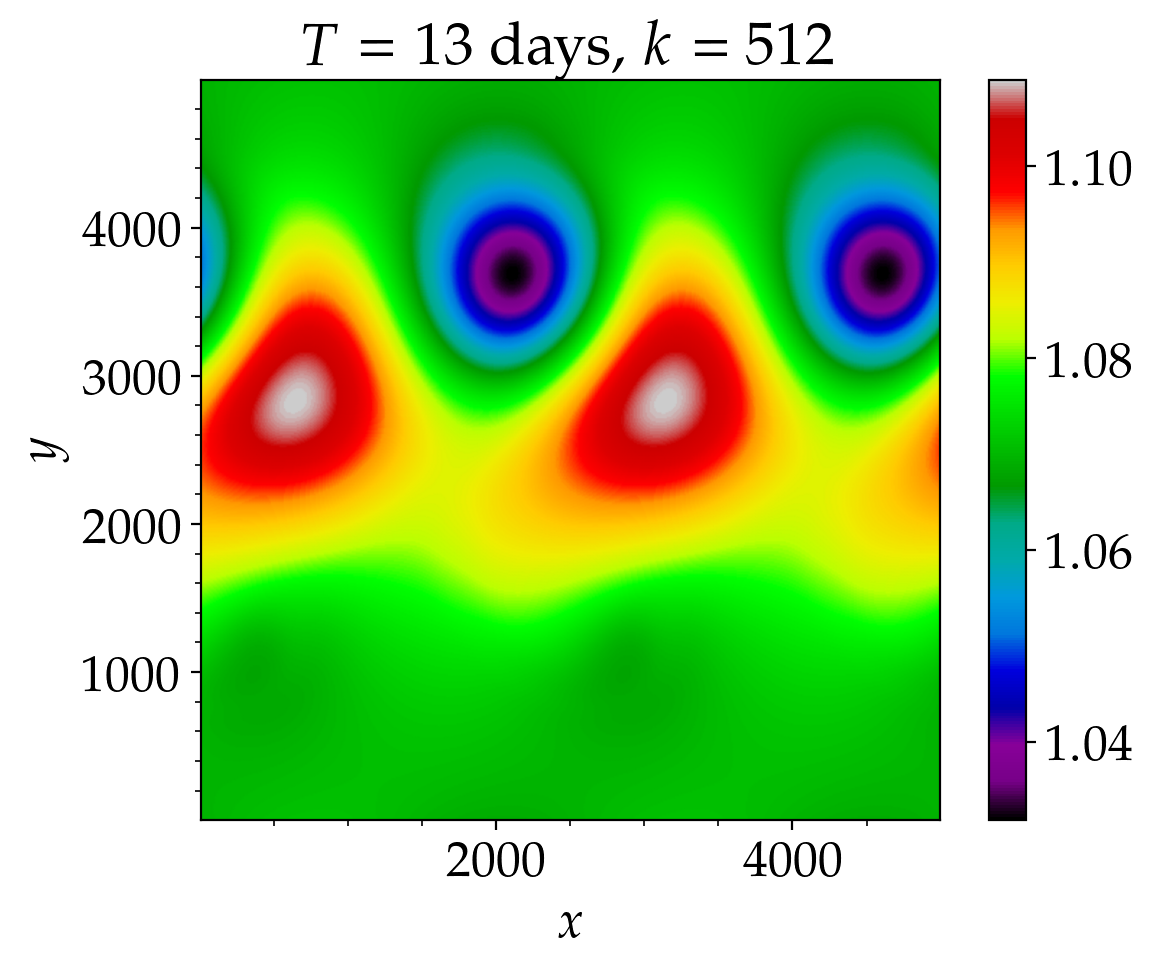} 
  \includegraphics[height =
  0.2\textheight]{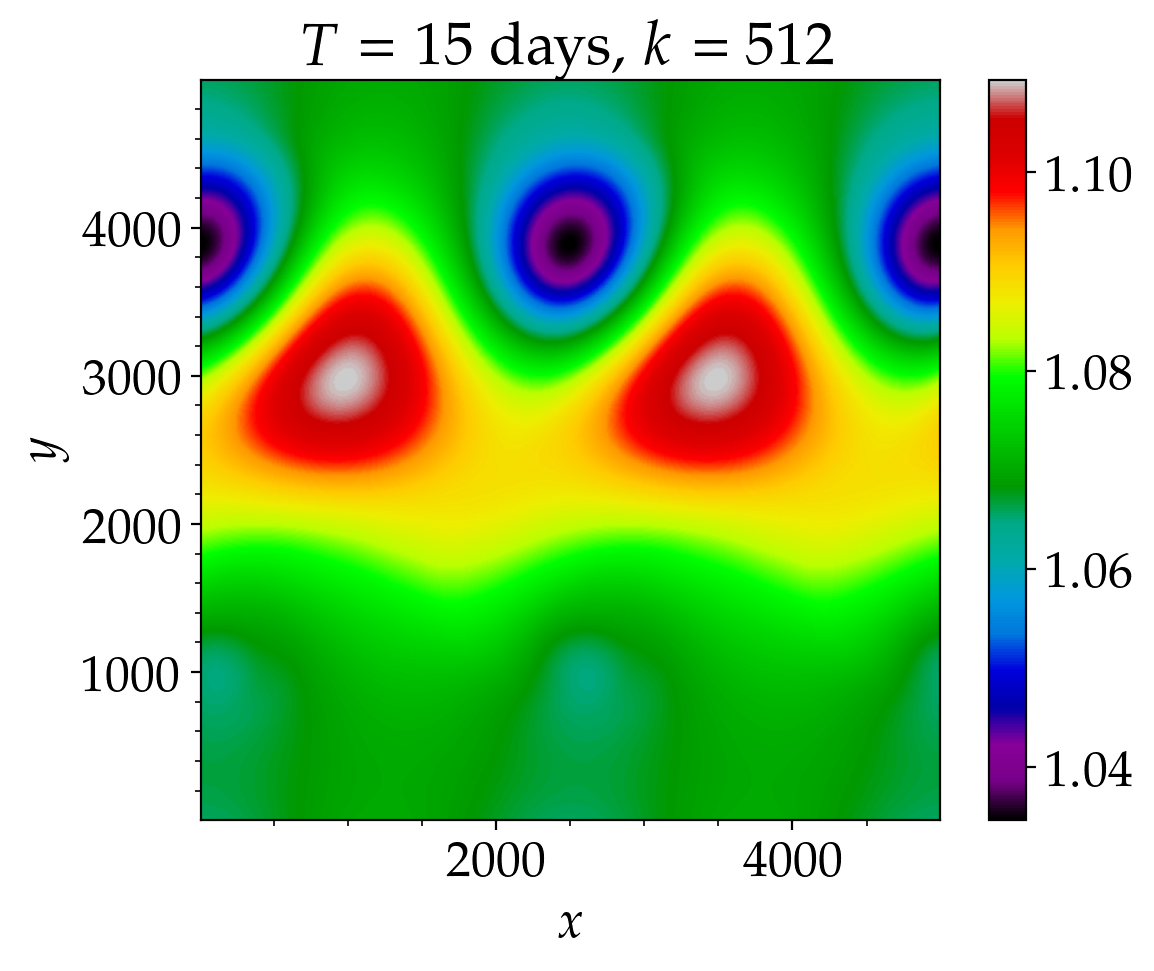} 
  \caption{Evolution of the height $h$ (in km) for the shear flow problem.}
  \label{fig:ht-shear-flow}
\end{figure}
\begin{figure}[htpb]
  \centering
  \includegraphics[height =
  0.2\textheight]{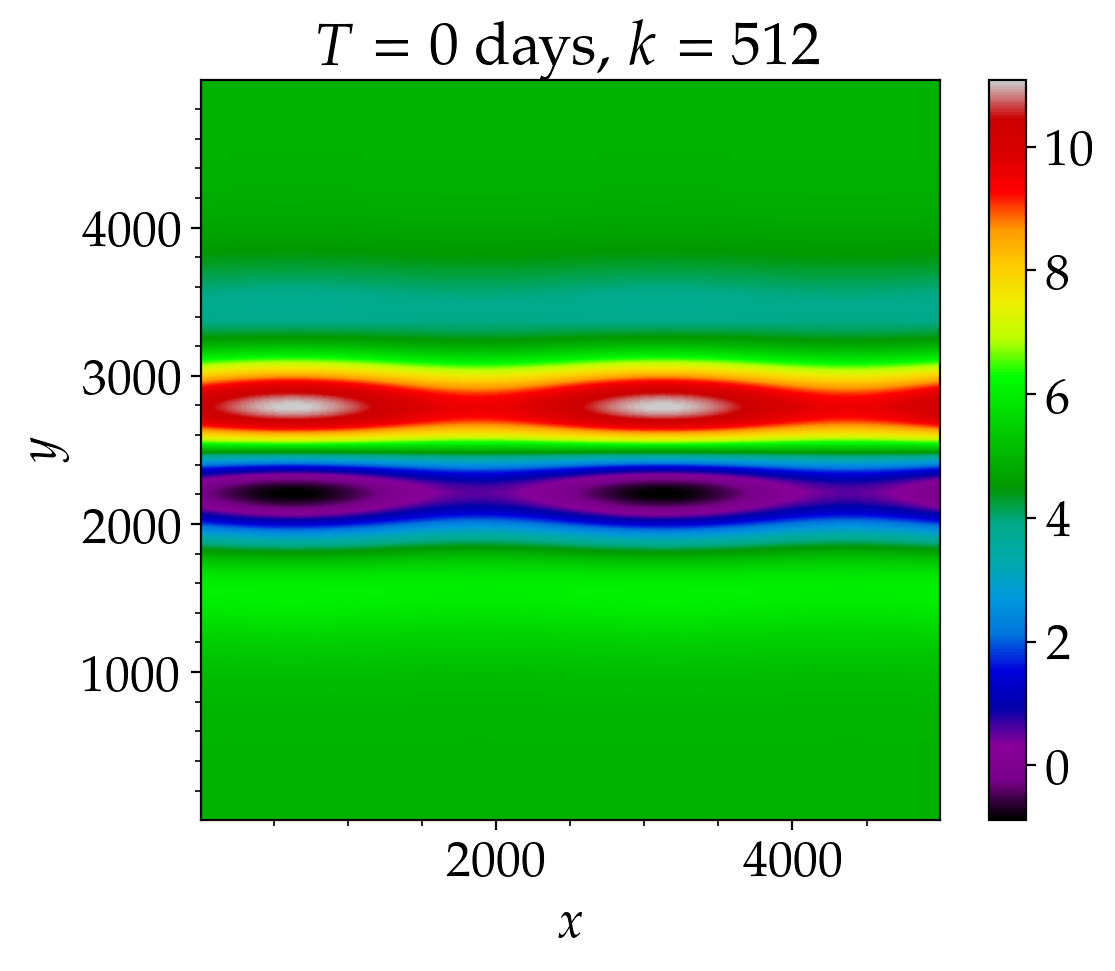} 
  \includegraphics[height =
  0.2\textheight]{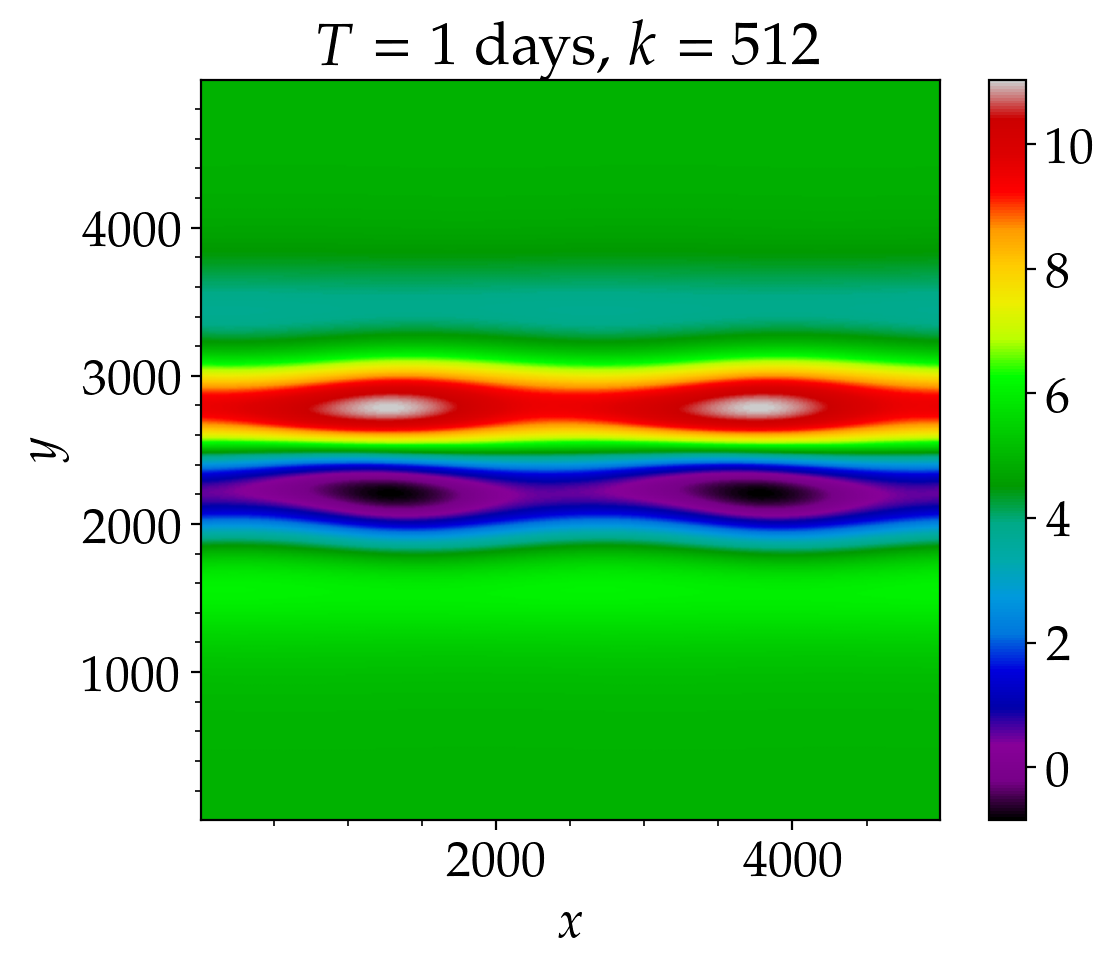} 
  \includegraphics[height =
  0.2\textheight]{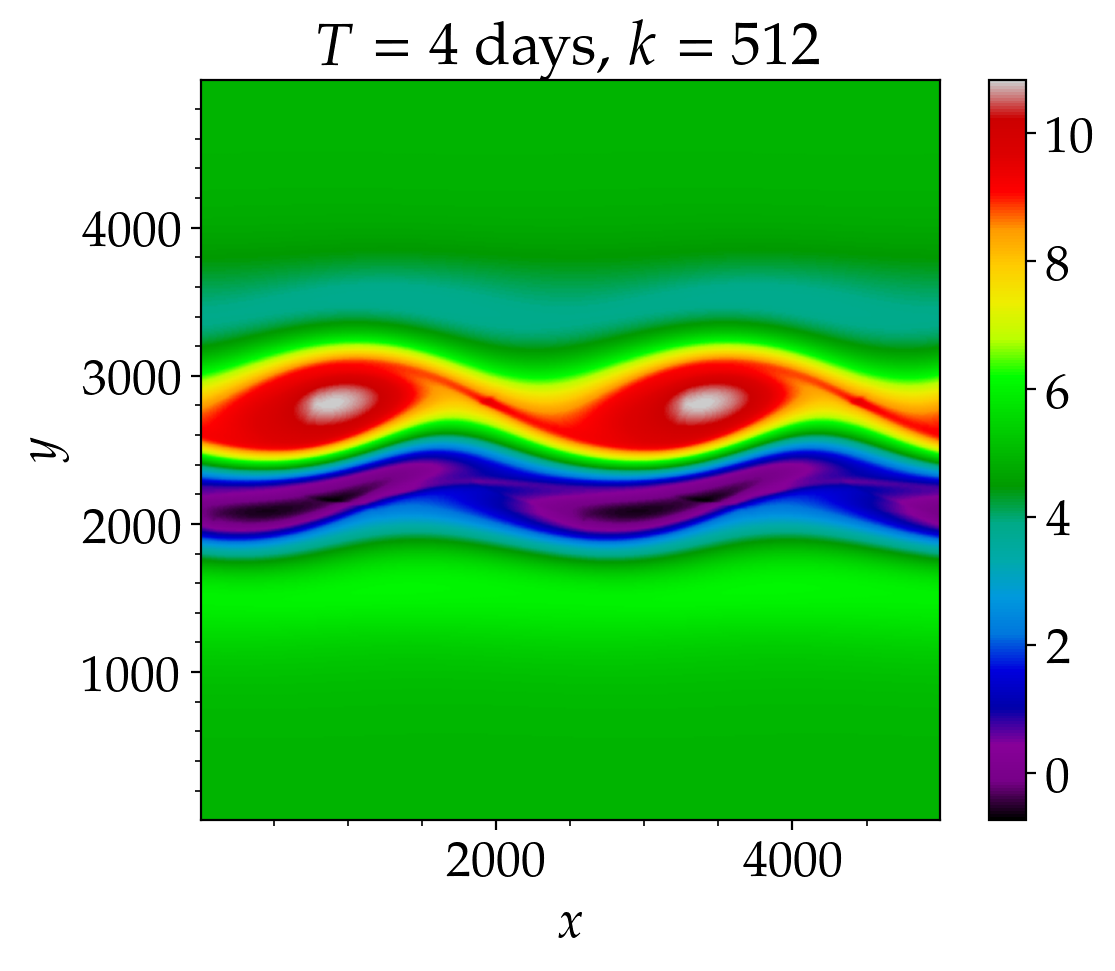} 
  \includegraphics[height =
  0.2\textheight]{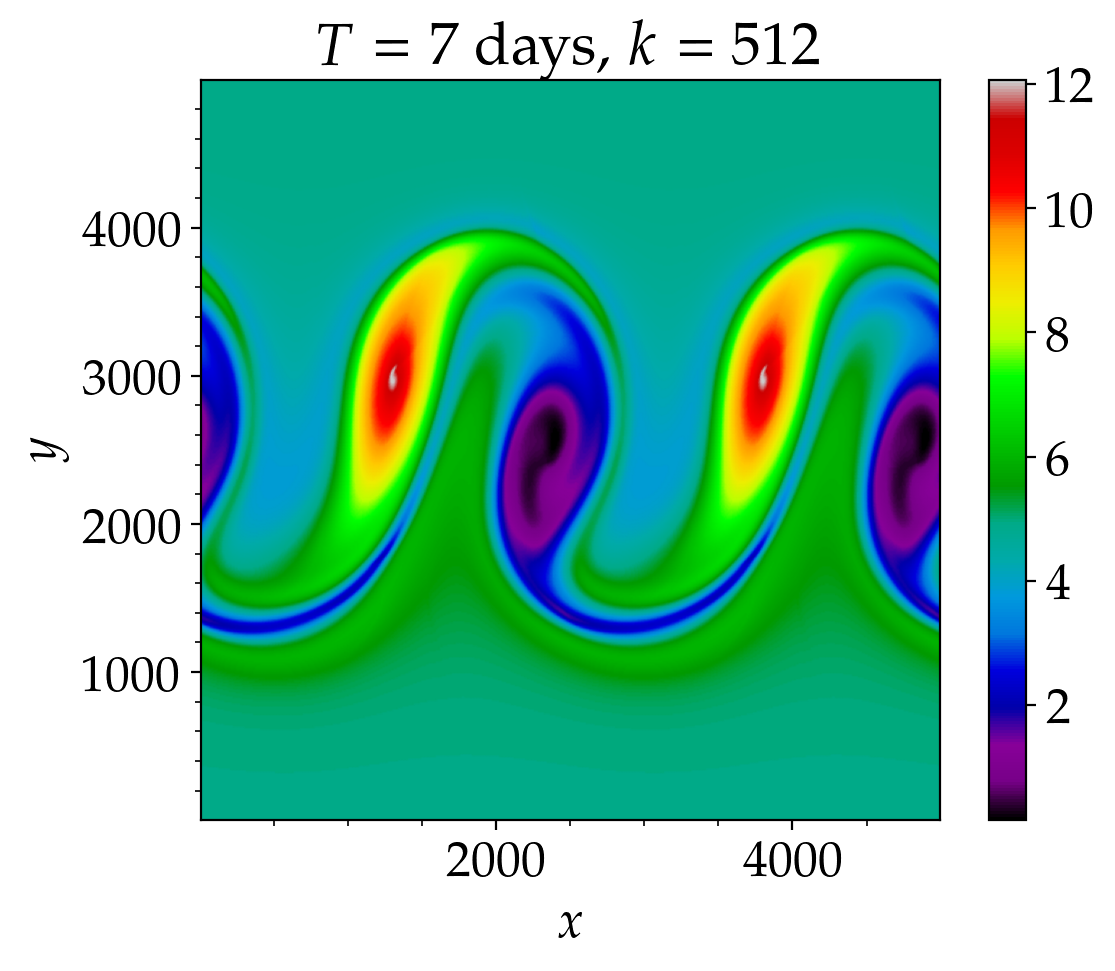} 
  \includegraphics[height =
  0.2\textheight]{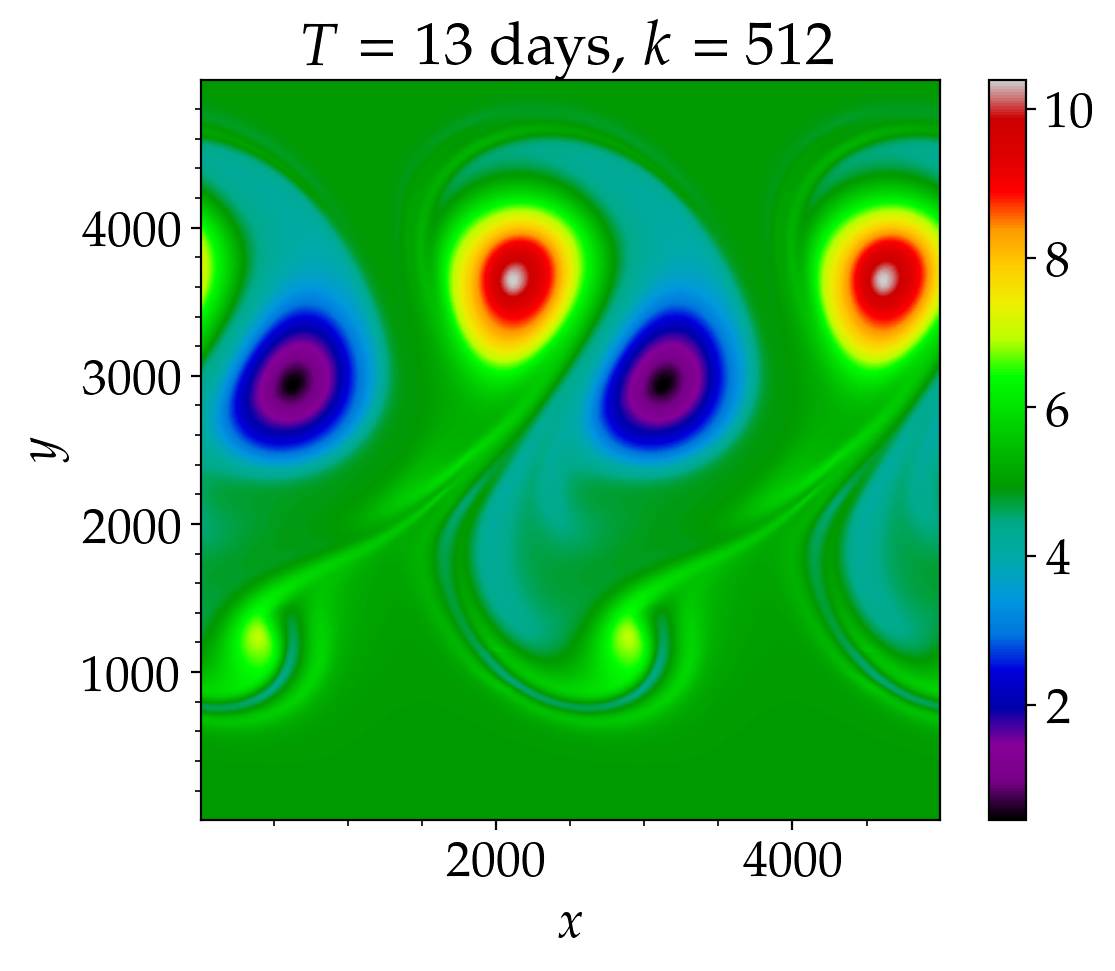} 
  \includegraphics[height =
  0.2\textheight]{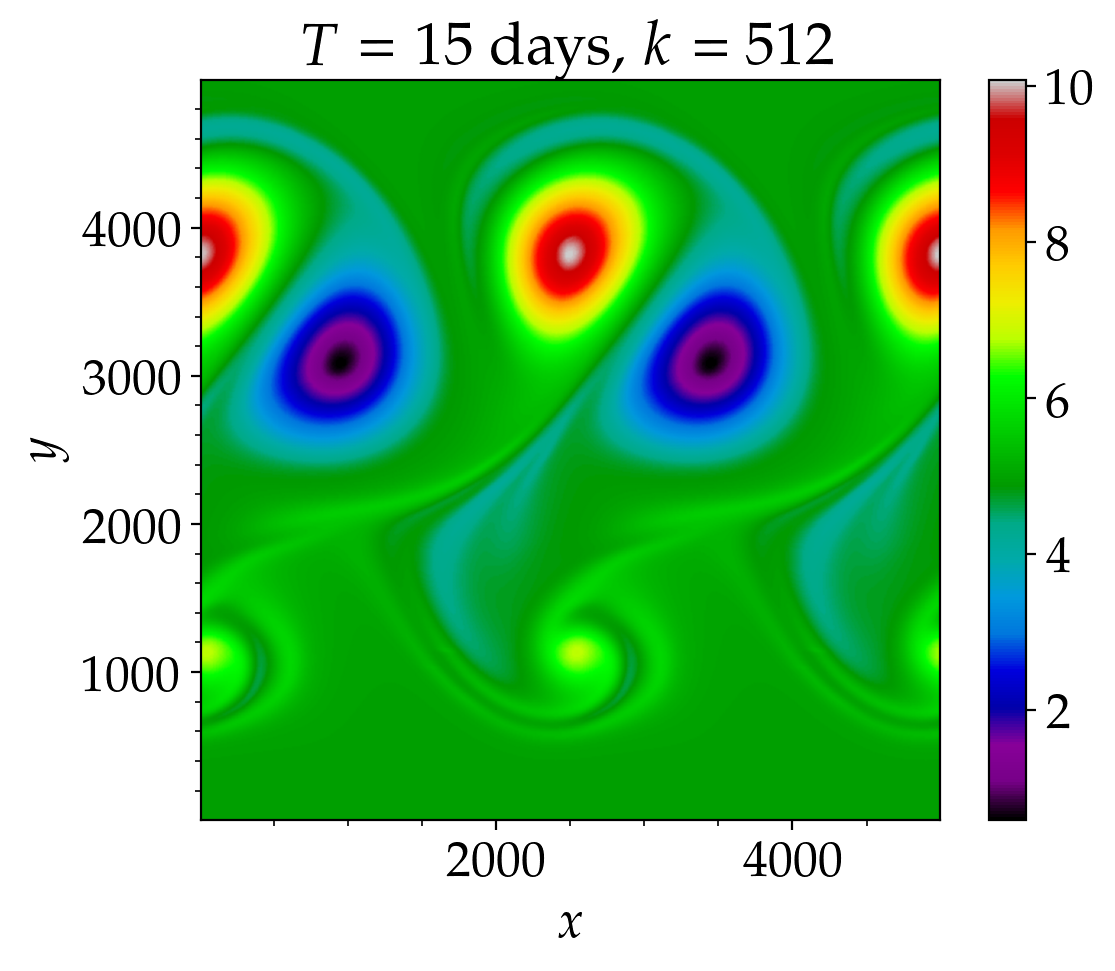} 
  \caption{Evolution of the potential vorticity $PV$ (in
    days$^{-1}$km$^{-1}$) for the shear flow problem.} 
  \label{fig:pv-shear-flow}
\end{figure}

\subsection{Vortex Pair Interaction (\cite{WB19})}

This test case is designed to demonstrate the convergence of numerical
solutions toward a DMV solution as the mesh is refined. We consider a
doubly periodic domain $[0, L_x] \times [0, L_y]$ with $L_x = L_y =
5000$ km. The bottom topography is flat ($b \equiv 0$), and the
parameters are set to $g = 9.80616\,\text{ms}^{-2}$ and $\omega =
6.147 \times 10^{-5}\,\text{s}^{-1}$. The initial data correspond to a
geostrophic equilibrium and are prescribed as follows: 
\begin{gather}
  h(0,x,y) = H_0 - H^\prime\biggl(e_1(x,y) + e_2(x,y) -
  \frac{4\pi\sigma_x\sigma_y}{L_x L_y}\biggr), \\ 
  (u,v)(0,x,y) = \frac{gH^\prime}{\omega}(-\sigma_y^{-1}(y_1^{\prime\prime}e_1(x,y) + y_2^{\prime\prime}e_2(x,y)),\ \sigma_x^{-1}(x_1^{\prime\prime}e_1(x,y) + x_2^{\prime\prime}e_2(x,y))),
\end{gather}
where $e_i(x,y) = \exp(-\half(x_i{^\prime}^2 + y_i{^\prime}^2))$ for
$i = 1,2$. Here, for $z = x,y$ and $i = 1, 2$. 
\begin{gather*}
  z_i^\prime = \frac{L_z}{\pi\sigma_z}\sin\biggl(\frac{\pi}{L_z}(z-z_{c_i})\biggr), \quad 
  z_i^{\prime\prime} = \frac{L_z}{2\pi\sigma_z}\sin\biggl(\frac{2\pi}{L_z}(z-z_{c_i})\biggr).
\end{gather*}
The centers of the vortices along with $\sigma_x$ and $\sigma_y$ are given by 
\begin{gather*}
    z_{c_1} = 0. 4L_z,\quad z_{c_2} = 0.6L_z, \quad \sigma_z = \frac{3}{40}L_z,
\end{gather*}
where $z = x,y$. We finally set $H^\prime = 75$ m and $H_0 = 750$ m in
order to give rise to a flow in the quasi-geostrophic regime. First, we compute the numerical solutions on a $512\times 512$ grid over a period of 10 days and present the evolution of the water height $h$ and the potential vorticity $PV$ in Figures \ref{fig:dbl-vortex-h}-\ref{fig:dbl-vortex-pv}. Owing to the chosen initial
separation between the vortex cores, the vortices remain too far apart
to merge, and nonlinear effects lead to a mutual repulsion of the
cores. 

Now, we wish to verify the convergence of the numerical solutions towards a DMV solution. To begin, we set a final time of $T = 15$ days and compute the numerical solutions on successive grids with resolutions $k = 2^j$, $j = 5, \dots, 9$. The pseudocolor plots of the water height $h$ and potential vorticity $PV$ for these mesh refinements are presented in Figure \ref{fig:pcol-dbl-vortex}. Next, to illustrate the convergence to a DMV solution, we use the techniques of $\mathcal{K}$-convergence; see
\cite{FLM+21a,FLM+21b}. Denoting by $U_k = [h_k, m_{x,k}, m_{y,k}]$ the numerical solution on a $k\times k$ grid, we compute the following errors:  
\[
  E_1 = \|U_k-U_{ref}\|_{L^1},\, E_2 = \|\overline{U}_k-\overline{U}_{ref}\|_{L^1},\, E_3 = \|\Tilde{U}_k-\Tilde{U}_{ref}\|_{L^1},\, E_4 = \|W_1(\overline{\mathcal{V}}^k_{t,x,y},\overline{\mathcal{V}}^{ref}_{t,x,y})\|_{L^1},
\]
where $\overline{U}_k = \frac{1}{k}\sum_{j=1}^k U_j$, $\Tilde{U}_k =
\frac{1}{k}\sum_{j=1}^k\abs{U_j-\overline{U}_k}$,
$\overline{\mathcal{V}}^k_{t,x,y} =
\frac{1}{k}\sum_{j=1}^k\delta_{U_j(t,x,y)}$, where
$\delta_{U_j(t,x,y)}$ denotes the Dirac measure centered at
$U_j(t,x,y)$. Also, $W_1$ here denotes the 1-Wasserstein distance on the space of probability measures and we compute it using the \texttt{wasserstein\_distance} function available in the python library SciPy.
The reference solution $U_{ref}$ is computed on a $1024 \times 1024$
grid using the semi-discrete WB scheme of Audusse et al.\
\cite{ADD+21}. The error $E_1$ is the standard $L^1$ error between the numerical and reference solutions. The errors $E_2$ and $E_3$ are respectively the errors in the Cesaro averages and the first variance, while $E_4$ is the $L^1$ norm of the Wasserstein distance between the refrence and numerical solutions. The main idea behind $\K$-convergence is that averaging out sequences `compactifies' them, allowing us to obtain strong convergence resutls. The errors $E_2$ and $E_3$ exactly encapsulate this idea, while the error $E_4$ reflects the probabilistic nature of the solution.  Since the convergence to a DMV solution is only in the weak sense, we do not expect the error $E_1$ to decrease as we refine the mesh. However, according to the theory of $\K$-convergence, the errors $E_2$, $E_3$ and $E_4$ should decrease to zero as we refine the mesh, which will imply the convergence of the numerical solutions towards a DMV solution of the RSW system. The corresponding error profiles are shown in Figure
\ref{fig:dvt-err-dmv}. The disordered behavior of the error $E_1$,
together with the consistent reduction of the errors $E_2, E_3$ and  $E_4$
under mesh refinement, indicates that the numerical solutions indeed converge
toward a DMV solution of the RSW system. 


\begin{figure}[htpb]
  \centering
  \includegraphics[height = 0.2\textheight]{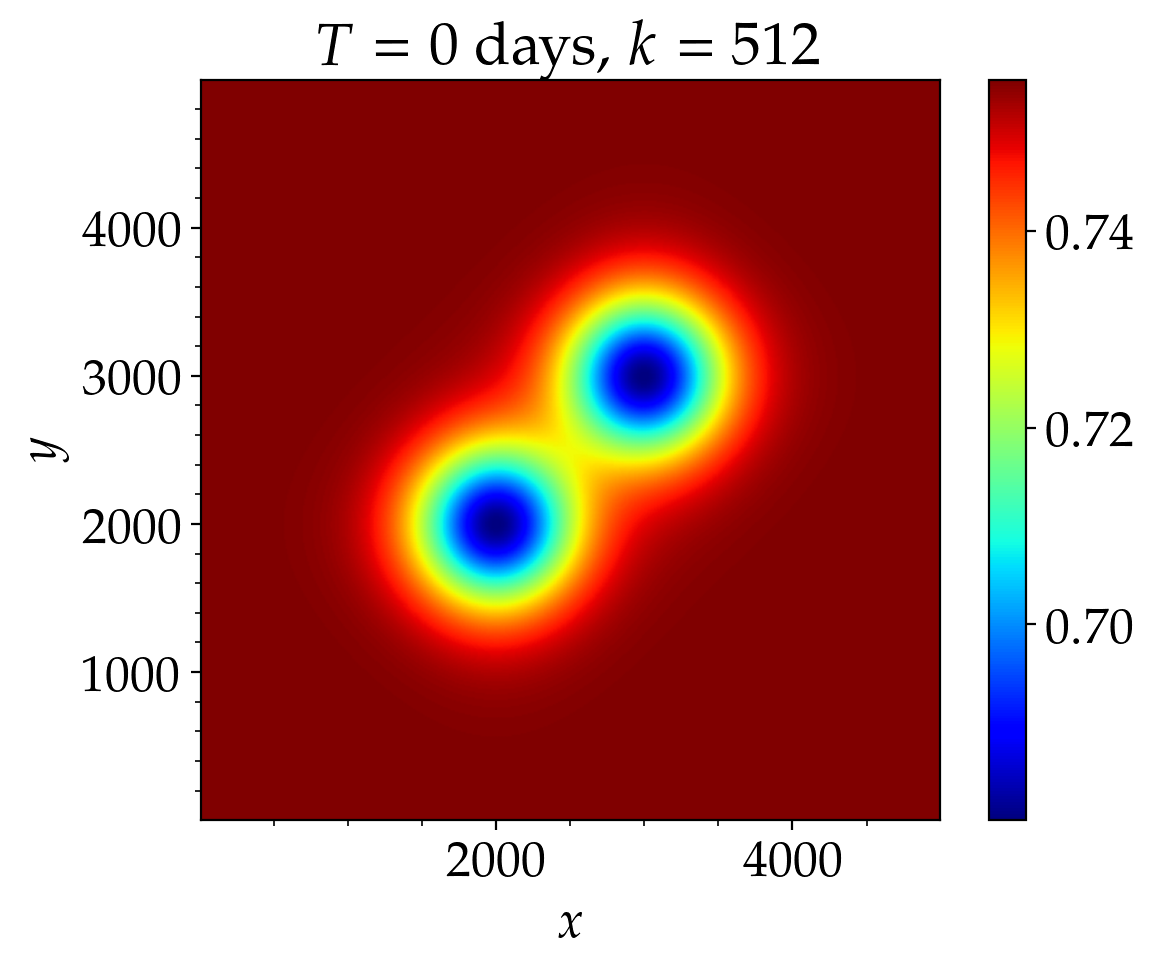}
  \includegraphics[height = 0.2\textheight]{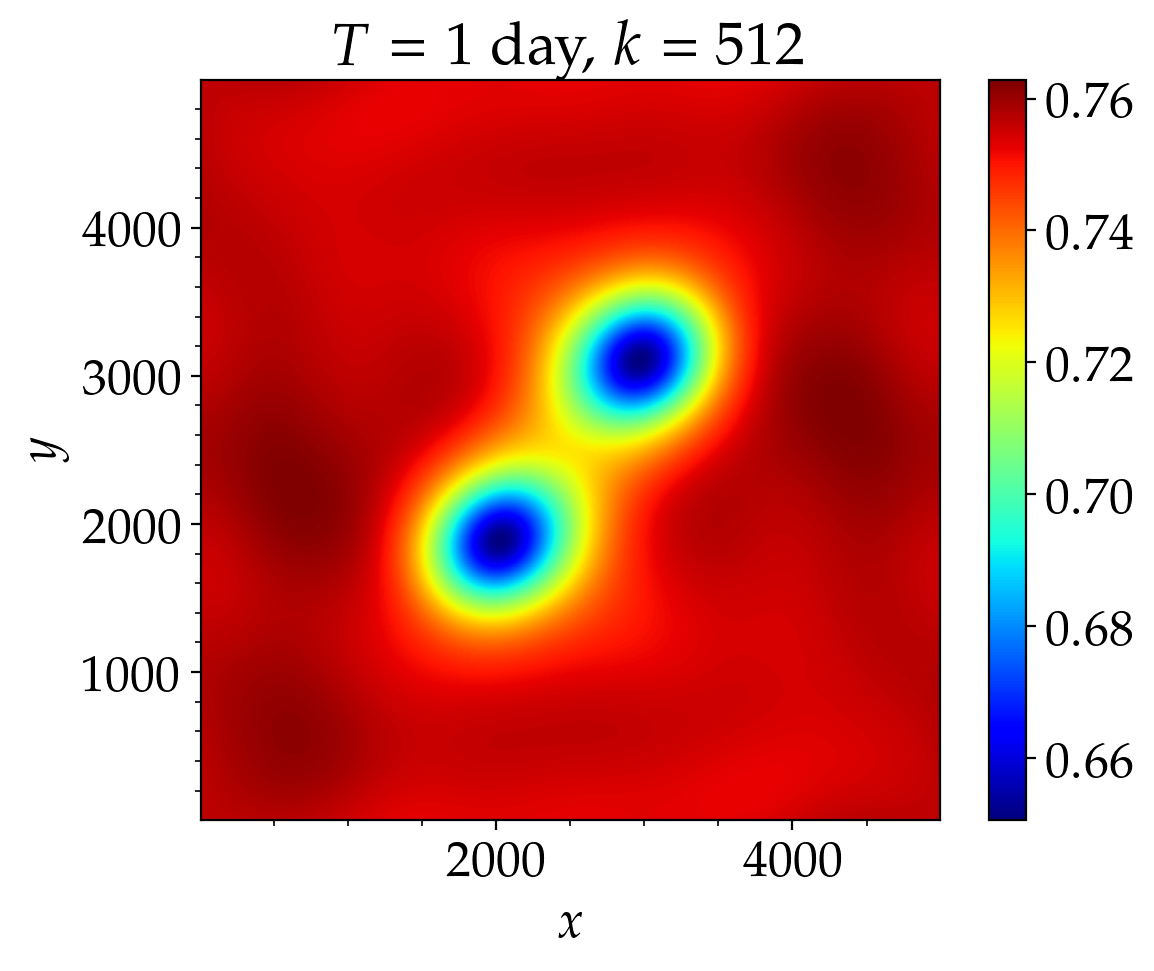}
  \includegraphics[height = 0.2\textheight]{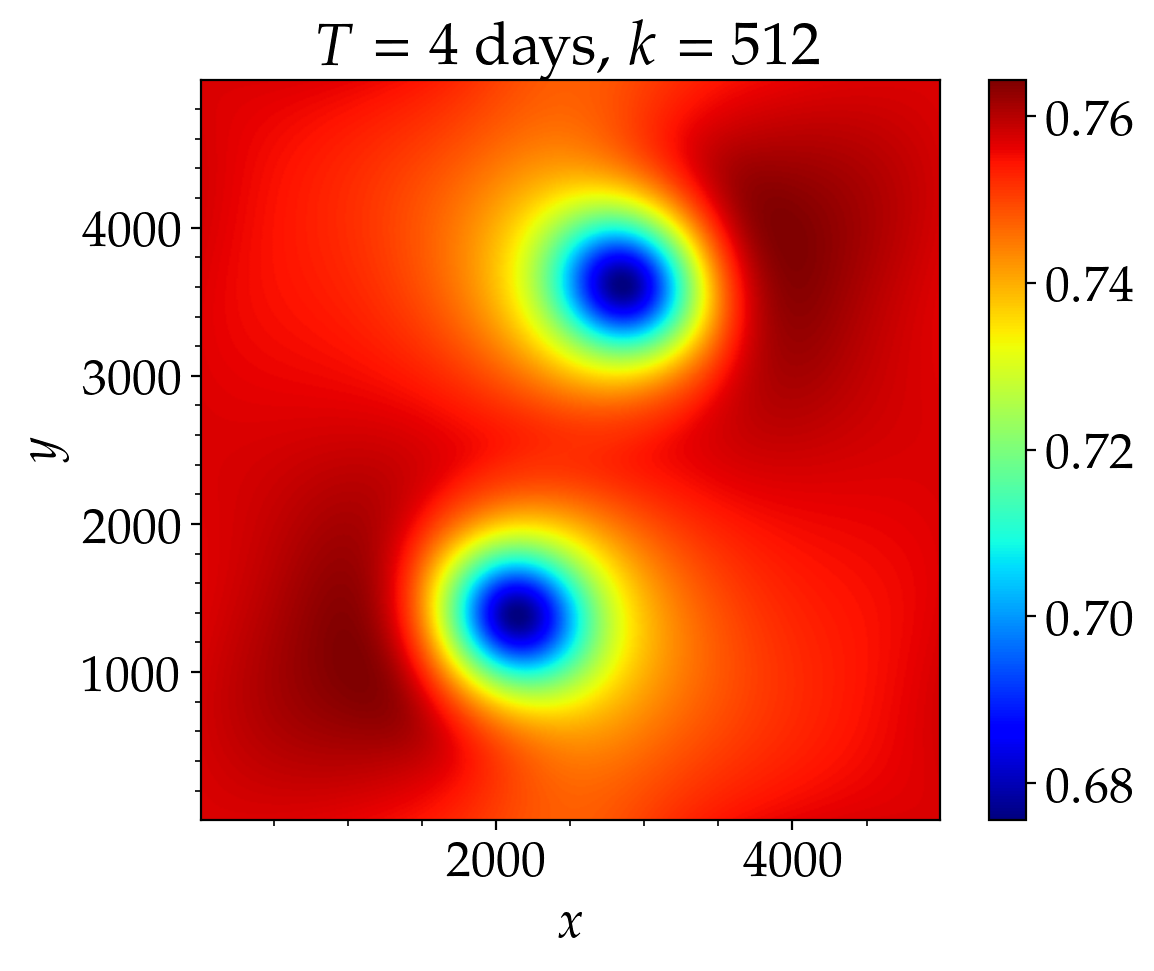}
  \includegraphics[height = 0.2\textheight]{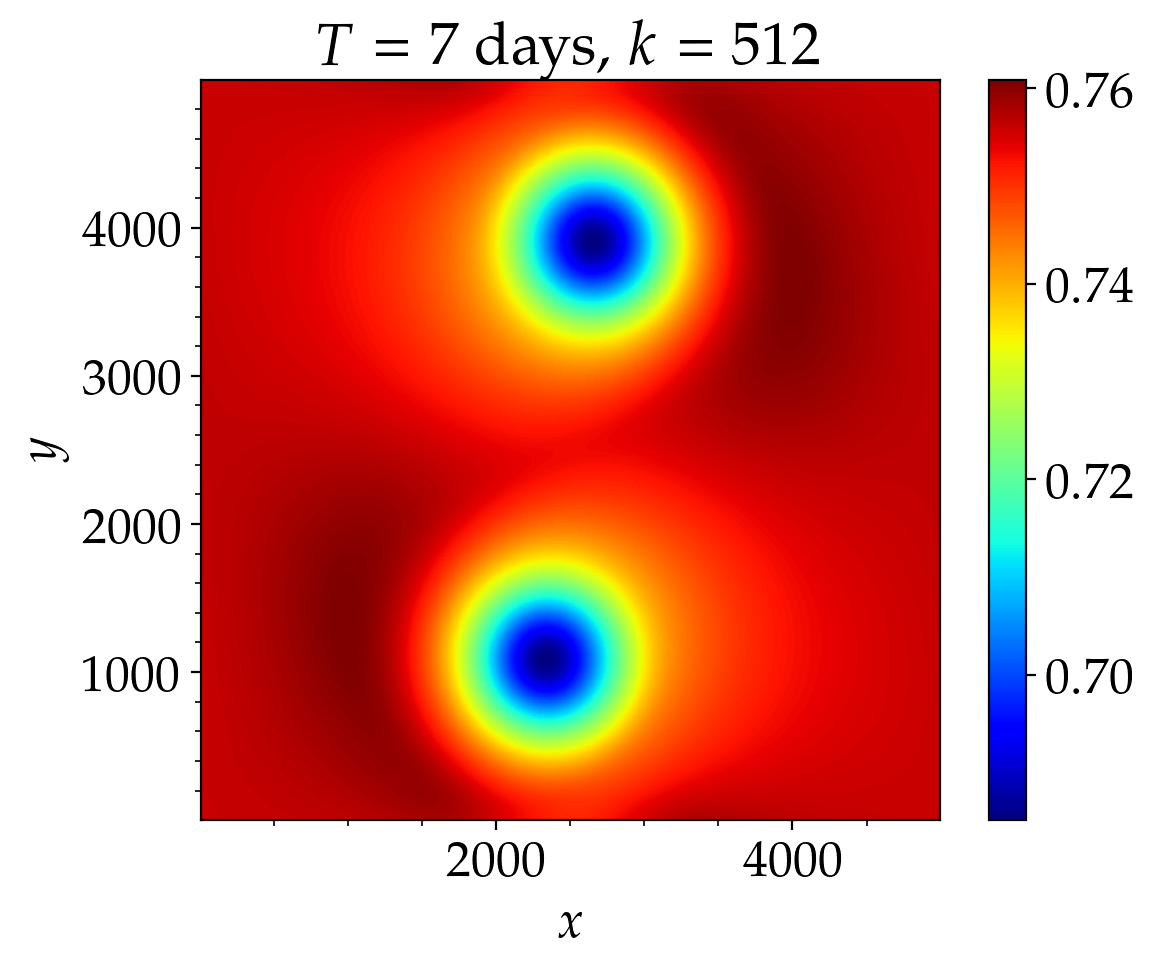}
  \includegraphics[height = 0.2\textheight]{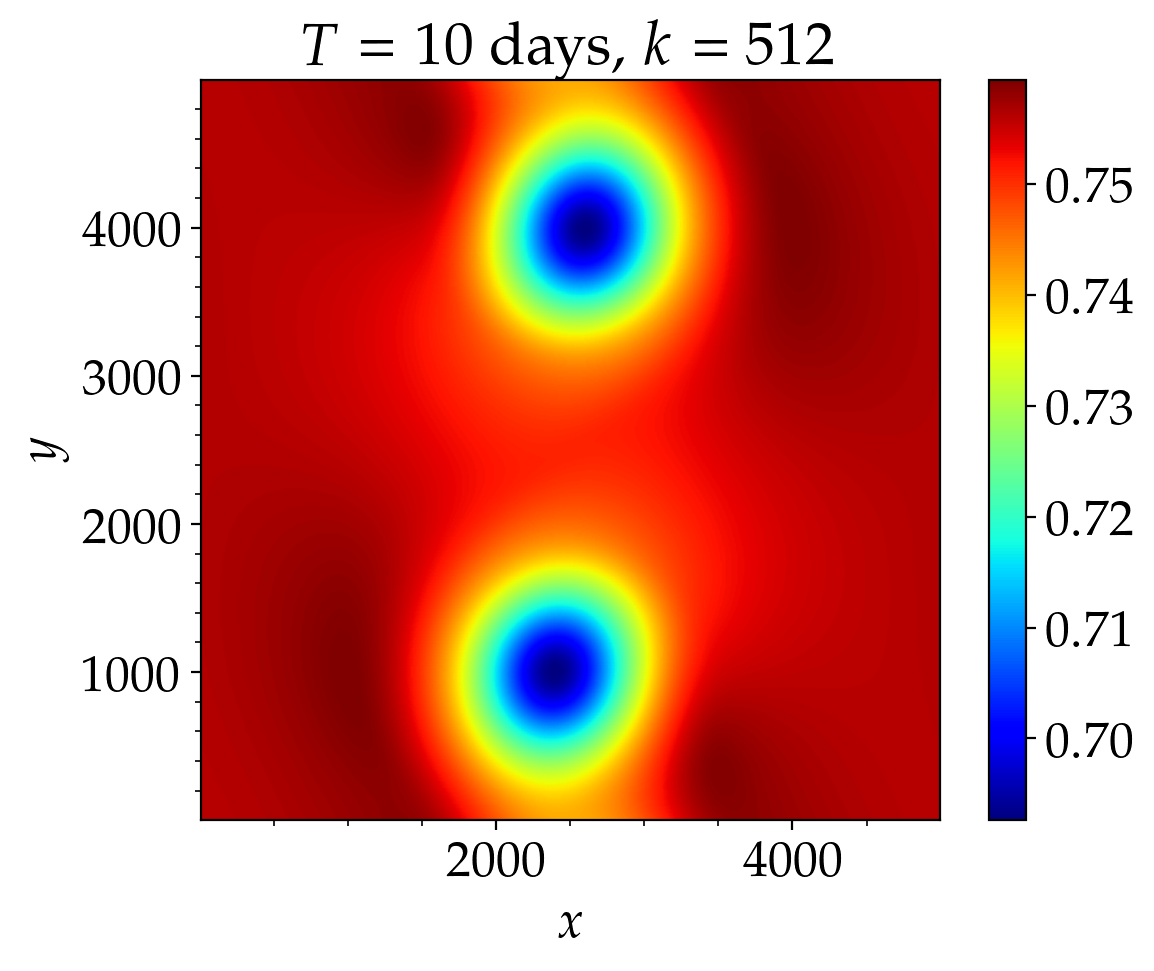}
  \caption{The water height $h$ (in km) over time for the vortex pair interaction problem}
  \label{fig:dbl-vortex-h}
\end{figure}

\begin{figure}[htpb]
  \centering
  \includegraphics[height = 0.2\textheight]{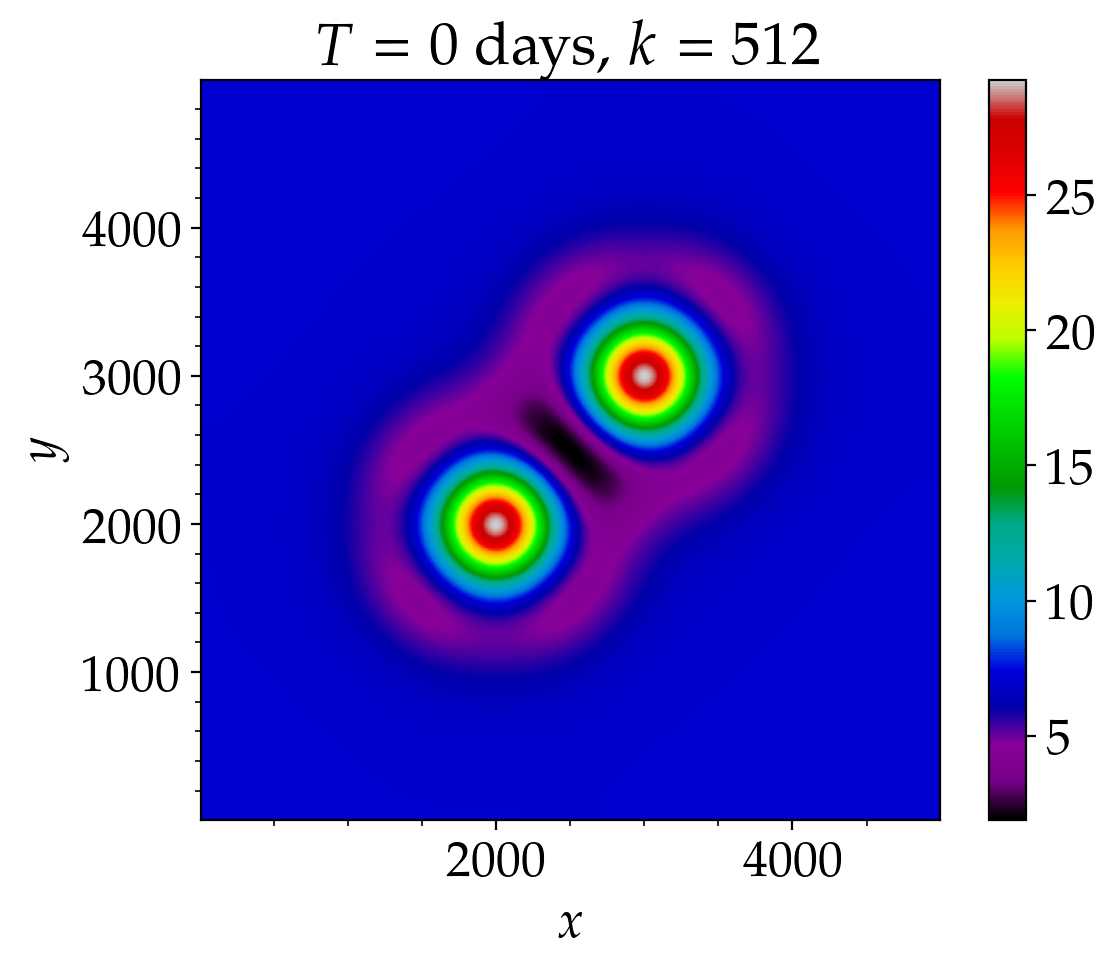}
  \includegraphics[height = 0.2\textheight]{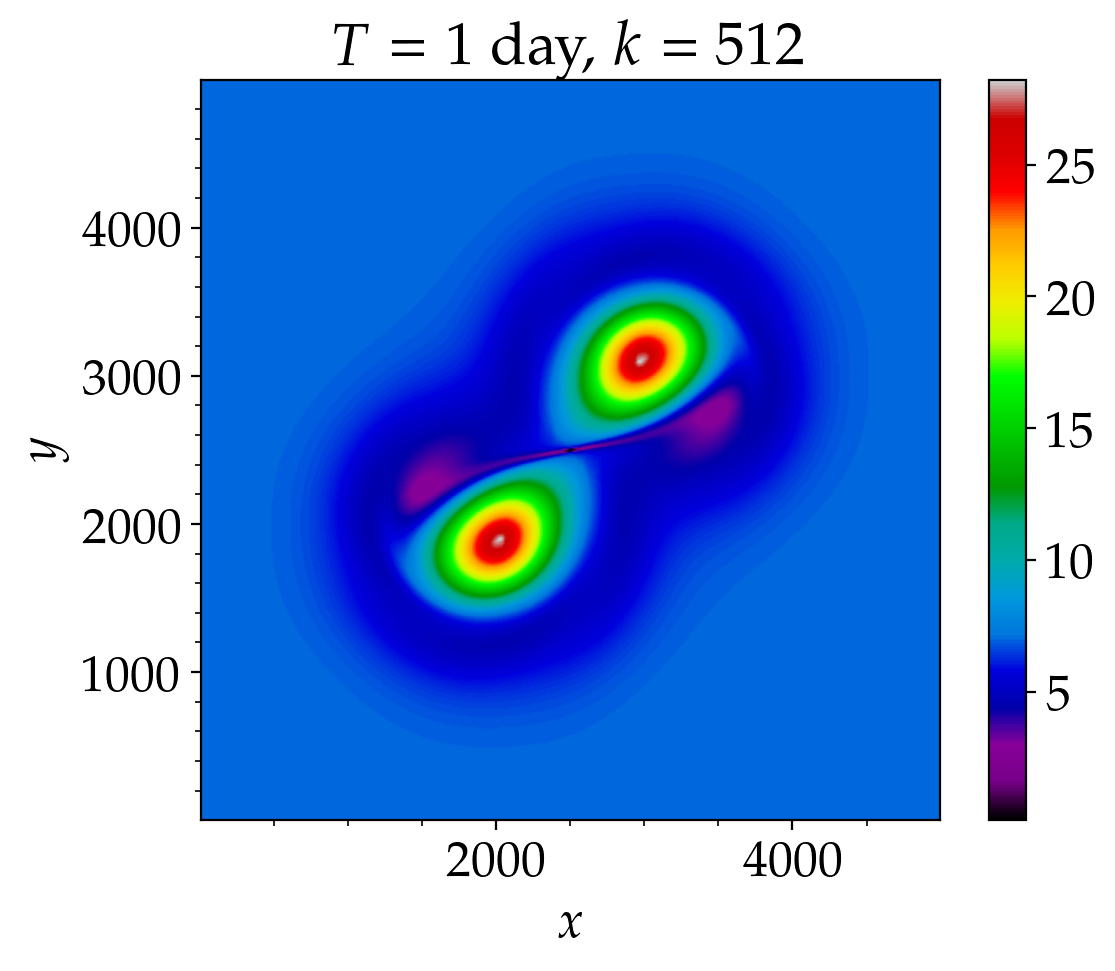}
  \includegraphics[height = 0.2\textheight]{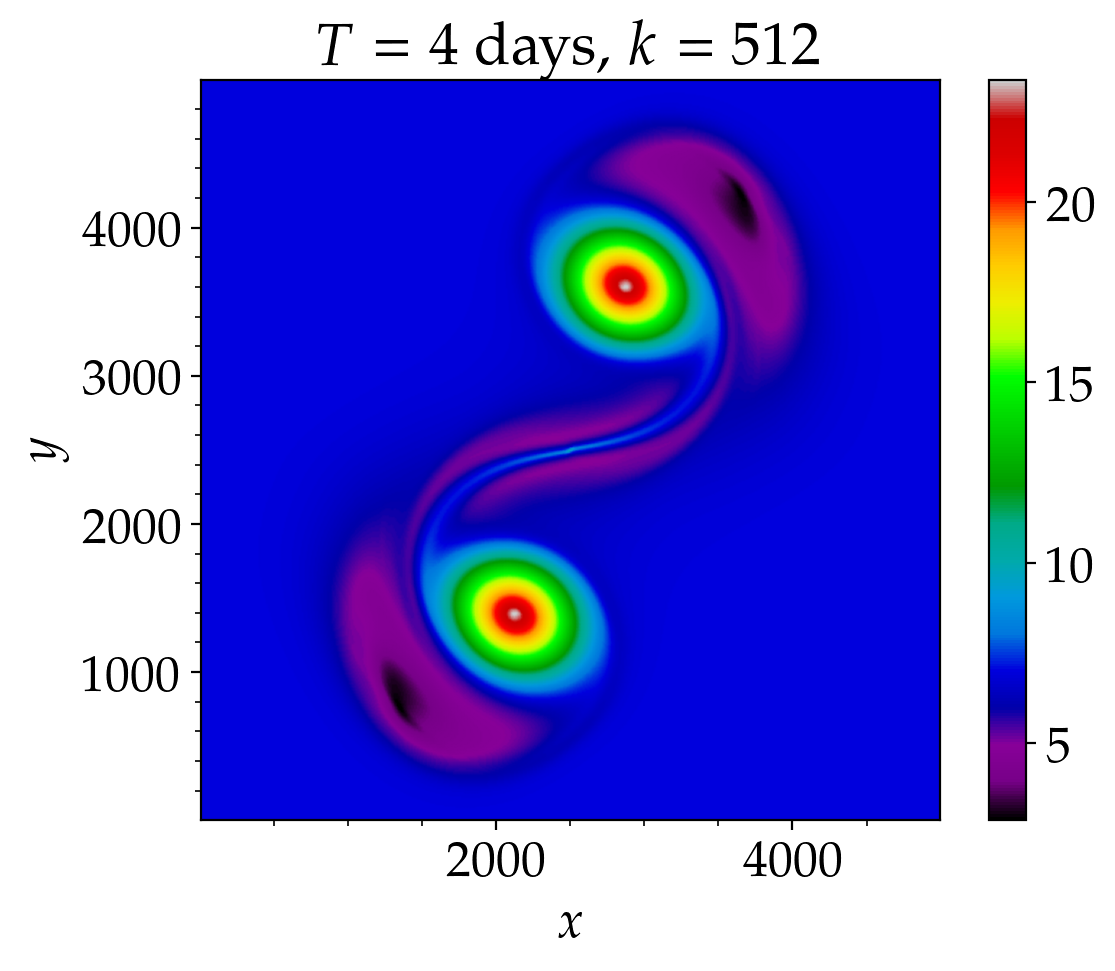}
  \includegraphics[height = 0.2\textheight]{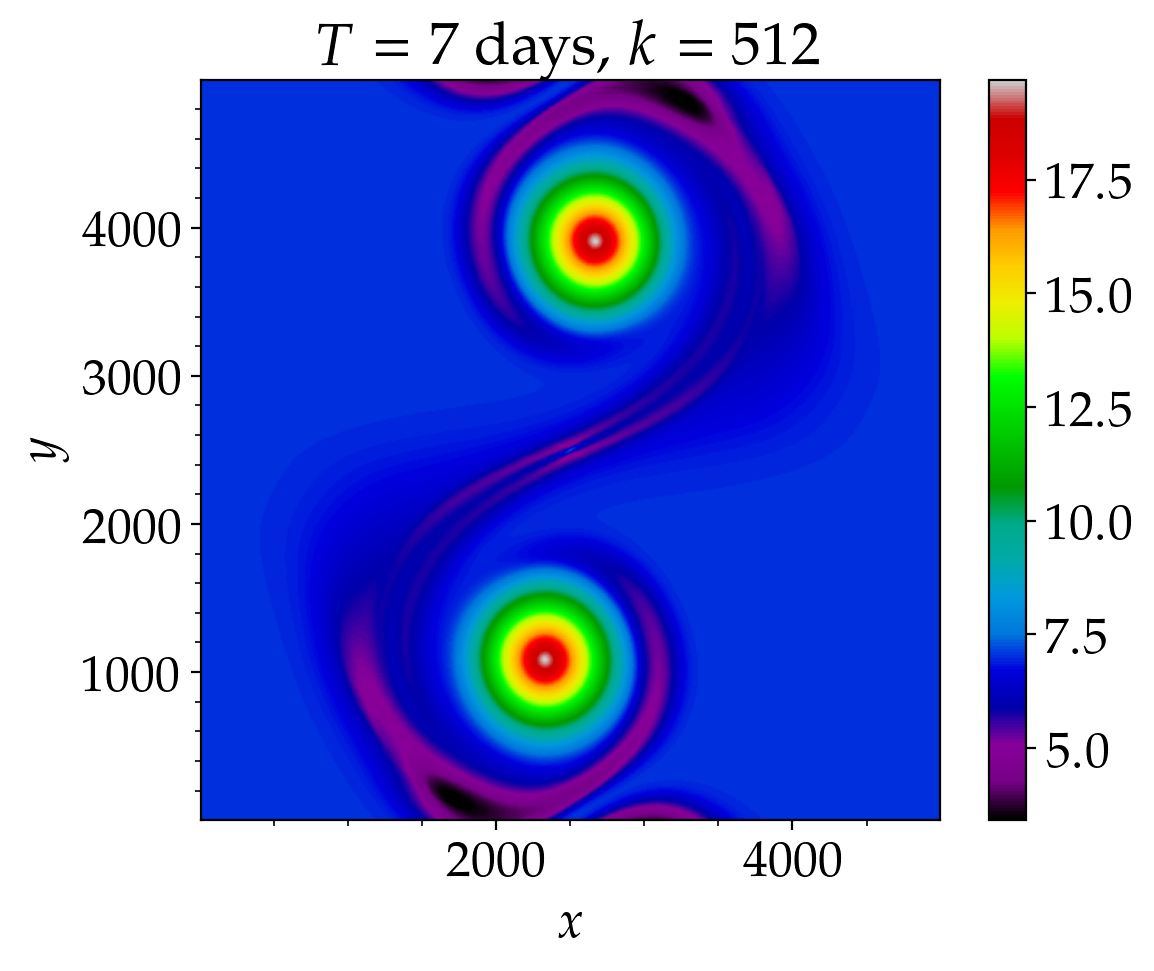}
  \includegraphics[height = 0.2\textheight]{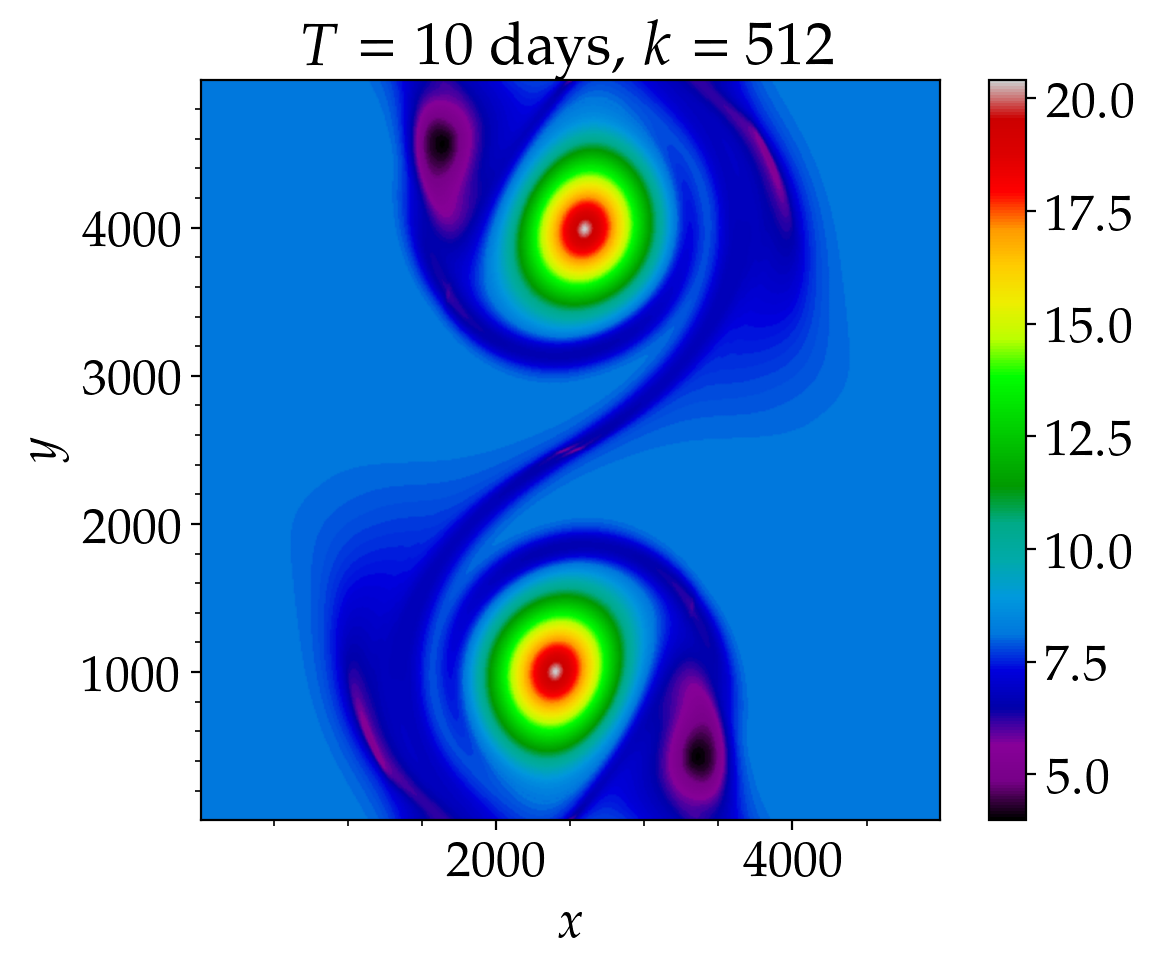}
  \caption{Potential voritcity $PV$ (in days$^{-1}$km$^{-1}$) over time for the vortex pair interaction problem}
  \label{fig:dbl-vortex-pv}
\end{figure}

\begin{figure}[htpb]
    \centering
    \includegraphics[height = 0.2\textheight]{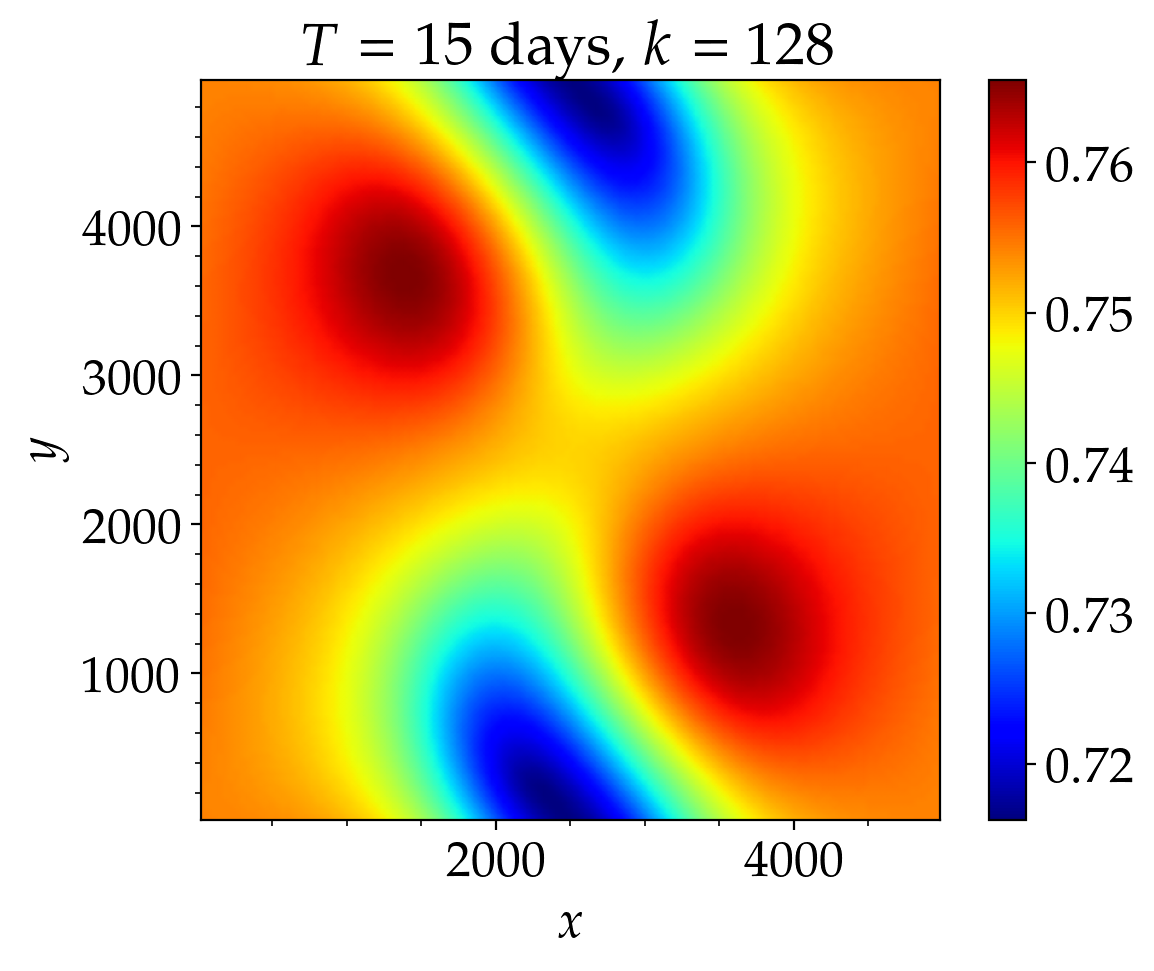}
    \includegraphics[height = 0.2\textheight]{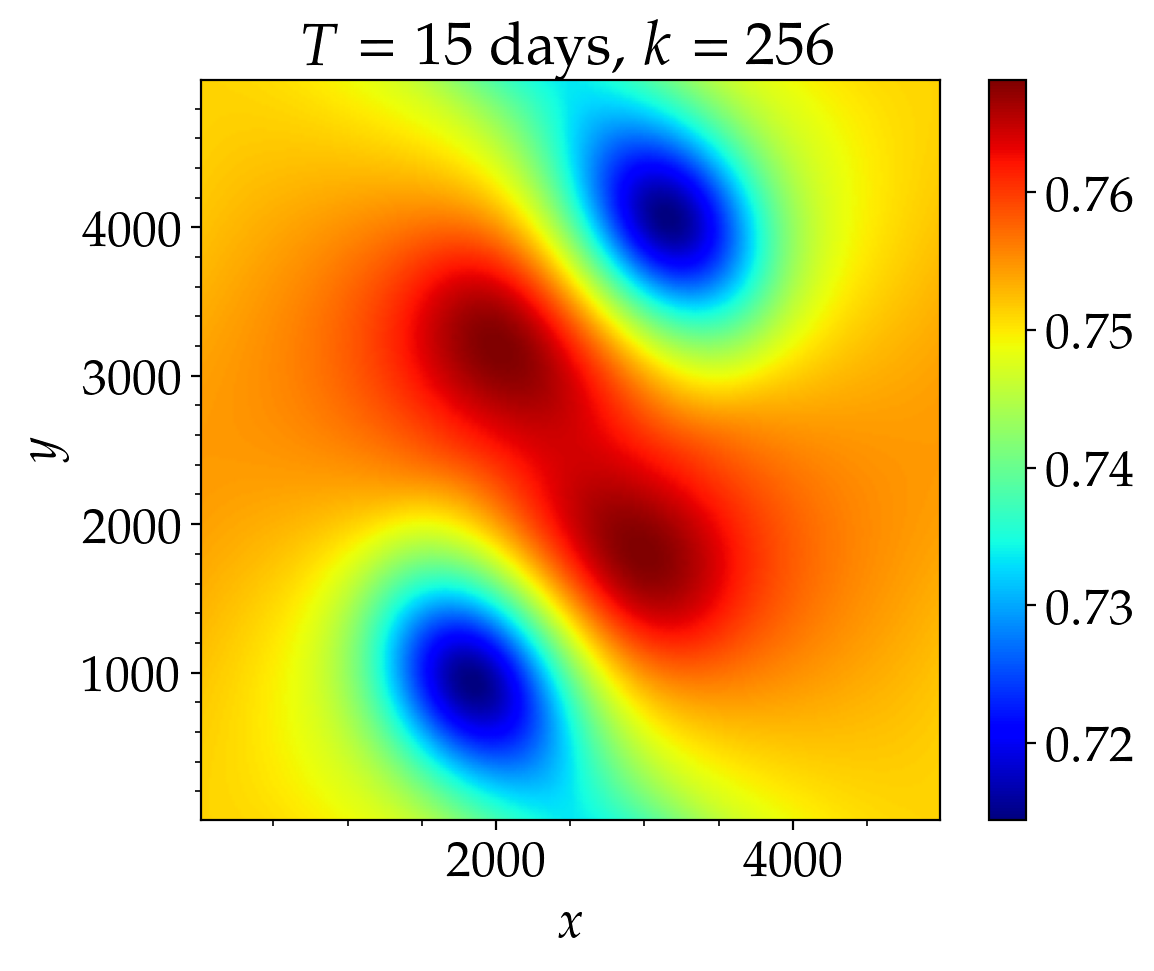}
    \includegraphics[height = 0.2\textheight]{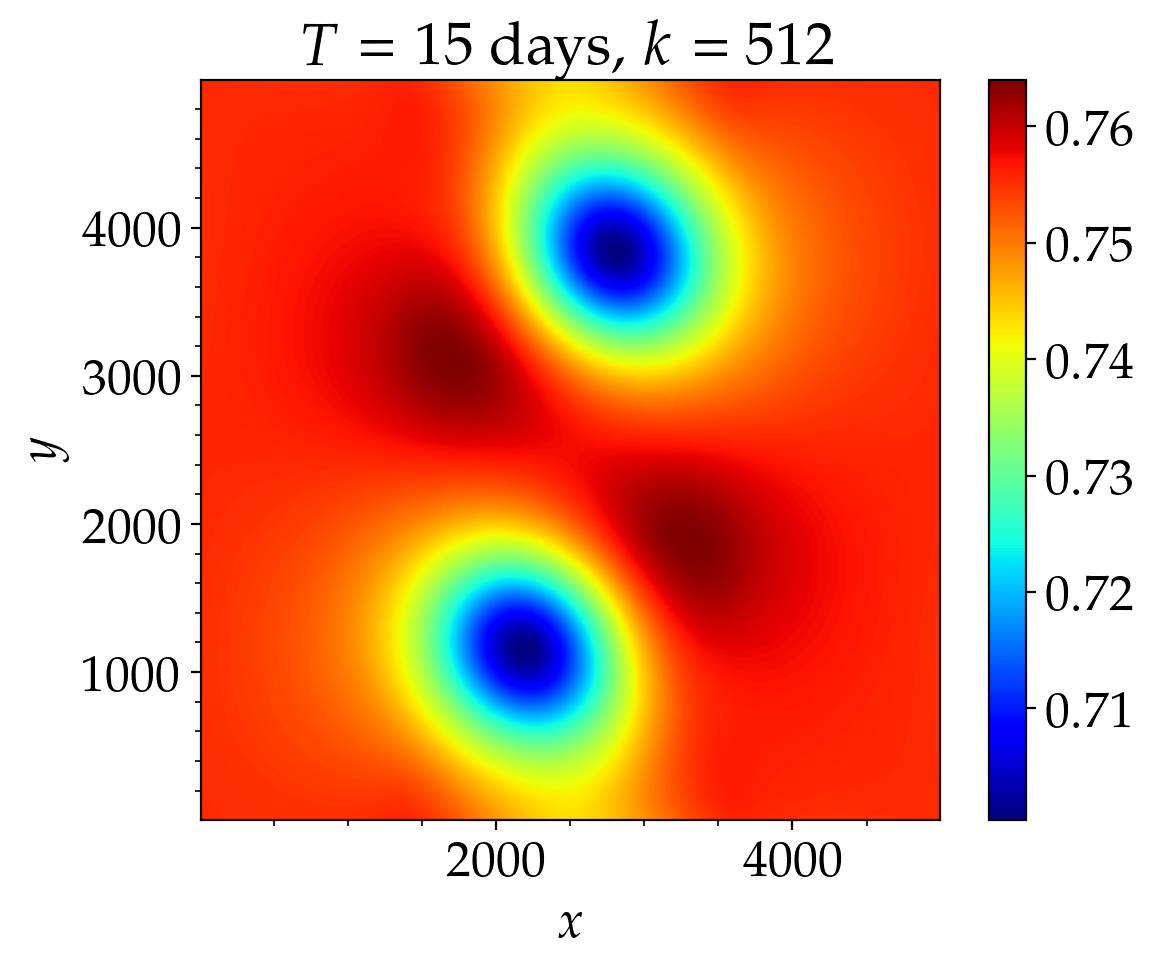}
    \includegraphics[height = 0.2\textheight]{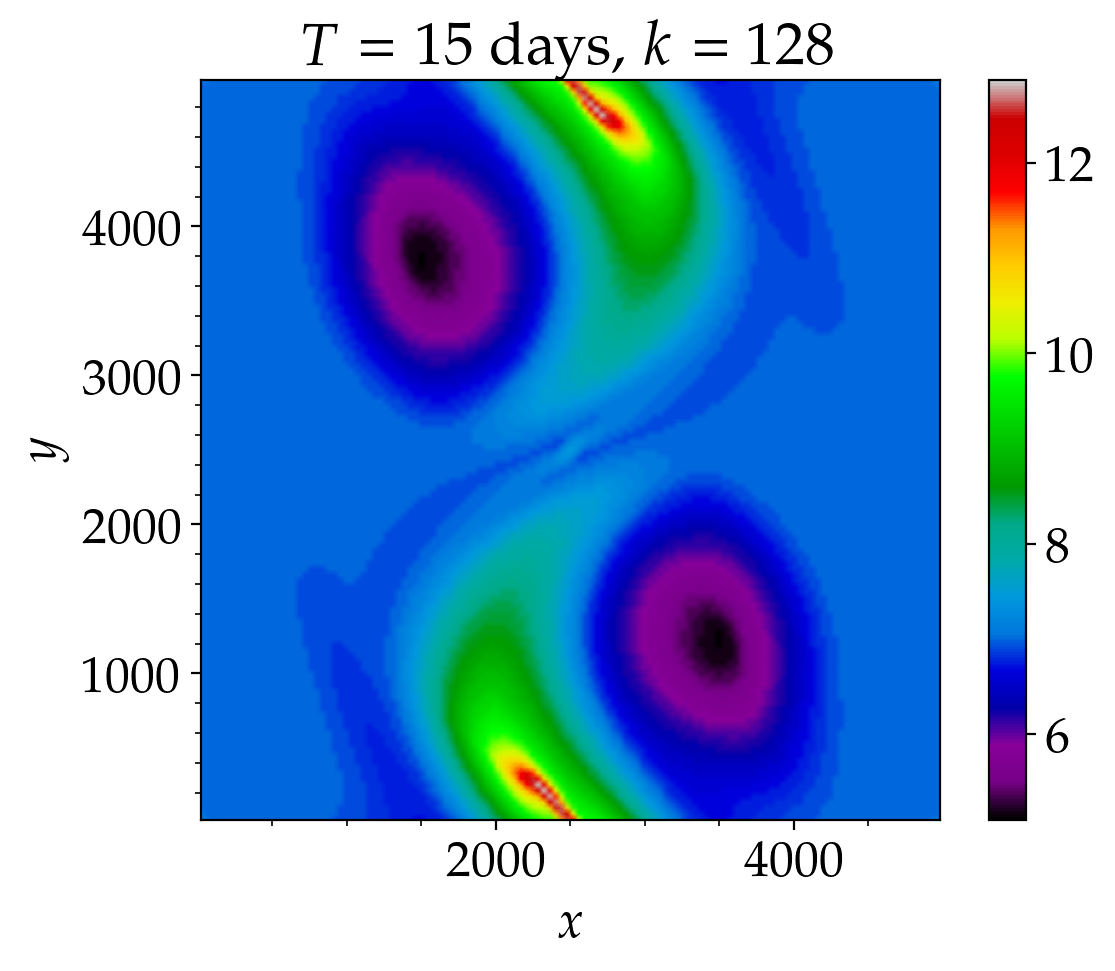}
    \includegraphics[height = 0.2\textheight]{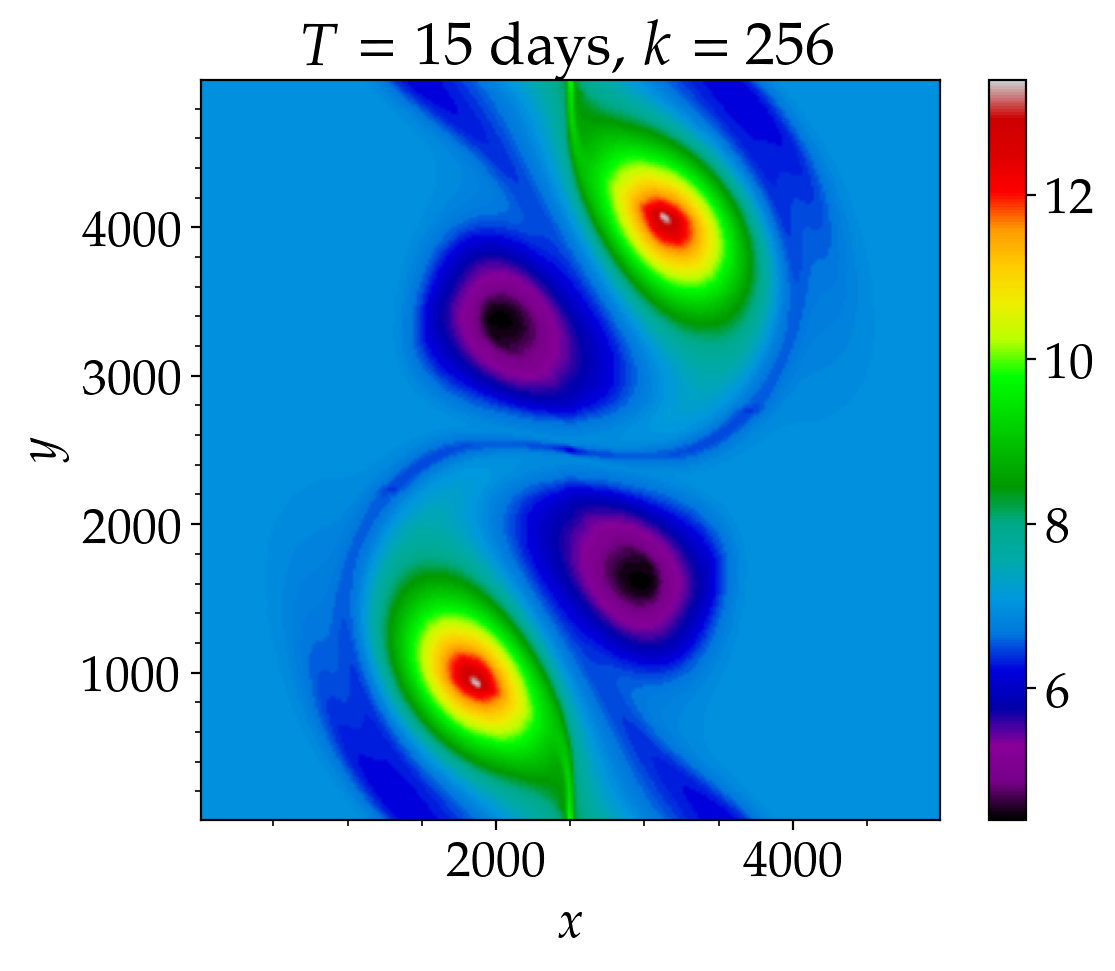}
    \includegraphics[height = 0.2\textheight]{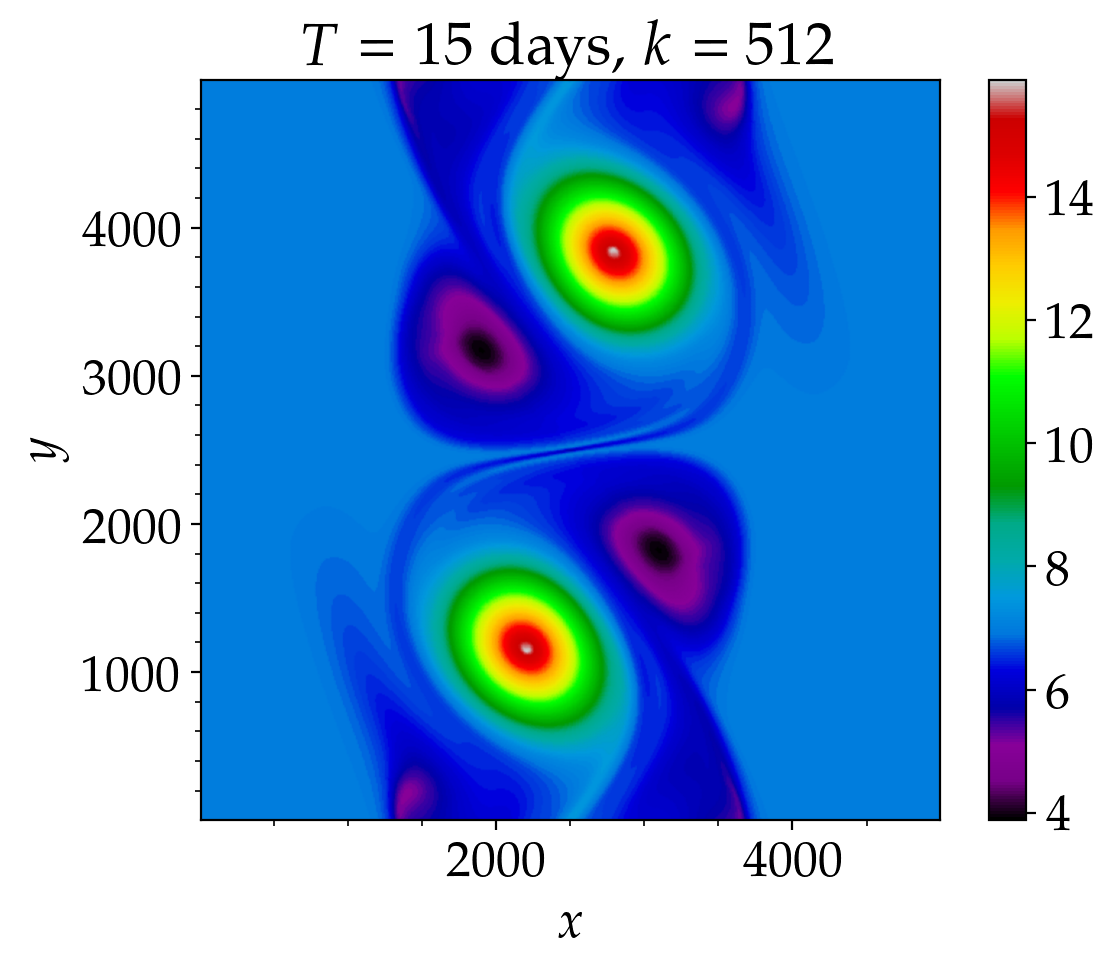}
    \caption{The water height $h$ (in km, top) and the potential
      vorticity $PV$ (in days$^{-1}$km$^{-1}$, bottom) on successive
      meshes at final time $T = 15$ days for the vortex pair
      interaction problem.} 
    \label{fig:pcol-dbl-vortex}
\end{figure}

\begin{figure}
  \centering
  \includegraphics[height = 0.4\textheight]{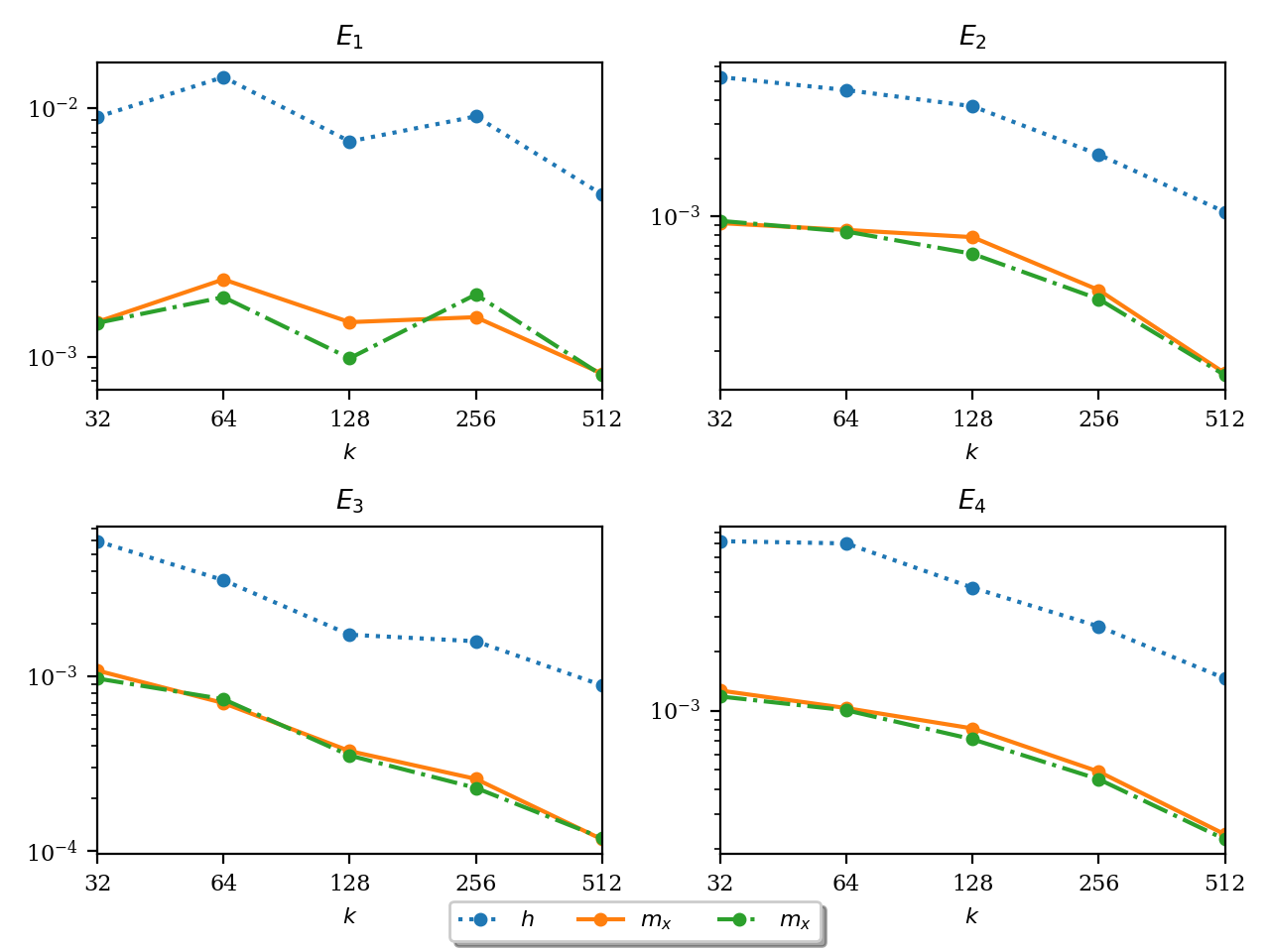}
  \caption{The error profiles for the vortex pair interaction test case.}
  \label{fig:dvt-err-dmv}
\end{figure}

\section{Conclusions}
\label{sec:conc}

We have developed and analyzed a semi-implicit in time finite volume
scheme for the 2D RSW system. The solutions generated by the proposed
scheme enjoy several key properties, including positivity of the water
height, discrete energy stability, well-balancing with respect to
geostrophic steady states, and consistency with the continuous
system. These properties are established using the technique of
stabilization, wherein the convective fluxes of the mass and momentum
equations, as well as the source terms, are modified through the
introduction of carefully chosen perturbation terms. Furthermore, we
have shown that, as the mesh parameters vanish, the numerical
solutions converge weakly to a DMV solution of the 2D RSW system under
suitable boundedness assumptions. Finally, the theoretical results
have been verified through extensive numerical case studies, which
confirm the efficacy and robustness of the proposed scheme.  

\bibliographystyle{abbrv}
\bibliography{ref}
\end{document}